\documentclass[11pt]{amsart}
\usepackage{amsmath,amssymb,amsthm,enumitem,tikz-cd,bbm,mathtools}
\usepackage[hmargin=20mm,vmargin=30mm]{geometry}

\numberwithin{equation}{section}

\setlist[itemize]{labelindent=\parindent,leftmargin=*}
\setlist[enumerate,1]{leftmargin=*,label=(\roman*),ref=\roman*}
\setlist[enumerate,2]{labelindent=\parindent,leftmargin=*,label=(\arabic*),ref=\arabic*}

\theoremstyle{plain}
\newtheorem{thm}{Theorem}[section]
\newtheorem{lem}[thm]{Lemma}
\newtheorem{prop}[thm]{Proposition}
\newtheorem{cor}[thm]{Corollary}
\theoremstyle{remark}
\newtheorem{rem}[thm]{Remark}

\newcommand\diag{\operatorname{diag}}

\newcommand\Hom{\operatorname{Hom}}
\newcommand\Ind{\operatorname{Ind}}
\newcommand\Irr{\operatorname{Irr}}

\newcommand\Res{\operatorname{Res}}
\newcommand\Tr{\operatorname{Tr}}

\newcommand\down{\mathrm{down}}
\newcommand\GL{\mathrm{GL}}
\newcommand\Mp{\mathrm{Mp}}

\newcommand\SL{\mathrm{SL}}
\newcommand\Sp{\mathrm{Sp}}

\newcommand\U{\mathrm{U}}
\newcommand\up{\mathrm{up}}

\newcommand\A{\mathbb{A}}
\newcommand\C{\mathbb{C}}
\newcommand\E{\mathbb{E}}
\newcommand\F{\mathbb{F}}
\newcommand\Q{\mathbb{Q}}
\newcommand\R{\mathbb{R}}
\newcommand\V{\mathbb{V}}
\newcommand\W{\mathbb{W}}
\newcommand\X{\mathbb{X}}
\newcommand\Y{\mathbb{Y}}
\newcommand\Z{\mathbb{Z}}

\renewcommand\b{\mathfrak{b}}
\newcommand\g{\mathfrak{g}}
\renewcommand\k{\mathfrak{k}}
\renewcommand\l{\mathfrak{l}}

\newcommand\q{\mathfrak{q}}
\renewcommand\t{\mathfrak{t}}
\renewcommand\u{\mathfrak{u}}

\renewcommand\AA{\mathcal{A}}
\newcommand\CC{\mathcal{C}}

\newcommand\LL{\mathcal{L}}
\renewcommand\SS{\mathcal{S}}

\newcommand\VV{\mathcal{V}}
\newcommand\XX{\mathcal{X}}
\newcommand\ZZ{\mathcal{Z}}

\newcommand\1{\mathbf{1}}

\newcommand{\hotimes}{\mathbin{\hat{\otimes}}}
\newcommand{\nA}{{\mathstrut}^n \! A}

\title{Theta lifting for tempered representations of real unitary groups}
\author{Atsushi Ichino}
\address{Department of Mathematics, Kyoto University, Kitashirakawa Oiwake-cho, Sakyo-ku, Kyoto 606-8502, Japan}
\email{ichino@math.kyoto-u.ac.jp}

\begin{document}

\maketitle

\begin{abstract}
We study the theta lifting for real unitary groups and completely determine the theta lifts of tempered representations.
In particular, we show that the theta lifts of (limits of) discrete series representations can be expressed as cohomologically induced representations in the weakly fair range.
This extends a result of J.-S.~Li in the case of discrete series representations with sufficiently regular infinitesimal character, whose theta lifts can be expressed as cohomologically induced representations in the good range.
\end{abstract}

\section{Introduction}

In his seminal papers \cite{howe1, howe2}, Howe introduced the notion of reductive dual pairs and developed the theory of theta lifting, which has been an important subject in the representation theory of real and $p$-adic reductive groups for more than $40$ years and which has many arithmetic applications to the theory of automorphic forms.
The theta lifting is defined as a correspondence between representations of the two groups in a reductive dual pair in terms of the restriction of the Weil representation \cite{weil1}.
In fact, it is shown that this correspondence is one-to-one by Howe himself \cite{howe2} in the real case and by Gan--Takeda \cite{gt} in the $p$-adic case, following earlier work of Howe \cite{howe1} and Waldspurger \cite{wal} for $p \ne 2$.
For the history and recent development of the theta lifting, the reader can consult the ICM report of Gan \cite{gan}.

In the theory of theta lifting, one of the basic problems is to describe it explicitly.
We consider this problem in the real case, which has been studied by M{\oe}glin \cite{moeglin89}, Li \cite{li90}, Adams--Barbasch \cite{ab1, ab2}, Paul \cite{paul1, paul2, paul3}, Li--Paul--Tan--Zhu \cite{lptz} to mention a few, but which has not been solved in general.
In this paper, we focus on the case of the reductive dual pair $(\U(p,q),\U(r,s))$ consisting of real unitary groups.
Recall the Weil representation $\omega$ of $\Mp_{2l}(\R)$ (relative to a fixed nontrivial character of $\R$), where $l = (p+q)(r+s)$ and $\Mp_{2l}(\R)$ is the metaplectic cover of the symplectic group $\Sp_{2l}(\R)$ of rank $l$.
Via the choice of a lift
\begin{align*}
 \U(p,q) \times \U(r,s) & \rightarrow \Mp_{2l}(\R) \\
 \shortintertext{of a natural homomorphism}
 \U(p,q) \times \U(r,s) & \rightarrow \Sp_{2l}(\R), 
\end{align*}
we may regard $\omega$ as a representation of $\U(p,q) \times \U(r,s)$.
Then for any irreducible representation $\pi$ of $\U(p,q)$, its theta lift to $\U(r,s)$ is defined as an irreducible representation $\theta_{r,s}(\pi)$ of $\U(r,s)$ such that 
\[
 \Hom_{\U(p,q) \times \U(r,s)}(\omega, \pi \boxtimes \theta_{r,s}(\pi)) \ne 0,
\]
which is uniquely determined (if exists) by the Howe duality \cite{howe2}.
If such a representation does not exist, we interpret $\theta_{r,s}(\pi)$ as zero.

When $p+q \le r+s$ and $\pi$ is a discrete series representation with sufficiently regular infinitesimal character, Li \cite{li90} showed that $\theta_{r,s}(\pi)$ is nonzero and expressed it as a cohomologically induced representation in the good range.
When $p+q=r+s$ or $p+q=r+s \pm 1$, Paul \cite{paul1, paul2} generalized his result and completely determined $\theta_{r,s}(\pi)$ for arbitrary $\pi$.
The purpose of this paper is to describe $\theta_{r,s}(\pi)$ explicitly when $\pi$ is tempered but for arbitrary $p,q,r,s$.
For this, we first prove the following generalization of \cite{li90} to the case of (limits of) discrete series representations (see Theorem \ref{t:main-lds} for more details).

\begin{thm}
\label{t:intro}
Let $\pi$ be a (limit of) discrete series representation of $\U(p,q)$.
Assume that its theta lift $\theta_{r,s}(\pi)$ to $\U(r,s)$ is nonzero.
Then we have
\[
 \theta_{r,s}(\pi) = A_\q(\lambda),
\]
where the right-hand side is a cohomologically induced representation in the weakly fair range, and $\q$ and $\lambda$ can be described explicitly.
Moreover, if $p+q \ge r+s-1$, then $\theta_{r,s}(\pi)$ is a (limit of) discrete series representation.
\end{thm}

We have stated the result under the assumption that $\theta_{r,s}(\pi)$ is nonzero, but there is a combinatorial criterion for the nonvanishing of $\theta_{r,s}(\pi)$ due to Atobe \cite{atobe} (see also \S \ref{s:atobe}).
Based on this theorem, we can describe $\theta_{r,s}(\pi)$ explicitly for any tempered representation $\pi$ (see Theorem \ref{t:main-temp} for more details).

We now give some details of the proof of Theorem \ref{t:intro}.
Our proof is global and relies on Arthur's endoscopic classification \cite{arthur,mok,kmsw}.
Thus our main result is conditional on Arthur's multiplicity formula for the automorphic discrete spectra of unitary groups announced by Kaletha--M\'inguez--Shin--White \cite{kmsw} (see \eqref{eq:amf} for details), whose proof will be completed in their subsequent work.
We first globalize the given local theta lift for real unitary groups.
Namely, we find a global theta lift such that
\begin{itemize}
\item at one real place, its localization is the theta lift of an arbitrary (limit of) discrete series representation;
\item at another real place, its localization is the theta lift of a discrete series representation with sufficiently regular infinitesimal character, which is determined explicitly by Li \cite{li90};
\item at the other places, its localizations are easy to describe explicitly.
\end{itemize}
Then we use Arthur's multiplicity formula (viewed as a product formula) to transfer the information from the case of sufficiently regular infinitesimal character to the general case.
However, there is a difficulty in this argument: it is not straightforward to globalize a local theta lift for real unitary groups.

More precisely, let $\pi$ be a (limit of) discrete series representation of $\U(p,q)$ and consider its theta lift $\theta_{r,s}(\pi)$ to $\U(r,s)$.
Switching the roles of $\U(p,q)$ and $\U(r,s)$ if necessary, we may assume that $p+q < r+s$.
Let $F \ne \Q$ be a totally real number field with ad\`ele ring of $\A$ and fix a real place $v_0$ of $F$.
Then it is easy to find
\begin{itemize}
\item anisotropic unitary groups $G$ and $H$ over $F$ such that $G_{v_0} = \U(p,q)$ and $H_{v_0} = \U(r,s)$, respectively;
\item an irreducible automorphic representation of $G(\A)$ such that $\varPi_{v_0} = \pi$.
\end{itemize}
But we need $G,H,\varPi$ such that the global theta lift $\theta(\varPi)$ to $H(\A)$ is nonzero.
For this, we proceed as follows.
\begin{enumerate}
\item 
\label{item:intro1}
Find $G,H,\varPi$ such that the local theta lift $\theta(\varPi_v)$ to $H_v$ is nonzero for all places $v$ of $F$.
\item 
\label{item:intro2}
Show that $\theta(\varPi)$ is nonzero if and only if $\theta(\varPi_v)$ is nonzero for all $v$.
\end{enumerate}
To show that $G,H,\varPi$ as in \eqref{item:intro1} exist, we appeal to Arthur's multiplicity formula.
In fact, we may impose further local conditions on $G,H,\varPi$ to make the global-to-local argument work.
On the other hand, \eqref{item:intro2} is largely but not completely known for unitary groups (see \cite[Theorem 1.3]{gqt}).
Indeed, the standard argument relies on the Rallis inner product formula, which involves the local integral at $v_0$ given by
\[
 \int_{\U(p,q)} (\omega(g) \varphi_1, \varphi_2) \overline{(\pi(g)v_1, v_2)} \, dg
\]
for $\varphi_1, \varphi_2 \in \omega$ and $v_1, v_2 \in \pi$.
Here $(\cdot, \cdot)$ denotes an invariant Hermitian inner product.
This integral is absolutely convergent and defines an invariant functional
\[
 \ZZ_{r,s}(\pi) : \omega \otimes \bar{\omega} \otimes \bar{\pi} \otimes \pi \rightarrow \C.
\]
Then we have $\theta_{r,s}(\pi) \ne 0$ if $\ZZ_{r,s}(\pi) \ne 0$, and we are reduced to proving the converse.
However, it was previously only known that if $\theta_{r,s}(\pi) \ne 0$, then $\ZZ_{r', s'}(\pi) \ne 0$ for some $r', s'$ such that $r'+s'=r+s$ and $r' \equiv r \bmod 2$ (see \cite[Proposition 11.5]{gqt}).
Thus we need to prove the following (see Proposition \ref{p:key}), which is the key innovation in this paper.

\begin{prop}
\label{p:intro}
We have
\[
 \ZZ_{r,s}(\pi) \ne 0 \, \Leftrightarrow \, \theta_{r,s}(\pi) \ne 0.
\]
\end{prop}

To prove this proposition, we modify an inductive argument of Atobe \cite{atobe} for the nonvanishing of $\theta_{r,s}(\pi)$, which relies on the Gan--Gross--Prasad conjecture \cite{ggp1} proved by Xue \cite{xue}.
Indeed, if $\theta_{r,s}(\pi) \ne 0$, then we can deduce that there exists a discrete series representation $\pi'$ of $\U(p+1,q)$ such that $\theta_{r,s}(\pi') \ne 0$ and $\Hom_{\U(p,q)}(\pi',\pi) \ne 0$.
In particular, by a result of Beuzart-Plessis \cite{bp}, we have 
\begin{equation}
\label{eq:bp}
 \int_{\U(p,q)} (\pi'(g) v'_1, v'_2) \overline{(\pi(g) v_1, v_2)} \,dg \ne 0
\end{equation}
for some $v'_1, v'_2 \in \pi'$ and $v_1, v_2 \in \pi$.
On the other hand, if we write $\omega'$ for the Weil representation of $\U(p+1,q) \times \U(r,s)$, then matrix coefficients of $\omega'$ (regarded as functions on $\U(p+1,q)$) belong to the Harish-Chandra Schwartz space $\CC(\U(p+1,q))$.
Since we may assume that $\ZZ_{r,s}(\pi') \ne 0$ by induction on $(r+s) - (p+q)$, the projection of these matrix coefficients to the $\pi' \boxtimes \bar{\pi}'$-isotypic component of $\CC(\U(p+1,q))$ is nonzero and hence dense.
This combined with \eqref{eq:bp} implies that 
\[
 \int_{\U(p,q)} (\omega'(g) \varphi'_1, \varphi'_2) \overline{(\pi(g) v_1, v_2)} \,dg \ne 0
\]
for some $\varphi'_1, \varphi'_2 \in \omega'$, from which Proposition \ref{p:intro} follows easily.
We stress that the proof is local and does not rely on Arthur's endoscopic classification, so that Proposition \ref{p:intro} is unconditional.

\subsection*{Acknowledgements}

The author is partially supported by JSPS KAKENHI Grant Number 19H01781.
He would like to thank Hiraku Atobe, Rapha\"el Beuzart-Plessis, Wee Teck Gan, and Hang Xue for useful discussions.
He would also like to thank the referee for helpful comments.

\subsection*{Notation}

For any representation $\pi$, we denote by $\pi^\vee$ the contragredient of $\pi$ and by $\bar{\pi}$ the complex conjugate of $\pi$.
For any real reductive group $G$, we work with the category of $(\g,K)$-modules unless otherwise specified, where $\g$ is the complexified Lie algebra of $G$ and $K$ is a maximal compact subgroup of $G$.
Thus by abuse of terminology, we usually mean a $(\g,K)$-module by a representation of $G$.

\section{Local theta lifting}

In this section, we review the notion of local theta lifting.
We follow the convention in \cite{gi1, gi2}, which is different from that in \cite{kudla,hks}.

\subsection{Hermitian and skew-Hermitian spaces}
\label{ss:herm-spaces}

Let $F$ be a local field of characteristic zero.
Let $E$ be an \'etale quadratic algebra over $F$, so that $E$ is either $F \times F$ or a quadratic extension of $F$.
We denote by $c$ the nontrivial automorphism of $E$ over $F$.
Let $\Tr_{E/F}$ and $\operatorname{N}_{E/F}$ be the trace and norm maps from $E$ to $F$, respectively.
Let $\omega_{E/F}$ be the (possibly trivial) quadratic character of $F^\times$ associated to $E/F$ by local class field theory, so that $\operatorname{Ker}(\omega_{E/F}) = \operatorname{N}_{E/F}(E^\times)$.
Fix an element $\delta \in E^\times$ such that $\Tr_{E/F}(\delta) = 0$.

Fix $\varepsilon = \pm 1$.
Let $V$ be an $n$-dimensional $\varepsilon$-Hermitian space over $E$.
Namely, $V$ is a free $E$-module of rank $n$ equipped with a nondegenerate sesquilinear form $\langle \cdot, \cdot \rangle_V : V \times V \rightarrow E$ satisfying
\[
 \langle a v, b w \rangle_V = a b^c \langle v, w \rangle_V, \quad
 \langle w, v \rangle_V = \varepsilon \langle v, w \rangle_V^c
\]
for $a, b \in E$ and $v, w \in V$.
Let $\det(V) \in E^\times / \operatorname{N}_{E/F}(E^\times)$ be the determinant of the matrix
\[
 (\langle v_i, v_j \rangle_V)_{1 \le i,j \le n},
\]
where $v_1, \dots, v_n$ is a basis of $V$.
Define $\epsilon(V) = \pm 1$ by
\[
 \epsilon(V) =
 \begin{cases}
  \omega_{E/F}((-1)^{\frac{1}{2}n(n-1)} \cdot \det(V)) & \text{if $\varepsilon = +1$;} \\
  \omega_{E/F}((-1)^{\frac{1}{2}n(n-1)} \cdot \det(V) \cdot \delta^{-n}) & \text{if $\varepsilon = -1$.}
 \end{cases}
\]
Note that $\epsilon(V)$ depends on $\delta$ if $\varepsilon = -1$, $E \ne F \times F$, and $n$ is odd.
We denote by $\U(V)$ the unitary group of $V$, i.e.
\[
 \U(V) = \{ g \in \GL(V) \, | \, \text{$\langle gv, gw \rangle_V = \langle v, w \rangle_V$ for all $v, w \in V$} \}.
\]

Recall that given a positive integer $n$, the $n$-dimensional $\varepsilon$-Hermitian spaces over $E$ (up to isometry) are classified as follows.
\begin{itemize}
\item 
If $E = F \times F$, then there is a unique such space.
We denote it by $V_n^+$.
Then we have $\epsilon(V_n^+) = +1$ and $V_n^+ = \VV_n \otimes_F E$ for some $n$-dimensional vector space $\VV_n$ over $F$.
Moreover, the first projection $V_n^+ = \VV_n \times \VV_n \rightarrow \VV_n$ induces an isomorphism $\U(V_n^+) \cong \GL(\VV_n)$.
\item 
If $F$ is nonarchimedean and $E \ne F \times F$, then there are precisely two such spaces, which are distinguished by their signs.
We denote them by $V_n^+$ and $V_n^-$ so that $\epsilon(V_n^+) = +1$ and $\epsilon(V_n^-) = -1$.
\item 
If $F = \R$ and $E = \C$, then there are precisely $n+1$ such spaces, which are distinguished by their signatures.
We denote by $V_{p,q}$ the space of signature $(p,q)$, where $p,q$ are nonnegative integers such that $p+q=n$.
More precisely, we require that $V_{p,q}$ has a basis $v_1, \dots, v_n$ such that
\[
 \langle v_i, v_j \rangle_{V_{p,q}} = \zeta \times 
 \begin{cases}
  1 & \text{if $i = j \le p$;} \\
  -1 & \text{if $i = j > p$;} \\
  0 & \text{if $i \ne j$,}
 \end{cases}
\]
where 
\[
 \zeta = 
 \begin{cases}
  1 & \text{if $\varepsilon = +1$;} \\
  \sqrt{-1} & \text{if $\varepsilon = -1$.}
 \end{cases}
\]
Then we have
\[
 \epsilon(V_{p,q}) = (-1)^{\frac{1}{2}(p-q)(p-q-1)}
\]
if we take $\delta = \sqrt{-1}$.
For uniformity, we write
\[
 V_n^+ = 
 \begin{cases}
  V_{\frac{n}{2}, \frac{n}{2}} & \text{if $n$ is even;} \\
  V_{\frac{n+1}{2}, \frac{n-1}{2}} & \text{if $n$ is odd.}
 \end{cases}
\]
\end{itemize}
Note that $\U(V_n^+)$ is quasi-split over $F$.

\subsection{Theta lifts}

Let $V$ be an $m$-dimensional Hermitian space over $E$ and $W$ an $n$-dimensional skew-Hermitian space over $E$.
We regard $\W = V \otimes_E W$ as a vector space over $F$ and equip it with the symplectic form given by
\[
 \langle \hspace{-1mm} \langle v_1 \otimes w_1, v_2 \otimes w_2 \rangle \hspace{-1mm} \rangle = \Tr_{E/F}(\langle v_1, v_2 \rangle_V \langle w_1, w_2 \rangle_W).
\]
Let $\Sp(\W)$ be the symplectic group of $\W$ and $\Mp(\W)$ the metaplectic $\C^1$-cover of $\Sp(\W)$.
Then it follows from \cite{kudla,hks} that the natural homomorphism $\U(V) \times \U(W) \rightarrow \Sp(\W)$ has a lift
\[
 \iota_{V,W,\chi_V,\chi_W,\psi} : \U(V) \times \U(W) \rightarrow \Mp(\W)
\]
which depends on the choice of the following datum:
\begin{itemize}
\item two unitary characters $\chi_V, \chi_W$ of $E^\times$ such that
\[
 \chi_V|_{F^\times} = \omega_{E/F}^m, \quad
 \chi_W|_{F^\times} = \omega_{E/F}^n;
\]
\item a nontrivial additive character $\psi$ of $F$.
\end{itemize}
Composing this with the Weil representation $\omega_\psi$ of $\Mp(\W)$ relative to $\psi$, we obtain a representation
\[
 \omega_{V, W, \chi_V, \chi_W, \psi} = \omega_\psi \circ \iota_{V, W, \chi_V, \chi_W, \psi}
\]
of $\U(V) \times \U(W)$.
Note that if we apply the construction to the spaces $W$ and $V$ equipped with the Hermitian form $\delta^{-1} \langle \cdot, \cdot \rangle_W$ and the skew-Hermitian form $\delta \langle \cdot, \cdot \rangle_V$, respectively, then we obtain the representation
\[
 \omega_{V, W, \chi_V, \chi_W, \psi} \circ \mathrm{sw},
\]
where $\mathrm{sw}: \U(W) \times \U(V) \rightarrow \U(V) \times \U(W)$ switches factors.
In particular, we can freely switch the roles of $V$ and $W$.

For any irreducible representation $\pi$ of $\U(W)$, we denote by $\theta_{V, W, \chi_V, \chi_W, \psi}(\pi)$ its theta lift to $\U(V)$, i.e.~an irreducible representation of $\U(V)$ such that
\[
 \Hom_{\U(V) \times \U(W)}(\omega_{V,W,\chi_V,\chi_W,\psi}, \theta_{V, W, \chi_V, \chi_W, \psi}(\pi) \boxtimes \pi) \ne 0,
\]
which is uniquely determined (if exists) by the Howe duality \cite{howe2,wal,minguez,gt}.
If such a representation does not exist, we interpret $\theta_{V, W, \chi_V, \chi_W, \psi}(\pi)$ as zero.

\section{Representations of real unitary groups}

In this section, we introduce some representations of real unitary groups which will be needed in this paper.

\subsection{Real unitary groups}
\label{ss:real-notation}

Fix $\varepsilon = \pm 1$.
Let $V$ be an $n$-dimensional $\varepsilon$-Hermitian space over $\C$ of signature $(p,q)$, so that $p + q = n$.
Let $G = \U(V)$ be the unitary group of $V$, which we identify with 
\[
 \U(p,q) = \left\{ g \in \GL_n(\C) \, \left| \, {}^t \bar{g}
 \begin{pmatrix}
  \1_p & \\
  & -\1_q
 \end{pmatrix}
 g = 
 \begin{pmatrix}
  \1_p & \\
  & -\1_q 
 \end{pmatrix}
 \right. \right\}
\]
via the basis as in \S \ref{ss:herm-spaces}
Define a Cartan involution $\theta$ of $G$ by
\[
 \theta(g) = {}^t \bar{g}^{-1}
\]
and let $K = \{ g \in G \, | \, \theta(g) = g \}$ be the associated maximal compact subgroup of $G$.
Let $\g_0$ be the Lie algebra of $G$ and $\t_0$ the Cartan subalgebra of $\g_0$ consisting of diagonal matrices.
Let $\g = \g_0 \otimes_\R \C$ and $\t = \t_0 \otimes_\R \C$ be their complexifications.
We identify $\t$ with $\C^n$ via the isomorphism
\[
 (x_1, \dots, x_n) \mapsto \diag(x_1, \dots, x_n)
\]
and $\t^*$ with $\C^n$ via the basis $\varepsilon_1,\dots,\varepsilon_n$ given by
\[
 \varepsilon_i(\diag(x_1, \dots, x_n)) = x_i.
\]
Define a bilinear form $\langle \cdot, \cdot \rangle : \t^* \times \t^* \rightarrow \C$ by
\[
 \langle \alpha, \beta \rangle = \alpha_1 \beta_1 + \dots + \alpha_n \beta_n
\]
for $\alpha = (\alpha_1,\dots,\alpha_n), \beta = (\beta_1,\dots,\beta_n) \in \t^* \cong \C^n$.
Let $\Delta$ be the set of roots of $\t$ in $\g$, so that 
\[
 \Delta = \{ \pm(\varepsilon_i - \varepsilon_j) \, | \, 1 \le i < j \le n \}.
\]
Let $\Delta_c$ be the set of compact roots in $\Delta$ and take the positive system $\Delta_c^+$ of $\Delta_c$ given by 
\[
 \Delta_c^+ = \{ \varepsilon_i - \varepsilon_j  \, | \, 1 \le i < j \le p \} \cup \{ \varepsilon_i - \varepsilon_j \, | \, p < i < j \le n \}.
\]
For any subspace $\mathfrak{f}$ of $\g$ stable under the adjoint action of $\t$, we denote by $\Delta(\mathfrak{f})$ the set of roots of $\t$ in $\mathfrak{f}$ and put $\rho(\mathfrak{f}) = \frac{1}{2} \sum_{\alpha \in \Delta(\mathfrak{f})} \alpha$.

\subsection{Parabolically induced representations}
\label{ss:parab-ind}

Let $d$ be a nonnegative integer with $d \le \min \{ p, q \}$.
When $d > 0$, we take elements $v'_1, \dots, v'_d, v''_1, \dots, v''_d \in V$ such that
\[
 \langle v'_i, v'_j \rangle_V = \langle v''_i, v''_j \rangle_V = 0, \quad 
 \langle v'_i, v''_j \rangle_V = \delta_{i,j}
\]
and put
\[
 X_i = \C v'_i, \quad
 X^*_i = \C v''_i.
\]
Let $V_0$ be the orthogonal complement of $X_1 \oplus \dots \oplus X_d \oplus X_1^* \oplus \dots \oplus X_d^*$ in $V$, so that $V_0$ is an $\varepsilon$-Hermitian space over $\C$ of signature $(p-d,q-d)$.
Let $P = MU$ be the parabolic subgroup of $G$ stabilizing the flag 
\[
 X_1 \subset X_1 \oplus X_2 \subset \dots \subset X_1 \oplus \dots \oplus X_d,
\]
where $M$ is the Levi component of $P$ stabilizing the flag
\[
 X^*_1 \subset X^*_1 \oplus X^*_2 \subset \dots \subset X^*_1 \oplus \dots \oplus X^*_d
\]
and $U$ is the unipotent radical of $P$.
As in the previous subsection, we identify $M \cong \GL(X_1) \times \dots \times \GL(X_d) \times \U(V_0)$ with $(\C^\times)^d \times \U(p-d,q-d)$.
For any characters $\chi_1, \dots, \chi_d$ of $\C^\times$ and any representation $\pi_0$ of $\U(p-d,q-d)$, we write
\[
 I(\chi_1, \dots, \chi_d, \pi_0) = \Ind^G_P(\chi_1 \boxtimes \dots \boxtimes \chi_d \boxtimes \pi_0)
\]
for the associated normalized parabolically induced representation.
When $d=0$, we interpret $I(\chi_1, \dots, \chi_d, \pi_0)$ as $\pi_0$.

\subsection{(Limits of) discrete series representations}
\label{ss:lds}

Recall that the discrete series representations of $G$ are parametrized by Harish-Chandra parameters (which are dominant for $\Delta_c^+$)
\[
 \lambda = (\lambda_1, \dots, \lambda_n) \in \sqrt{-1} \t_0^*,
\]
where
\begin{itemize}
\item $\lambda_i \in \Z + \frac{n-1}{2}$;
\item $\lambda_i \ne \lambda_j$ if $i \ne j$;
\item $\lambda_1 > \dots > \lambda_p$ and $\lambda_{p+1} > \dots > \lambda_n$.
\end{itemize}
More generally, the (limits of) discrete series representations of $G$ are parametrized by pairs $(\lambda,\Psi)$ consisting of $\lambda \in \sqrt{-1} \t_0^*$ of the form
\begin{equation}
\label{eq:lambda}
 \lambda = (\underbrace{\lambda_1, \dots, \lambda_1}_{p_1}, \dots, \underbrace{\lambda_k, \dots, \lambda_k}_{p_k}, \underbrace{\lambda_1, \dots, \lambda_1}_{q_1}, \dots, \underbrace{\lambda_k, \dots, \lambda_k}_{q_k}),
\end{equation}
where
\begin{itemize}
\item $\lambda_i \in \Z + \frac{n-1}{2}$;
\item $\lambda_1 > \dots > \lambda_k$;
\item $p_i, q_j \ge 0$;
\item $(p_i,q_i) \ne (0,0)$ and $|p_i - q_i| \le 1$ for all $i$;
\item $p_1 + \dots + p_k = p$ and $q_1 + \dots + q_k = q$,
\end{itemize}
and a positive system $\Psi$ of $\Delta$ such that
\begin{itemize}
\item $\Delta_c^+ \subset \Psi$;
\item $\langle \lambda, \alpha \rangle \ge 0$ for all $\alpha \in \Psi$;
\item if $\alpha$ is a simple root in $\Psi$ such that $\langle \lambda, \alpha \rangle = 0$, then $\alpha$ is noncompact.
\end{itemize}
Note that if $(\lambda, \Psi)$ corresponds to a discrete series representation, then $\Psi$ is uniquely determined by $\lambda$.

\subsection{Tempered representations}

We say that a character $\chi$ of $\C^\times$ is conjugate-selfdual with sign $+1$ (resp.~$-1$) if $\chi|_{\R^\times} = \mathbbm{1}$ (resp.~$\chi|_{\R^\times} = \omega_{\C/\R}$).

Recall that any irreducible tempered representation of $G$ can be realized as a subrepresentation of $I(\chi_1, \dots, \chi_d, \pi_0)$, where
\begin{itemize}
\item $d$ is a nonnegative integer with $d \le \min \{ p, q \}$;
\item $\chi_1, \dots, \chi_d$ are unitary characters of $\C^\times$;
\item $\pi_0$ is a discrete series representation of $\U(p-d,q-d)$.
\end{itemize}
More precisely, we have the following results of Knapp--Zuckerman \cite{kz1, kz2}.

\begin{lem}
Assume that $p,q > 0$.
Let $\chi$ be a conjugate-selfdual character of $\C^\times$ with sign $(-1)^{n-1}$, so that 
\[
 \chi(z) = \bigg( \frac{z}{\sqrt{z \bar{z}}} \bigg)^{2 \kappa} 
\]
for some $\kappa \in \Z + \frac{n-1}{2}$.
Let $\pi_0$ be a (limit of) discrete series representation of $\U(p-1,q-1)$ associated to a pair $(\lambda_0, \Psi_0)$ as in \S \ref{ss:lds}.
\begin{itemize}
\item 
If $\kappa = \lambda_{0,i}$ for some $i$, then $I(\chi,\pi_0)$ is irreducible and is a limit of discrete series representation of $G$.
\item 
If $\kappa \ne \lambda_{0,i}$ for all $i$, then we have $I(\chi,\pi_0) = \pi \oplus \pi'$, where $\pi$ and $\pi'$ are distinct limits of discrete series representations of $G$.
\end{itemize}
(See \S \ref{sss:packets-real-lds} below for more explicit description.)
\end{lem}

\begin{lem}
\label{l:irred-temp}
Let $d$ be a nonnegative integer with $d \le \min \{ p, q \}$.
Let $\xi_1, \dots, \xi_d$ be unitary characters of $\C^\times$ which are not conjugate-selfdual with sign $(-1)^{n-1}$.
Let $\pi_0$ be a (limit of) discrete series representation of $\U(p-d,q-d)$.
Then $I(\xi_1, \dots, \xi_d, \pi_0)$ is irreducible and tempered.
\end{lem}

In particular, we may write an irreducible tempered representation $\pi$ of $G$ as
\begin{equation}
\label{eq:temp}
 \pi = I(\xi_1, \dots, \xi_d, \pi_0),
\end{equation}
where 
\begin{itemize}
\item $d$ is a nonnegative integer with $d \le \min \{ p, q \}$;
\item $\xi_1, \dots, \xi_d$ are unitary characters of $\C^\times$ which are not conjugate-selfdual with sign $(-1)^{n-1}$;
\item $\pi_0$ is a (limit of) discrete series representation of $\U(p-d,q-d)$.
\end{itemize}

\subsection{Cohomologically induced representations}

For $x \in \sqrt{-1} \t_0$, let $\l$ (resp.~$\u$) be the sum of zero (resp.~positive) eigenspaces of $\operatorname{ad}(x)$ in $\g$.
Then $\q = \l \oplus \u$ is a $\theta$-stable parabolic subalgebra of $\g$.
We also write 
\[
 \q = \q(x)
\]
to indicate the dependence on $x$.
Let $L$ be the normalizer of $\q$ in $G$, so that $\l$ is the complexified Lie algebra of $L$.
If $x$ is of the form
\begin{equation}
\label{eq:x} 
 x = (\underbrace{x_1,\dots,x_1}_{p_1}, \dots, \underbrace{x_k,\dots,x_k}_{p_k}, \underbrace{x_1,\dots,x_1}_{q_1}, \dots, \underbrace{x_k,\dots,x_k}_{q_k}),
\end{equation}
where
\begin{itemize}
\item $x_i \in \R$;
\item $x_1 > \dots > x_k$;
\item $p_i, q_j \ge 0$;
\item $(p_i,q_i) \ne (0,0)$ for all $i$;
\item $p_1 + \dots + p_k = p$ and $q_1 + \dots + q_k = q$,
\end{itemize}
then we have
\[
 L \cong \U(p_1,q_1) \times \dots \times \U(p_k,q_k).
\]

Let $\lambda$ be the differential of a character of $L$ and regard it as an element in $\sqrt{-1} \t_0^*$ by restriction.
We consider a cohomologically induced representation
\[
 A_\q(\lambda)
\]
defined by \cite[(5.6)]{kv}.
The following summarizes some basic properties of $A_\q(\lambda)$.
\begin{itemize}
\item
The infinitesimal character of $A_\q(\lambda)$ is $\lambda+\rho$.
Here we choose a positive system $\Delta^+$ of $\Delta$ containing $\Delta(\u)$ and put $\rho = \frac{1}{2} \sum_{\alpha \in \Delta^+} \alpha$.
\item
If $\lambda$ is in the good range, i.e.
\[
 \langle \lambda + \rho, \alpha \rangle > 0
\]
for all $\alpha \in \Delta(\u)$, then $A_\q(\lambda)$ is nonzero and irreducible.
(Note that this condition does not depend the choice of $\rho$.)
\item 
If $\lambda$ is in the weakly fair range, i.e.
\[
 \langle \lambda + \rho(\u), \alpha \rangle \ge 0
\]
for all $\alpha \in \Delta(\u)$, then $A_\q(\lambda)$ is unitary (but possibly zero).
\end{itemize}
We also have the following, which is special to unitary groups.
\begin{itemize}
\item 
If $\lambda$ is in the weakly fair range and $A_\q(\lambda)$ is nonzero, then it is irreducible by \cite{matumoto, trapa}.
\item 
There is an algorithm due to Trapa \cite{trapa} which determines the nonvanishing and the Langlands parameter of $A_\q(\lambda)$ with $\lambda$ in the weakly fair range.
\end{itemize} 
Moreover, we have the following irreducibility result of Matumoto \cite[Theorem 3.3.1(2)]{matumoto}.

\begin{lem}
\label{l:irred-cohom}
Let $d$ be a nonnegative integer with $d \le \min \{ p, q \}$.
Let $\xi_1, \dots, \xi_d$ be unitary characters of $\C^\times$ which are not conjugate-selfdual with sign $(-1)^{n-1}$.
Let $\pi_0$ be a cohomologically induced representation of $\U(p-d,q-d)$ which is weakly fair and nonzero.
Then $I(\xi_1, \dots, \xi_d, \pi_0)$ is irreducible.
\end{lem}

In this paper, we will use a normalized version of $A_\q(\lambda)$, which makes the statement of the main theorems cleaner.
Put 
\[
 \nA_\q(\lambda) = A_\q(\lambda - \rho(\u)), 
\]
where if $x$ is of the form \eqref{eq:x}, then $\lambda \in \sqrt{-1} \t_0^*$ is of the form 
\[
 \lambda = (\underbrace{\lambda_1,\dots,\lambda_1}_{p_1}, \dots, \underbrace{\lambda_k,\dots,\lambda_k}_{p_k}, \underbrace{\lambda_1,\dots,\lambda_1}_{q_1}, \dots, \underbrace{\lambda_k,\dots,\lambda_k}_{q_k})
\]
with $\lambda_i \in \Z + \frac{1}{2}(n-p_i-q_i)$.
Then
\begin{itemize}
\item
$\nA_\q(\lambda)$ is good if and only if $\lambda_i \ge \lambda_{i+1} + \frac{1}{2}(p_i + q_i + p_{i+1} + q_{i+1})$ for all $i$;
\item 
$\nA_\q(\lambda)$ is weakly fair if and only if $\lambda_i \ge \lambda_{i+1}$ for all $i$,
\end{itemize}
noting that 
\[
 \rho - \rho(\u) = (\alpha^{(1)}_1, \dots, \alpha^{(1)}_{p_1}, \dots, \alpha^{(k)}_1, \dots, \alpha^{(k)}_{p_k}, \beta^{(1)}_1, \dots, \beta^{(1)}_{q_1}, \dots, \beta^{(k)}_1, \dots, \beta^{(k)}_{q_k})
\]
with 
\[
 \{ \alpha^{(i)}_1, \dots, \alpha^{(i)}_{p_i}, \beta^{(i)}_1, \dots, \beta^{(i)}_{q_i} \}
 = \left\{ \left. \frac{p_i+q_i+1}{2} - j \, \right| \, 1 \le j \le p_i+q_i \right\}.
\]
With this normalization, we may write a (limit of) discrete series representation $\pi$ of $G$ associated to a pair $(\lambda, \Psi)$ as in \S \ref{ss:lds} as
\begin{equation}
\label{eq:lds}
 \pi = \nA_{\b}(\lambda), 
\end{equation}
where $\b = \t \oplus \mathfrak{n}$ is the $\theta$-stable Borel subalgebra of $\g$ with nilpotent radical $\mathfrak{n}$ such that $\Delta(\mathfrak{n}) = \Psi$ (see \cite[\S XI.8]{kv}).
More explicitly, if $\lambda$ is of the form \eqref{eq:lambda} and $\b = \q(x)$ is associated to 
\[
 x = (x^{(1)}_1, \dots, x^{(1)}_{p_1}, \dots, x^{(k)}_1, \dots, x^{(k)}_{p_k}, y^{(1)}_1, \dots, y^{(1)}_{q_1}, \dots, y^{(k)}_1, \dots, y^{(k)}_{q_k}),
\]
then the conditions on $\Psi$ in \S \ref{ss:lds} are equivalent to the following conditions on $x$:
\begin{itemize}
\item $x^{(1)}_1 > \dots > x^{(1)}_{p_1} > \dots > x^{(k)}_1 > \dots > x^{(k)}_{p_k}$;
\item $y^{(1)}_1 > \dots > y^{(1)}_{q_1} > \dots > y^{(k)}_1 > \dots > y^{(k)}_{q_k}$;
\item $x^{(i)}_{p_i} > y^{(i+1)}_1$ for all $1 \le i < k$;
\item $y^{(i)}_{q_i} > x^{(i+1)}_1$ for all $1 \le i < k$;
\item if $p_i - q_i = 0$, then either
\[
 x^{(i)}_1 > y^{(i)}_1 > x^{(i)}_2 > y^{(i)}_2 > \dots > x^{(i)}_{p_i} > y^{(i)}_{q_i}
\]
or 
\[
 y^{(i)}_1 > x^{(i)}_1 > y^{(i)}_2 > x^{(i)}_2 > \dots > y^{(i)}_{q_i} > x^{(i)}_{p_i};
\]
\item if $p_i - q_i = 1$, then 
\[
 x^{(i)}_1 > y^{(i)}_1 > x^{(i)}_2 > y^{(i)}_2 > \dots > x^{(i)}_{q_i} > y^{(i)}_{q_i} > x^{(i)}_{p_i};
\]
\item if $p_i - q_i = -1$, then
\[
 y^{(i)}_1 > x^{(i)}_1 > y^{(i)}_2 > x^{(i)}_2 > \dots > y^{(i)}_{p_i} > x^{(i)}_{p_i} > y^{(i)}_{q_i}.
\]
\end{itemize}
Note that if $\pi$ is a discrete series representation, then $\b$ is uniquely determined by $\lambda$.

\section{Statement of the main theorems}

In this section, we state the main theorems of this paper, which describe the theta lifts of tempered representations of real unitary groups explicitly.

\subsection{Setup}
\label{ss:main-theorem-setup}

We consider the theta lifting from $\U(W)$ to $\U(V)$, where $W$ is an $n$-dimensional skew-Hermitian space over $\C$ and $V$ is an $m$-dimensional Hermitian space over $\C$.
Let $(p,q)$ and $(r,s)$ be the signatures of $W$ and $V$, respectively, so that $p+q = n$ and $r+s = m$. 
As in \S \ref{ss:real-notation}, we identify $\U(W)$ and $\U(V)$ with $\U(p,q)$ and $\U(r,s)$, respectively.

From now on, we take the characters $\chi_V, \chi_W$ of $\C^\times$ given by
\[
 \chi_V(z) = \left( \frac{z}{\sqrt{z \bar{z}}} \right)^{m_0}, \quad
 \chi_W(z) = \left( \frac{z}{\sqrt{z \bar{z}}} \right)^{n_0} 
\]
for some fixed integers $m_0, n_0$ such that
\[
 m_0 \equiv m \bmod 2, \quad
 n_0 \equiv n \bmod 2,
\]
and the character $\psi$ of $\R$ given by
\[
 \psi(x) = e^{-2 \pi \sqrt{-1} x}.
\]
(We make this choice so that Lemma \ref{l:K-type-corresp} below holds.)
Then we write the theta lift of an irreducible representation $\pi$ of $\U(W) = \U(p,q)$ to $\U(V) = \U(r,s)$ as
\[
 \theta_{r,s}(\pi) = \theta_{V, W, \chi_V, \chi_W, \psi}(\pi).
\]

\subsection{Explicit description of theta lifts}

We now state our main theorems.

\begin{thm}
\label{t:main-lds}
Let $\pi$ be a (limit of) discrete series representation of $\U(W) = \U(p,q)$ and write $\pi = \nA_\b(\lambda)$ as in \eqref{eq:lds}.
Assume that its theta lift $\theta_{r,s}(\pi)$ to $\U(V) = \U(r,s)$ is nonzero.
Then we have 
\[
 \theta_{r,s}(\pi) = \nA_\q(\lambda'), 
\]
where $\q$ and $\lambda'$ are given as follows.
\begin{enumerate}
\item 
\label{main-lds-1}
Assume that $m > n$.
Write 
\[
 \lambda = (\alpha_1, \dots, \alpha_{p^+}, \beta_1, \dots, \beta_{p^-}, \gamma_1, \dots, \gamma_{q^+}, \delta_1, \dots, \delta_{q^-}) + \bigg( \frac{m_0}{2}, \dots, \frac{m_0}{2} \bigg),
\]
where 
\begin{itemize}
\item $\alpha_i, \gamma_j > 0$ and $\beta_i, \delta_j \le 0$;
\item $p^++p^- = p$ and $q^++q^- = q$,
\end{itemize}
and $\b = \q(x)$ with 
\[
 x = (x^+_1, \dots, x^+_{p^+}, x^-_1, \dots, x^-_{p^-}, y^+_1, \dots, y^+_{q^+}, y^-_1, \dots, y^-_{q^-}).
\]
We assume without loss of generality that 
\begin{itemize}
\item $x^+_1 > \dots > x^+_{p^+} > 0 > x^-_1 > \dots > x^-_{p^-}$;
\item $y^+_1 > \dots > y^+_{q^+} > 0 > y^-_1 > \dots > y^-_{q^-}$.
\end{itemize}
Then
\begin{itemize}
\item $p^++q^- \le r$ and $p^-+q^+ \le s$;
\item $\lambda'$ is given by
\[
 \lambda' = (\alpha_1, \dots, \alpha_{p^+}, \underbrace{0, \dots, 0}_{r-p^+-q^-}, \delta_1, \dots, \delta_{q^-}, \gamma_1, \dots, \gamma_{q^+}, \underbrace{0, \dots, 0}_{s-p^--q^+}, \beta_1, \dots, \beta_{p^-}) + \bigg(\frac{n_0}{2}, \dots, \frac{n_0}{2}\bigg);
\]
\item $\q = \q(x')$ is associated to
\[
 x' = (x^+_1, \dots, x^+_{p^+}, \underbrace{0, \dots, 0}_{r-p^+-q^-}, y^-_1, \dots, y^-_{q^-}, y^+_1, \dots, y^+_{q^+}, \underbrace{0, \dots, 0}_{s-p^--q^+}, x^-_1, \dots, x^-_{p^-}).
\]
\end{itemize}
\item
\label{main-lds-2}
Assume that $m \le n$.
Put $k = n - m$.
Write
\begin{align*}
 \lambda & = \bigg( \alpha_1, \dots, \alpha_{p^+}, 
 \underbrace{\frac{k-1}{2}, \dots, \frac{k-1}{2}}_{p_1},
 \underbrace{\frac{k-3}{2}, \dots, \frac{k-3}{2}}_{p_2}, \dots,
 \underbrace{-\frac{k-1}{2}, \dots, -\frac{k-1}{2}}_{p_k},
 \beta_1, \dots, \beta_{p^-}, \\
 & \phantom{{} = \bigg(}
 \gamma_1, \dots, \gamma_{q^+}, 
 \underbrace{\frac{k-1}{2}, \dots, \frac{k-1}{2}}_{q_1},
 \underbrace{\frac{k-3}{2}, \dots, \frac{k-3}{2}}_{q_2}, \dots,
 \underbrace{-\frac{k-1}{2}, \dots, -\frac{k-1}{2}}_{q_k},
 \delta_1, \dots, \delta_{q^-} \bigg) \\
 & + \bigg( \frac{m_0}{2}, \dots, \frac{m_0}{2} \bigg),
\end{align*}
where
\begin{itemize}
\item $\alpha_i, \gamma_j > \frac{k-1}{2}$ and $\beta_i, \delta_j < -\frac{k-1}{2}$;
\item $p_i, q_j \ge 0$;
\item $|p_i-q_i| \le 1$ for all $i$;
\item $p^++p^-+p_1 + \dots + p_k = p$ and $q^++q^-+q_1 + \dots +q_k = q$,
\end{itemize}
and $\b = \q(x)$ with
\begin{align*}
 x & = (x^+_1, \dots, x^+_{p^+}, x^{(1)}_1, \dots, x^{(1)}_{p_1}, x^{(2)}_1, \dots, x^{(2)}_{p_2}, \dots, x^{(k)}_1, \dots, x^{(k)}_{p_k}, x^-_1, \dots, x^-_{p^-}, \\
 & \phantom{{} = (} y^+_1, \dots, y^+_{q^+}, y^{(1)}_1, \dots, y^{(1)}_{q_1}, y^{(2)}_1, \dots, y^{(2)}_{q_2}, \dots, y^{(k)}_1, \dots, y^{(k)}_{q_k}, y^-_1, \dots, y^-_{q^-}).
\end{align*}
(When $k=0$, we interpret $\lambda$ and $x$ as 
\[
 (\alpha_1, \dots, \alpha_{p^+}, \beta_1, \dots, \beta_{p^-}, 
 \gamma_1, \dots, \gamma_{q^+}, \delta_1, \dots, \delta_{q^-})
 + \bigg( \frac{m_0}{2}, \dots, \frac{m_0}{2} \bigg)
\]
and
\[
 x = (x^+_1, \dots, x^+_{p^+}, x^-_1, \dots, x^-_{p^-}, y^+_1, \dots, y^+_{q^+}, y^-_1, \dots, y^-_{q^-}),
\]
respectively.) 
Then
\begin{itemize}
\item $p_i + q_i >0$ for all $1 \le i \le k$;
\item if $k \ge 2$, then either the conditions
\begin{enumerate}
\item 
\label{e+1}
$p_i - q_i = 1$ for all $1 < i < k$;
\item
\label{e+2}
$p_1 - q_1 = 1$ or $0$;
\item
\label{e+3}
$p_k - q_k = 1$ or $0$,
\end{enumerate}
or the conditions
\begin{enumerate}[resume]
\item
\label{e-1}
$p_i - q_i = -1$ for all $1 < i < k$;
\item
\label{e-2}
$p_1 - q_1 = -1$ or $0$;
\item
\label{e-3}
$p_k - q_k = -1$ or $0$
\end{enumerate}
hold;
\item if $k \ge 2$, then
\[
\begin{cases}
 y^{(1)}_1 > x^{(1)}_1 > \dots > y^{(1)}_{q_1} > x^{(1)}_{p_1} &
 \text{if the conditions \eqref{e+1}, \eqref{e+2}, \eqref{e+3} hold and  $p_1 - q_1 = 0$;} \\
 x^{(k)}_1 > y^{(k)}_1 > \dots > x^{(k)}_{p_k} > y^{(k)}_{q_k} & 
 \text{if the conditions \eqref{e+1}, \eqref{e+2}, \eqref{e+3} hold and  $p_k - q_k = 0$;} \\
 x^{(1)}_1 > y^{(1)}_1 > \dots > x^{(1)}_{p_1} > y^{(1)}_{q_1} & 
 \text{if the conditions \eqref{e-1}, \eqref{e-2}, \eqref{e-3} hold and  $p_1 - q_1 = 0$;} \\
 y^{(k)}_1 > x^{(k)}_1 > \dots > y^{(k)}_{q_k} > x^{(k)}_{p_k} & 
 \text{if the conditions \eqref{e-1}, \eqref{e-2}, \eqref{e-3} hold and  $p_k - q_k = 0$;}
\end{cases} 
\]
\item
$r = p^+ + q^- + r_1 + \dots + r_k$ and $s = p^- + q^+ + s_1 + \dots + s_k$, where
\[
 (r_i, s_i) =
\begin{cases}
 (p_i-1, q_i) & \text{if $p_i-q_i = 1$, or $p_i-q_i = 0$ and $y^{(i)}_1 > x^{(i)}_1 > \dots > y^{(i)}_{q_i} > x^{(i)}_{p_i}$;} \\
 (p_i, q_i-1) & \text{if $p_i-q_i = -1$, or $p_i-q_i = 0$ and $x^{(i)}_1 > y^{(i)}_1 > \dots > x^{(i)}_{p_i} > y^{(i)}_{q_i}$;}
\end{cases}
\]
\item $\lambda'$ is given by 
\begin{align*}
 \lambda' & = \bigg( \alpha_1, \dots, \alpha_{p^+}, \underbrace{\frac{k-1}{2}, \dots, \frac{k-1}{2}}_{r_1}, \underbrace{\frac{k-3}{2}, \dots, \frac{k-3}{2}}_{r_2}, \dots, 
\underbrace{-\frac{k-1}{2}, \dots, -\frac{k-1}{2}}_{r_k}, 
 \delta_1, \dots, \delta_{q^-}, \\
 & \phantom{{} = \bigg(}
 \gamma_1, \dots, \gamma_{q^+},
\underbrace{\frac{k-1}{2}, \dots, \frac{k-1}{2}}_{s_1}, \underbrace{\frac{k-3}{2}, \dots, \frac{k-3}{2}}_{s_2}, \dots, 
\underbrace{-\frac{k-1}{2}, \dots, -\frac{k-1}{2}}_{s_k}, 
 \beta_1, \dots, \beta_{p^-} \bigg) \\
 & + \bigg( \frac{n_0}{2}, \dots, \frac{n_0}{2} \bigg);
\end{align*}
\item 
$\q = \q(x')$ is associated to
\begin{align*}
 x' & = (x^+_1, \dots, x^+_{p^+}, z^{(1)}_1, \dots, z^{(1)}_{r_1}, z^{(2)}_1, \dots, z^{(2)}_{r_2}, \dots, z^{(k)}_1, \dots, z^{(k)}_{r_k}, y^-_1, \dots, y^-_{q^-}, \\
 & \phantom{{} = (} y^+_1, \dots, y^+_{q^+}, w^{(1)}_1, \dots, w^{(1)}_{s_1}, w^{(2)}_1, \dots, w^{(2)}_{s_2}, \dots, w^{(k)}_1, \dots, w^{(k)}_{s_k},  x^-_1, \dots, x^-_{p^-})
\end{align*}
such that
\[
\begin{cases}
 z^{(i)}_1 > w^{(i)}_1 > \dots > z^{(i)}_{r_i} > w^{(i)}_{s_i} & 
 \text{if $p_i - q_i = 1$;} \\
 w^{(i)}_1 > z^{(i)}_1 > \dots > w^{(i)}_{s_i} > z^{(i)}_{r_i} &
 \text{if $p_i - q_i = -1$.}
\end{cases} 
\]
\end{itemize}
\end{enumerate}
\end{thm}

In particular, $\theta_{r,s}(\pi)$ is a (limit of) discrete series representation when $m \le n+1$.
Also, this theorem shows that if $\theta_{r,s}(\pi)$ is nonzero, then the associated cohomologically induced representation $\nA_\q(\lambda')$ is nonzero, which was not known when $m \ge n+2$ and $\nA_\q(\lambda')$ is not good.
It is not clear to the author whether this nonvanishing follows directly from a result of Trapa \cite[Theorem 7.9]{trapa}.

\begin{thm}
\label{t:main-temp}
Let $\pi$ be an irreducible tempered representation of $\U(W) = \U(p,q)$ and write $\pi = I(\xi_1, \dots, \xi_d, \pi_0)$ as in \eqref{eq:temp}.
Assume that its theta lift $\theta_{r,s}(\pi)$ to $\U(V) = \U(r,s)$ is nonzero.
Then we have $d \le \min \{ r,s \}$ and
\[
 \theta_{r,s}(\pi) = I(\xi_1 \chi_V^{-1} \chi_W, \dots, \xi_d \chi_V^{-1} \chi_W, \theta_{r-d,s-d}(\pi_0)).
\]
\end{thm}

The rest of this paper is devoted to the proof of Theorems \ref{t:main-lds} and \ref{t:main-temp}.

\section{$L$- and $A$-packets}

In this section, we describe the representations in some local $L$- and $A$-packets for unitary groups explicitly.

\subsection{Parameters and packets}

Let $F$ be a local field of characteristic zero and $W_F$ the Weil group of $F$.
Put
\[
 L_F =
 \begin{cases}
  W_F & \text{if $F$ is archimedean;} \\
  W_F \times \SL_2(\C) & \text{if $F$ is nonarchimedean.}
 \end{cases}
\]
Let $E$ be a quadratic extension of $F$.
Following \cite[\S 8]{ggp1}, we regard an $L$-parameter $\phi : L_F \rightarrow {}^L \U_n$ (resp.~an $A$-parameter $\phi : L_F \times \SL_2(\C) \rightarrow {}^L \U_n$) for $\U_n$ as an $n$-dimensional conjugate-selfdual representation of $L_E$ (resp.~$L_E \times \SL_2(\C)$) with sign $(-1)^{n-1}$.
Here $\U_n$ stands for the unitary group of any $n$-dimensional Hermitian or skew-Hermitian space over $E$ and ${}^L \U_n =  \GL_n(\C) \rtimes W_F$ is the $L$-group of $\U_n$.
For any such a parameter $\phi$, we denote by $S_\phi$ the component group of the centralizer of the image of $\phi$ in $\GL_n(\C)$ and by $\widehat{S}_\phi$ the group of characters of $S_\phi$.
Note that $S_\phi$ is a finitely generated free $\Z/2\Z$-module.
For any positive integer $d$, we denote by $S_d$ the unique $d$-dimensional irreducible representation of $\SL_2(\C)$.

Fix $\varepsilon = \pm 1$.
Let $V$ be an $n$-dimensional $\varepsilon$-Hermitian space over $E$ and $\Irr(\U(V))$ the set of equivalence classes of irreducible representations of $\U(V)$.
Then the local Langlands correspondence \cite{mok,kmsw,mr18} gives a partition of $\Irr(\U(V))$ into finite sets called $L$-packets:
\begin{equation}
\label{eq:llc}
 \Irr(\U(V)) = \bigsqcup_\phi \Pi_\phi(\U(V)),  
\end{equation}
where $\phi$ runs over $L$-parameters for $\U_n$.
Moreover, given the choice of a Whittaker datum, there exists a canonical bijection
\[
 \bigsqcup_V \Pi_\phi(\U(V)) \, \leftrightarrow \, \widehat{S}_\phi,
\]
where $V$ runs over isometry classes of $n$-dimensional $\varepsilon$-Hermitian spaces over $E$.
We denote by $\pi(\phi, \eta)$ the irreducible representation associated to $\eta \in \widehat{S}_\phi$.

To any $A$-parameter $\phi$ for $\U_n$, Arthur's endoscopic classification \cite{mok,kmsw} assigns a finite set called an $A$-packet
\[
 \Pi_{\phi}(\U(V))
\]
consisting of semisimple representations of $\U(V)$ of finite length, which are indexed by $\widehat{S}_\phi$.
We denote by $\sigma(\phi, \eta)$ the representation associated to $\eta \in \widehat{S}_{\phi}$.

\subsection{Whittaker data}
\label{ss:whit}

To index the representations in $L$- and $A$-packets as in the previous subsection, we take the following Whittaker datum (which is a conjugacy class of pairs $(N,\psi_N)$ consisting of the unipotent radical $N$ of a Borel subgroup of $\U(V_n^+)$ and a generic character $\psi_N$ of $N$) in this paper.
If $n$ is odd, then there is a unique Whittaker datum.
Thus assume that $n$ is even.
Then by \cite[Proposition 12.1]{ggp1}, the Whittaker data are parametrized by $\mathrm{N}_{E/F}(E^\times)$-orbits of nontrivial additive characters of $E/F$ (resp.~$F$) if $\varepsilon = +1$ (resp.~$\varepsilon = -1$).
On the other hand, we have fixed an element $\delta \in E^\times$ such that $\Tr_{E/F}(\delta) = 0$ and a nontrivial additive character $\psi$ of $F$.
Define a nontrivial additive character $\psi^E$ of $E/F$ by $\psi^E(x) = \psi(\frac{1}{2} \Tr_{E/F}(\delta x))$.
Following \cite[\S 2.4]{gi2}, we take the Whittaker datum associated to $\psi^E$ (resp.~$\psi$) if $\varepsilon = +1$ (resp.~$\varepsilon = -1$).

If $F = \R$, we always assume that $\delta = \sqrt{-1}$ and $\psi(x) = e^{- 2 \pi \sqrt{-1} x}$.
Then our Whittaker datum agrees with the Whittaker datum $\mathfrak{w}_+$ as in \cite[\S A.3]{atobe}.
Moreover, by \cite[Theorem A.4]{atobe}, it also agrees with the Whittaker datum as in \cite[Remarque 4.5]{mr19}.

\subsection{The real case}

Suppose that $F = \R$.
For any $\kappa \in \frac{1}{2} \Z$, we define a character $\chi_\kappa$ of $W_\C = \C^\times$ by 
\[
 \chi_\kappa(z) = \bigg( \frac{z}{\sqrt{z \bar{z}}} \bigg)^{2 \kappa}.
\]
For any character $\xi$ of $\C^\times$, we define another character $\check{\xi}$ of $\C^\times$ by $\check{\xi}(z) = \xi(\bar{z})^{-1}$.

\subsubsection{(Limits of) discrete series $L$-packets}
\label{sss:packets-real-lds}

Let $\Irr_{\mathrm{lds}}(\U(p,q))$ be the set of equivalence classes of (limits of) discrete series representations of $\U(p,q)$.
Then \eqref{eq:llc} restricts to a partition
\[
 \Irr_{\mathrm{lds}}(\U(p,q)) = \bigsqcup_\phi \Pi_\phi(\U(p,q)),
\]
where $\phi$ runs over (limits of) discrete series $L$-parameters for $\U_n$ with $n = p+q$.
Here we say that an $L$-parameter $\phi$ for $\U_n$ is (limit of) discrete series if $\phi$ is of the form 
\[
 \phi = \chi_{\kappa_1} \oplus \dots \oplus \chi_{\kappa_n}
\]
with $\kappa_i \in \Z + \frac{n-1}{2}$.
For such a parameter $\phi$, we assume without loss of generality that $\kappa_1 \ge \dots \ge \kappa_n$ and identify $S_\phi$ with a quotient of a free $\Z/2\Z$-module
\[
 \widetilde{S}_\phi = (\Z/2\Z) e_1 \oplus \dots \oplus (\Z/2\Z) e_n,
\]
where $e_i$ corresponds to $\chi_{\kappa_i}$, in such a way that $\widehat{S}_\phi$ consists of the characters $\eta$ of $\widetilde{S}_\phi$ satisfying 
\[
 \eta(e_i) = \eta(e_j)
\]
for all $i,j$ such that $\kappa_i = \kappa_j$.

Let $\eta \in \widehat{S}_\phi$.
For $1 \le i \le n$, we define a pair of integers $(p_i,q_i)$ by
\[
 (p_i, q_i) = 
 \begin{cases}
  (1,0) & \text{if $\eta(e_i) = (-1)^{i-1}$;} \\
  (0,1) & \text{if $\eta(e_i) = (-1)^i$.}
 \end{cases}
\]
Then by \cite[Th\'eor\`eme 1.1]{mr19}, $\pi(\phi,\eta)$ is a representation of $\U(p,q)$ if and only if
\[
 p = p_1 + \dots + p_n, \quad
 q = q_1 + \dots + q_n,
\]
in which case we have
\[
 \eta(e_1 + \dots + e_n) = (-1)^{\frac{1}{2}(p-q)(p-q-1)}
\]
and 
\[
 \pi(\phi,\eta) = \nA_\b(\lambda).
\]
Here $\lambda$ is given by
\[
 \lambda = (\underbrace{\kappa_1,\dots,\kappa_1}_{p_1}, \dots, \underbrace{\kappa_n,\dots,\kappa_n}_{p_n}, \underbrace{\kappa_1,\dots,\kappa_1}_{q_1}, \dots, \underbrace{\kappa_n,\dots,\kappa_n}_{q_n})
\]
and $\b = \q(x)$ is associated to
\[
 x = (\underbrace{x_1,\dots,x_1}_{p_1}, \dots, \underbrace{x_n,\dots,x_n}_{p_n}, \underbrace{x_1,\dots,x_1}_{q_1}, \dots, \underbrace{x_n,\dots,x_n}_{q_n})
\]
for any $x_1, \dots, x_n \in \R$ such that $x_1 > \dots > x_n$.

Assume that $p,q > 0$.
Let $\chi$ be a conjugate-selfdual character of $\C^\times$ with sign $(-1)^{n-1}$, so that $\chi = \chi_\kappa$ for some $\kappa \in \Z + \frac{n-1}{2}$.
Let $\pi_0$ be a (limit of) discrete series representation of $\U(p-1,q-1)$ and write $\pi_0 = \pi(\phi_0, \eta_0)$, where $\phi_0$ is a (limit of) discrete series $L$-parameter for $\U_{n-2}$ and $\eta_0$ is a character of $S_{\phi_0}$.
Define a (limit of) discrete series $L$-parameter for $\U_n$ by 
\[
 \phi = 2 \chi \oplus \phi_0.
\]
We may naturally identify $S_{\phi_0}$ with a subgroup of $S_\phi$.
Then we have 
\[
 I(\chi, \pi_0) = \bigoplus_\eta \pi(\phi, \eta),
\]
where $\eta$ runs over elements in $\widehat{S}_\phi$ such that $\eta|_{S_{\phi_0}} = \eta_0$.

\subsubsection{Tempered $L$-packets}

Let $\Irr_{\mathrm{temp}}(\U(p,q))$ be the set of equivalence classes of irreducible tempered representations of $\U(p,q)$.
Then \eqref{eq:llc} restricts to a partition
\[
 \Irr_{\mathrm{temp}}(\U(p,q)) = \bigsqcup_\phi \Pi_\phi(\U(p,q)),  
\]
where $\phi$ runs over tempered $L$-parameters for $\U_n$ with $n=p+q$.
Here we say that an $L$-parameter $\phi$ for $\U_n$ is tempered if $\phi$ is of the form 
\[
 \phi = \chi_{\kappa_1} \oplus \dots \oplus \chi_{\kappa_{n_0}} \oplus \xi_1 \oplus \dots \oplus \xi_d \oplus \check{\xi}_1 \oplus \dots \oplus \check{\xi}_d,
\]
where
\begin{itemize}
\item $\kappa_i \in \Z + \frac{n-1}{2}$;
\item $\xi_i$ is a unitary character of $\C^\times$ which is not conjugate-selfdual with sign $(-1)^{n-1}$;
\item $n_0 + 2d = n$.
\end{itemize}
For such a parameter $\phi$, we define a (limit of) discrete series $L$-parameter $\phi_0$ for $\U_{n_0}$ by 
\[
 \phi_0 = \chi_{\kappa_1} \oplus \dots \oplus \chi_{\kappa_{n_0}}.
\]
Then $\Pi_{\phi}(\U(p,q))$ consists of the parabolically induced representations
\[
 I(\xi_1, \dots, \xi_d, \pi_0)
\]
for all $\pi_0 \in \Pi_{\phi_0}(\U(p-d,q-d))$, which are irreducible by Lemma \ref{l:irred-temp}.
(When $d > \min\{ p, q \}$, we interpret $\Pi_{\phi_0}(\U(p-d,q-d))$ as the empty set.)
Moreover, via the natural identification $S_\phi = S_{\phi_0}$, the character of $S_\phi$ associated to $I(\xi_1, \dots, \xi_d, \pi_0)$ is equal to the character of $S_{\phi_0}$ associated to $\pi_0$.

\subsubsection{Some $A$-packets}
\label{sss:packets-real-coh}

We consider the $A$-packet $\Pi_{\phi'}(\U(r,s))$, where $\phi'$ is an $A$-parameter for $\U_m$ with $m=r+s$ of the form
\[
 \phi' = \chi_{\mu_1} \oplus \cdots \oplus \chi_{\mu_n} \oplus (\chi_{\mu_0} \boxtimes S_{m-n}),
\]
where 
\begin{itemize}
\item $\mu_i \in \Z + \frac{m-1}{2}$ for $i \ne 0$;
\item $\mu_0 \in \Z + \frac{n}{2}$;
\item $\mu_1 \ge \dots \ge \mu_{i_0-1} > \mu_0 \ge \mu_{i_0} \ge \dots \ge \mu_n$;
\item $n < m$.
\end{itemize}
For such a parameter $\phi'$, we identify $S_{\phi'}$ with a quotient of a free $\Z/2\Z$-module
\[
 \widetilde{S}_{\phi'} = (\Z / 2 \Z) e'_1 \oplus \dots \oplus (\Z / 2 \Z) e'_n \oplus (\Z / 2 \Z) e_0',
\]
where $e'_i$ corresponds to $\chi_{\mu_i}$ (resp.~$\chi_{\mu_0} \boxtimes S_{m-n}$) if $i \ne 0$ (resp.~$i=0$), in such a way that $\widehat{S}_{\phi'}$ consists of the characters $\eta'$ of $\widetilde{S}_{\phi'}$ satisfying
\[
 \eta'(e_i') = \eta'(e_j')
\]
for all $i,j$ such that $\mu_i = \mu_j$ with either $i, j \ne 0$ or $i \ne 0$, $j = 0$, $m-n=1$.

Let $\eta' \in \widehat{S}_{\phi'}$.
For $1 \le i \le n+1$, we define a pair of integers $(r_i, s_i)$ by 
\[
 (r_i, s_i) = 
 \begin{cases}
  (1,0) & \text{if $i<i_0$ and $\eta'(e_i') = (-1)^{i-1}$;} \\
  (0,1) & \text{if $i<i_0$ and $\eta'(e_i') = (-1)^i$;} \\
  (1,0) & \text{if $i>i_0$ and $\eta'(e_{i-1}') = (-1)^{i+m-n-2}$;} \\
  (0,1) & \text{if $i>i_0$ and $\eta'(e_{i-1}') = (-1)^{i+m-n-1}$}
 \end{cases}
\]
and 
\[
 (r_{i_0}, s_{i_0}) = (r - r_1 - \dots - r_{i_0-1} - r_{i_0+1} \dots - r_{n+1}, s - s_1 - \dots - s_{i_0-1} - s_{i_0+1} \dots - s_{n+1}).
\]
Note that $r_{i_0} + s_{i_0} = m-n$.
Then by \cite[Th\'eor\`eme 1.1]{mr19}, the representation $\sigma(\phi', \eta')$ of $\U(r,s)$ is nonzero only if
\[
 r_{i_0}, s_{i_0} \ge 0
\]
and 
\begin{equation}
\label{eq:eta'}
 \eta'(e_1' + \dots + e_n' + e_0') = (-1)^{\frac{1}{2}(r-s)(r-s-1)} 
\end{equation}
(see also \cite[(1-3)]{mr19} and Lemma \ref{l:eta'} below), in which case we have
\[
 \sigma(\phi', \eta') = \nA_\q(\lambda').
\]
Here $\lambda'$ is given by 
\[
 \lambda' = (\underbrace{\lambda'_1,\dots,\lambda'_1}_{r_1}, \dots, \underbrace{\lambda'_{n+1},\dots,\lambda'_{n+1}}_{r_{n+1}}, \underbrace{\lambda'_1,\dots,\lambda'_1}_{s_1}, \dots, \underbrace{\lambda'_{n+1},\dots,\lambda'_{n+1}}_{s_{n+1}})
\]
with 
\[
 \lambda'_i = 
 \begin{cases}
  \mu_i & \text{if $i < i_0$;} \\
  \mu_0 & \text{if $i = i_0$;} \\
  \mu_{i-1} & \text{if $i > i_0$}
 \end{cases}
\]
and $\q = \q(x')$ is associated to
\[
 x' = (\underbrace{x'_1,\dots,x'_1}_{r_1}, \dots, \underbrace{x'_{n+1},\dots,x'_{n+1}}_{r_{n+1}}, \underbrace{x'_1,\dots,x'_1}_{s_1}, \dots, \underbrace{x'_{n+1},\dots,x'_{n+1}}_{s_{n+1}})
\]
for any $x'_1, \dots, x'_{n+1} \in \R$ such that $x'_1 > \dots > x'_{n+1}$, so that $\nA_\q(\lambda')$ is weakly fair.
(Note that there is a typo in \cite[(4-2)]{mr19}: $(t_i+a_i-N)/2 - a_{<i}$ should be $(t_i+a_i-N)/2 + a_{<i}$.)
Moreover, if two representations $\sigma(\phi', \eta_1'), \sigma(\phi', \eta_2')$ with $\eta_1', \eta_2' \in \widehat{S}_{\phi'}$ are nonzero and isomorphic, then we have $\eta_1' = \eta_2'$. 

\begin{lem}
\label{l:eta'}
Let $\eta' \in \widehat{S}_{\phi'}$ and define $(r_i,s_i)$ as above.
Then $\eta'$ satisfies \eqref{eq:eta'} if and only if
\[
 \eta'(e_0') = (-1)^{r_{i_0}(i_0-1) + s_{i_0} i_0 + \frac{1}{2}(m-n)(m-n-1)}.
\]
\end{lem}

\begin{proof}
It suffices to show that 
\[
 \eta'(e_1' + \dots + e_n') \cdot (-1)^{r_{i_0}(i_0-1) + s_{i_0} i_0 + \frac{1}{2}(m-n)(m-n-1)} \cdot (-1)^{\frac{1}{2}(r-s)(r-s-1)} = 1.
\]
We may write $\eta'(e_1' + \dots + e_n') = (-1)^j$, where
\begin{align*}
 j & = \sum_{i=1}^{i_0-1} (i-1+s_i) + \sum_{i=i_0+1}^{n+1} (i+m-n-2+s_i) \\
 & = \frac{1}{2}(i_0-1)(i_0-2) + \frac{1}{2}(n-i_0+1)(2m-n+i_0-2) + s-s_{i_0} \\
 & = \frac{1}{2}n(2m-n-1) - (m-n)(i_0-1) + s-s_{i_0} \\
 & = \frac{1}{2}n(2m-n-1) - r_{i_0}(i_0-1) - s_{i_0} i_0 + s.
\end{align*}
Then we have
\begin{align*}
 & j + r_{i_0}(i_0-1) + s_{i_0} i_0 + \frac{1}{2}(m-n)(m-n-1) + \frac{1}{2}(r-s)(r-s-1) \\
 & = \frac{1}{2}n(2m-n-1) + s + \frac{1}{2}(m-n)(m-n-1) + \frac{1}{2}(r-s)(r-s-1) \\
 & = \frac{1}{2} m (m-1) + s + \frac{1}{2}(r-s)(r-s-1) \\
 & = \frac{1}{2} (r+s) (r+s-1) + s + \frac{1}{2}(r-s)(r-s-1) \\
 & = r(r-1) + s(s+1) \\
 & \equiv 0 \bmod 2.
\end{align*}
This implies the assertion.
\end{proof}

\subsection{The nonarchimedean case}

Suppose that $F$ is nonarchimedean.
Recall that given a positive integer $n$, there are precisely two $n$-dimensional $\varepsilon$-Hermitian spaces $V_n^+$ and $V_n^-$ over $E$ (up to isometry).
Consider a normalized parabolically induced representation
\[
 \Ind^{\U(V_n^+)}_P(\xi_1 \boxtimes \dots \boxtimes \xi_d \boxtimes \pi_0),
\]
where 
\begin{itemize}
\item $d$ is a nonnegative integer with $2d \le n$;
\item $P$ is a parabolic subgroup of $\U(V_n^+)$ with Levi component $(E^\times)^d \times \U(V_{n-2d}^+)$ defined as in \S \ref{ss:parab-ind};
\item $\xi_1, \dots, \xi_d$ are characters of $E^\times$;
\item $\pi_0$ is an irreducible tempered representation of $\U(V_{n-2d}^+)$.
\end{itemize}
If this representation is a standard module, we denote its unique irreducible quotient by
\[
 J(\xi_1, \dots, \xi_d, \pi_0).
\]

\subsubsection{Some $L$-packets}

We consider the $L$-packet $\Pi_\phi(\U(V_n^+))$, where $\phi$ is an $L$-parameter for $\U_n$ of the form
\[
 \phi = \chi_1 \oplus \dots \oplus \chi_n
\]
with (not necessarily distinct) conjugate-selfdual characters $\chi_1, \dots, \chi_n$ of $E^\times$ with sign $(-1)^{n-1}$.
Then $\pi(\phi,\mathbbm{1})$ is an irreducible tempered representation of $\U(V_n^+)$.
For more properties, we refer the reader to \cite[\S 2.5]{gi2}.

\subsubsection{Some $A$-packets}

We consider the $A$-packet $\Pi_{\phi'}(\U(V_m^+))$, where $\phi'$ is an $A$-parameter for $\U_m$ of the form
\[
 \phi' = \chi_1 \oplus \dots \oplus \chi_n \oplus (\chi_0 \boxtimes S_{m-n})
\]
with $n < m$ and (not necessarily distinct) conjugate-selfdual characters $\chi_1, \dots, \chi_n,\chi_0$ of $E^\times$ with sign
\[
\begin{cases}
 (-1)^{m-1} & \text{if $i \ne 0$;} \\
 (-1)^n & \text{if $i=0$.}
\end{cases} 
\]
Then by \cite[\S 4.1]{mr18}, the representation $\sigma(\phi', \eta')$ of $\U(V_m^+)$ with $\eta' \in \widehat{S}_{\phi'}$ is either zero or irreducible.
Moreover, if two representations $\sigma(\phi', \eta_1'), \sigma(\phi', \eta_2')$ with $\eta_1', \eta_2' \in \widehat{S}_{\phi'}$ are nonzero and isomorphic, then we have $\eta_1' = \eta_2'$.

\begin{lem}
\label{l:local-Apacket}
\begin{enumerate}
\item 
If $m \equiv n \bmod 2$, then we have
\[
 \sigma(\phi', \mathbbm{1}) = J(\chi_0 | \cdot |^{\frac{1}{2}(m-n-1)}, \chi_0 | \cdot |^{\frac{1}{2}(m-n-3)}, \dots, \chi_0 | \cdot |^{\frac{1}{2}}, \pi(\phi_0, \mathbbm{1}))
\]
with an $L$-parameter $\phi_0 = \chi_1 \oplus \dots \oplus \chi_n$ for $\U_n$.
\item
If $m \not \equiv n \bmod 2$, then we have
\[
 \sigma(\phi', \mathbbm{1}) = J(\chi_0 | \cdot |^{\frac{1}{2}(m-n-1)}, \chi_0 | \cdot |^{\frac{1}{2}(m-n-3)}, \dots, \chi_0 | \cdot |^1, \pi(\phi_1, \mathbbm{1}))
\]
with an $L$-parameter $\phi_1 = \chi_1 \oplus \dots \oplus \chi_n \oplus \chi_0$ for $\U_{n+1}$.
\end{enumerate}
\end{lem}

\begin{proof}
The assertion follows from \cite[Proposition 8.4.1]{mok} and the irreducibility of $\sigma(\phi', \mathbbm{1})$.
\end{proof}

\subsection{The split case}

We also need to consider the case when $F$ is nonarchimedean and $E = F \times F$.
Recall that given a positive integer $n$, there is a unique $n$-dimensional $\varepsilon$-Hermitian space $V_n^+ = \VV_n \otimes_F E$ over $E$ (up to isometry), where $\VV_n$ is an $n$-dimensional vector space over $F$.
Via the isomorphism $\U(V_n^+) \cong \GL(\VV_n)$ induced by the first projection, we may regard an $L$-parameter $\phi : L_F \rightarrow {}^L \U_n$ (resp.~an $A$-parameter $\phi : L_F \times \SL_2(\C) \rightarrow {}^L \U_n$) for $\U_n$ as an $n$-dimensional representation of $L_F$ (resp.~$L_F \times \SL_2(\C)$).
For any such a parameter $\phi$, the component group $S_\phi$ is always trivial.

Let $\phi$ and $\phi'$ be $L$- and $A$-parameters for $\U_n$ and $\U_m$, respectively, of the form
\begin{align*}
 \phi & = \chi_1 \oplus \dots \oplus \chi_n, \\
 \phi' & = \chi_1 \oplus \dots \oplus \chi_n \oplus (\chi_0 \boxtimes S_{m-n})
\end{align*}
with $n < m$ and (not necessarily distinct) unitary characters $\chi_1, \dots, \chi_n, \chi_0$ of $F^\times$.
We denote by $\pi(\phi, \mathbbm{1})$ and $\sigma(\phi', \mathbbm{1})$ the unique representations of $\U(V_n^+) \cong \GL(\VV_n)$ and $\U(V_m^+) \cong \GL(\VV_m)$ in the $L$- and $A$-packets $\Pi_\phi(\U(V_n^+))$ and $\Pi_{\phi'}(\U(V_m^+))$, respectively.
Then we have
\begin{align*}
 \pi(\phi, \mathbbm{1}) & = \Ind^{\GL(\VV_n)}_{\mathcal{B}}(\chi_1 \boxtimes \dots \boxtimes \chi_n), \\
 \sigma(\phi', \mathbbm{1}) & = \Ind^{\GL(\VV_m)}_{\mathcal{P}}(\chi_1 \boxtimes \dots \boxtimes \chi_n \boxtimes (\chi_0 \circ \det)), 
\end{align*}
where $\mathcal{B}$ is a Borel subgroup of $\GL(\VV_n)$ and $\mathcal{P}$ is a parabolic subgroup of $\GL(\VV_m)$ with Levi component $(F^\times)^n \times \GL_{m-n}(F)$.
Note that the parabolically induced representations on the right-hand side are irreducible by \cite[Theorem 4.2]{zel}.

\section{Nonvanishing of theta lifts}
\label{s:atobe}

In this section, we review a criterion for the nonvanishing of theta lifts due to Atobe \cite{atobe}.

\subsection{Some invariants}
\label{ss:atobe-def}

Let $W$ be an $n$-dimensional skew-Hermitian space over $\C$.
Fix $k_0 = -1$ or $0$.
We consider the theta lifting from $\U(W)$ to $\U(V)$, where $V$ varies over $m$-dimensional Hermitian spaces over $\C$ with
\[
 m \equiv n+k_0 \bmod 2.
\]
Fix an integer $m_0$ with $m_0 \equiv n+k_0 \bmod 2$ and take the character $\chi_V$ of $\C^\times$ given by 
\[
 \chi_V(z) = \bigg( \frac{z}{\sqrt{z \bar{z}}} \bigg)^{m_0}.
\]
Let $\pi$ be an irreducible tempered representation of $\U(W)$.
Following \cite[\S 4.1]{atobe}, we define some invariants of $\pi$ (which depend on $k_0$ and $\chi_V$) as follows.

Write $\pi = \pi(\phi, \eta)$, where $\phi$ is a tempered $L$-parameter for $\U_n$ and $\eta$ is a character of $S_\phi$.
We may write $\phi$ as
\[
 \phi = (m_1 \chi_{\kappa_1} \oplus \dots \oplus m_a \chi_{\kappa_a} \oplus n_1 \chi_{\mu_1} \oplus \dots \oplus n_b \chi_{\mu_b} \oplus \xi_1 \oplus \dots \oplus \xi_d \oplus \check{\xi}_1 \oplus \dots \oplus \check{\xi}_d) \otimes \chi_V,
\]
where 
\begin{itemize}
\item $\kappa_i, \mu_j \in \Z + \frac{k_0-1}{2}$;
\item $\kappa_i \ne \mu_j$ for all $i,j$;
\item $\kappa_1 > \dots > \kappa_a$ and $\mu_1 > \dots > \mu_b$;
\item $\xi_i$ is a unitary character of $\C^\times$ which is not conjugate-selfdual with sign $(-1)^{n-1}$;
\item $m_i$ and $n_j$ are odd and even positive integers, respectively;
\item $m_1 + \dots + m_a + n_1 + \dots + n_b + 2d = n$.
\end{itemize}
Then $S_\phi$ is a free $\Z/2\Z$-module of the form
\[
 S_\phi = (\Z/2\Z) \tilde{e}_{\kappa_1} \oplus \dots \oplus (\Z/2\Z) \tilde{e}_{\kappa_a} \oplus (\Z/2\Z) \tilde{e}_{\mu_1} \oplus \dots \oplus (\Z/2\Z) \tilde{e}_{\mu_b},
\]
where $\tilde{e}_{\kappa_i}$ and $\tilde{e}_{\mu_j}$ correspond to $\chi_{\kappa_i} \chi_V$ and $\chi_{\mu_j} \chi_V$, respectively.
Put
\[
 \epsilon_{\kappa_i} = \eta(\tilde{e}_{\kappa_i}), \quad
 \epsilon_{\mu_j} = \eta(\tilde{e}_{\mu_j}).
\]
\begin{enumerate}
\item
Let $k_\pi$ be the largest positive integer such that 
\begin{itemize}
\item $k_\pi \equiv k_0 \bmod 2$;
\item $\{ \frac{k_\pi-1}{2}, \frac{k_\pi-3}{2}, \dots, -\frac{k_\pi-1}{2} \} \subset \{ \kappa_1, \dots, \kappa_a \}$;
\item $\epsilon_{\frac{k_\pi+1}{2} - i} \ne \epsilon_{\frac{k_\pi-1}{2} - i}$ for all $1 \le i < k_\pi$.
\end{itemize}
If such an integer does not exist, we put $k_\pi = k_0$.
\item
Put
\begin{align*}
 r_\pi & = \# \left\{ 1 \le i \le a \, \left| \, |\kappa_i| \ge \frac{k_\pi+1}{2}, (-1)^{i-1} \epsilon_{\kappa_i} \kappa_i > 0 \right. \right\} + \frac{n-a}{2}, \\
 s_\pi & = \# \left\{ 1 \le i \le a \, \left| \, |\kappa_i| \ge \frac{k_\pi+1}{2}, (-1)^{i-1} \epsilon_{\kappa_i} \kappa_i < 0 \right. \right\} + \frac{n-a}{2}.
\end{align*}
\item 
Define a finite subset $\XX_\pi$ of $\frac{1}{2} \Z \times \{ \pm 1 \}$ by
\[
 \XX_\pi = \{ (\kappa_i, (-1)^{i-1} \epsilon_{\kappa_i}) \, | \, 1 \le i \le a \} \cup \{ (\mu_j, +1), (\mu_j, -1) \, | \, 1 \le j \le b, \, \epsilon_{\mu_j} \ne (-1)^{c_j} \},
\]
where $c_j = \# \{ 1 \le i \le a \, | \, \kappa_i > \mu_j \}$.
\item
Define a sequence
\[
 \XX_\pi = \XX_\pi^{(0)} \supset \XX_\pi^{(1)} \supset \dots \supset \XX_\pi^{(j)} \supset \cdots
\]
inductively as follows.
Write the image of $\XX_\pi^{(j)}$ under the projection $\frac{1}{2} \Z \times \{ \pm 1 \} \rightarrow \frac{1}{2} \Z$ as
\[
 \{ \nu_1, \nu_2, \dots \}
\]
with $\nu_1 > \nu_2 > \cdots$ and define a subset $\XX_\pi^{(j+1)}$ of $\XX_\pi^{(j)}$ by 
\[
 \XX_\pi^{(j+1)} = \XX_\pi^{(j)} \smallsetminus
 \bigg( \bigcup_i \{ (\nu_i, +1), (\nu_{i+1}, -1) \} \bigg), 
\]
where $i$ runs over indices such that 
\begin{itemize}
\item $(\nu_i, +1), (\nu_{i+1},-1) \in \XX_\pi^{(j)}$;
\item $\min \{ |\nu_i|, |\nu_{i+1}| \} \ge \frac{k_\pi+1}{2}$;
\item $\nu_i \nu_{i+1} \ge 0$.
\end{itemize}
\item 
Put
\[
 \XX_\pi^{(\infty)} = \XX_\pi^{(n)} = \XX_\pi^{(n+1)} = \cdots.
\]
\item 
For any integer $x$, we define subsets $\CC_\pi^\pm(x)$ of $\XX_\pi^{(\infty)}$ by
\begin{align*}
 \CC_\pi^+(x) & = \left\{ (\nu,+1) \in \XX_\pi^{(\infty)} \, \left| \, 0 \le \frac{k_\pi-1}{2} + \nu < x \right. \right\}, \\ 
 \CC_\pi^-(x) & = \left\{ (\nu,-1) \in \XX_\pi^{(\infty)} \, \left| \, 0 \le \frac{k_\pi-1}{2} - \nu < x \right. \right\}.
\end{align*}
\end{enumerate}

\subsection{A result of Atobe}

We have the following criterion for the nonvanishing of $\theta_{r,s}(\pi)$ due to Atobe \cite[Theorem 4.2]{atobe}, which relies on the Gan--Gross--Prasad conjecture \cite{ggp1} for real unitary groups but is now unconditional thanks to a recent result of Xue \cite{xue}.
Note that the choice of $\psi$ in \cite[\S 3.3]{atobe} is not made explicit, but it agrees with our choice (see \cite[Lemma 1.4.5]{paul1} and Lemma \ref{l:K-type-corresp} below).
Note also that \cite[Theorem 4.2(1)]{atobe} is not correct as stated, but the argument in \cite[\S 5]{atobe} in fact proves:

\begin{thm}[Atobe]
\label{t:atobe}
Let $\pi$ be an irreducible tempered representation of $\U(W)$.
Let $l,t$ be integers with $t \ge 1$.
\begin{enumerate}
\item 
Assume that $k_\pi = -1$.
(In particular, we have either $(0, +1), (0, -1) \notin \XX_\pi$ or $(0, +1), (0, -1) \in \XX_\pi$.)
Then
\begin{itemize}
\item 
$\theta_{r_\pi + l + 2t + 1, s_\pi + l}(\pi)$ is nonzero if and only if 
\[
\left\{ 
\begin{aligned}
 l & \ge 0, \, &
 \# \CC_\pi^+(l+t) & \le l, \, &
 \# \CC_\pi^-(l+t) & \le l \, &
 \text{if $(0, +1), (0, -1) \notin \XX_\pi$;} \\
 l & \ge 1, \, &
 \# \CC_\pi^+(l+t) & \le l-1, \, &
 \# \CC_\pi^-(l+t) & \le l-1 \, &
 \text{if $(0, +1), (0, -1) \in \XX_\pi$;}
\end{aligned}
\right.
\]
\item 
$\theta_{r_\pi + l + 1, s_\pi + l}(\pi)$ is nonzero if and only if
\[
\begin{cases}
 l \ge 0 & \text{if $0 \notin \{ \mu_1, \dots, \mu_b \}$;} \\
 l \ge -1 & \text{if $0 \in \{ \mu_1, \dots, \mu_b \}$ and $(0, +1), (0, -1) \notin \XX_\pi$;} \\
 l \ge 1 & \text{if $0 \in \{ \mu_1, \dots, \mu_b \}$ and $(0,+1), (0,-1) \in \XX_\pi$.}
\end{cases}
\]
\end{itemize}
\item 
Assume that $k_\pi \ge 0$.
Then
\begin{itemize}
\item 
$\theta_{r_\pi + l + 2t, s_\pi + l}(\pi)$ is nonzero if and only if
\[
 l \ge k_\pi, \quad 
 \# \CC_\pi^+(l+t) \le l, \quad
 \# \CC_\pi^-(l+t) \le l;
\]
\item 
$\theta_{r_\pi + l, s_\pi + l}(\pi)$ is nonzero if and only if 
\[
\begin{cases}
 l \ge -1 & \text{if the conditions \eqref{atobe-1}, \eqref{atobe-2}, \eqref{atobe-3} below hold;} \\
 l \ge 0 & \text{otherwise,}
\end{cases}
\]
where 
\begin{enumerate}
\item
\label{atobe-1}
$\{ \frac{k_\pi+1}{2}, -\frac{k_\pi+1}{2} \} \subset \{ \kappa_1, \dots, \kappa_a, \mu_1, \dots, \mu_b \}$;
\item 
\label{atobe-2}
$\{ \frac{k_\pi+1}{2}, -\frac{k_\pi+1}{2} \} \cap \{ \mu_1, \dots, \mu_b \} \ne \varnothing$;
\item
\label{atobe-3}
$\epsilon_{\frac{k_\pi+1}{2} - i} \ne \epsilon_{\frac{k_\pi-1}{2} - i}$ for all $0 \le i \le k_\pi$.
\end{enumerate}
\end{itemize}
\end{enumerate}
\end{thm}

\begin{rem}
\label{r:atobe}
To determine the nonvanishing of $\theta_{r,s}(\pi)$ with $r+s=m$, we may assume that 
\[
 r-r_\pi \ge s-s_\pi
\]
by replacing $(r,s)$ by $(s,r)$ and $\pi$ by $\tilde{\pi} = \bar{\pi} \otimes (\chi_V \circ \det)$ if necessary.
Indeed, we have $\theta_{r,s}(\pi) \ne 0$ if and only if $\theta_{s,r}(\tilde{\pi}) \ne 0$ by \cite[Proposition 3.9]{atobe}, while we have
\[
 k_{\tilde{\pi}} = k_\pi, \quad 
 (r_{\tilde{\pi}}, s_{\tilde{\pi}}) = (s_\pi, r_\pi)
\]
by \cite[Lemma 4.4]{atobe}.
If $r-r_\pi \ge s-s_\pi$, then since $m - n \equiv k_\pi \bmod 2$ and
\[
 n = 
 \begin{cases}
  r_\pi + s_\pi & \text{if $k_\pi = -1$;} \\
  r_\pi + s_\pi + k_\pi  & \text{if $k_\pi \ge 0$,}
 \end{cases}
\]
we have
\[
 (r-r_\pi) - (s-s_\pi) = 
 \begin{cases}
  2t+1 & \text{if $k_\pi = -1$;} \\
  2t & \text{if $k_\pi \ge 0$}
 \end{cases}
\]
for some nonnegative integer $t$.
Thus Theorem \ref{t:atobe} completely determines the nonvanishing of $\theta_{r,s}(\pi)$.
\end{rem}

\subsection{Some corollaries}

In this subsection, we state some corollaries of Theorem \ref{t:atobe} which will be used later.
We consider the theta lifting from $\U(p,q)$ to $\U(r,s)$ with $p+q=n$ and $r+s=m$.

\begin{cor}
\label{c:jacquet}
Assume that $p,q > 0$.
Let $\pi$ be an irreducible tempered representation of $\U(p,q)$ such that $\pi \subset I(\chi, \pi_0)$, where
\begin{itemize}
\item $\chi$ is a unitary character of $\C^\times$;
\item $\pi_0$ is an irreducible tempered representation of $\U(p-1,q-1)$.
\end{itemize}
Assume further that $\theta_{r,s}(\pi) \ne 0$ and that either $m>n$ or $\chi$ satisfies one of the following conditions (with the notation of \S \ref{ss:atobe-def}):
\begin{itemize}
\item $\chi = \chi_{\kappa_i}$ for some $i$ such that $m_i \ge 3$;
\item $\chi = \chi_{\mu_j}$ for some $j$ such that $n_j \ge 4$;
\item $\chi$ is not conjugate-selfdual with sign $(-1)^{n-1}$.
\end{itemize}
Then we have $r,s > 0$ and $\theta_{r-1,s-1}(\pi_0) \ne 0$.
\end{cor}

\begin{proof}
If $m>n$, then the assertion was proved in \cite[Corollary 4.5]{atobe}.
Thus we assume that $\chi$ satisfies one of the conditions above.
We may assume that $r-r_\pi \ge s-s_\pi$.
By Theorem \ref{t:atobe}, we have
\[
\left\{ 
\begin{aligned}
 r & \ge r_\pi -1, \, & s & \ge s_\pi -1 \, & \text{if $b > 0$;} \\
 r & \ge r_\pi, \, & s & \ge s_\pi \, & \text{if $b = 0$.}
\end{aligned}
\right.
\]
On the other hand, we have
\[
 r_\pi, s_\pi \ge \sum_{i=1}^a \frac{m_i - 1}{2} + \sum_{j=1}^b \frac{n_j}{2} + d.
\]
From this and the assumption on $\chi$, we can deduce that $r, s > 0$.
Moreover, we have
\[
 r_{\pi_0} = r_{\pi} -1, \quad
 s_{\pi_0} = s_{\pi} -1, \quad 
 \XX_{\pi_0} = \XX_\pi,  \quad
 \CC^\pm_{\pi_0}(x) = \CC^\pm_\pi(x).
\]
Hence by Theorem \ref{t:atobe}, we have $\theta_{r-1,s-1}(\pi_0) \ne 0$.
This completes the proof.
\end{proof}

\begin{cor}
\label{c:local-theta-xyzw}
Let $\pi$ be a (limit of) discrete series representation of $\U(p,q)$ and write $\pi = \nA_\b(\lambda)$ as in \eqref{eq:lds} with 
\[
 \lambda = (\alpha_1, \dots, \alpha_{p^+}, \beta_1, \dots, \beta_{p^-}, \gamma_1, \dots, \gamma_{q^+}, \delta_1, \dots, \delta_{q^-}) + \bigg( \frac{m_0}{2}, \dots, \frac{m_0}{2} \bigg),
\]
where
\begin{itemize}
\item $\alpha_i, \gamma_j > 0$ and $\beta_i, \delta_j \le 0$;
\item $p^++p^- = p$ and $q^++q^- = q$.
\end{itemize}
Assume that $m \ge n$ and $\theta_{r,s}(\pi) \ne 0$.
Then we have
\[
 p^++q^- \le r, \quad
 p^-+q^+ \le s.
\]
\end{cor}

\begin{proof}
We may assume that $r - r_\pi \ge s - s_\pi$, so that 
\[
 (r,s) = 
 \begin{cases}
  (r_\pi + l + 2t + 1, s_\pi + l) & \text{if $k_\pi = -1$;} \\
  (r_\pi + l + 2t, s_\pi + l) & \text{if $k_\pi \ge 0$}
 \end{cases}
\]
for some integers $l,t$ with $t \ge 0$.
Since $m \ge n$ and $\theta_{r,s}(\pi) \ne 0$, it follows from Theorem \ref{t:atobe} that 
\[
\begin{cases}
 l \ge 0 & \text{if $k_\pi = -1$;} \\
 l \ge \frac{k_\pi}{2} & \text{if $k_\pi \ge 0$.}
\end{cases}
\]
On the other hand, we have
\[
 (r_\pi, s_\pi) = 
\begin{cases}
 (p^++q^-, p^-+q^+) & \text{if $k_\pi = -1$}; \\
 (p^++q^--\frac{k_\pi}{2}, p^-+q^+-\frac{k_\pi}{2}) & \text{if $k_\pi \ge 0$ and $k_\pi$ is even}; \\
 (p^++q^--\frac{k_\pi \pm 1}{2}, p^-+q^+-\frac{k_\pi \mp 1}{2}) & \text{if $k_\pi \ge 0$ and $k_\pi$ is odd.}
\end{cases}
\]
This implies the assertion.
\end{proof}

\begin{cor}
\label{c:going-down}
Let $\pi$ be a (limit of) discrete series representation of $\U(p,q)$.
Assume that $m \le n-2$ and $\theta_{r,s}(\pi) \ne 0$.
Put $k = n-m$.
Then one of the following holds:
\begin{itemize}
\item 
$k_\pi \ge 2$, $2 \le k \le k_\pi$, and 
\[
 (r,s) = (r_\pi + \tfrac{k_\pi - k}{2}, s_\pi + \tfrac{k_\pi - k}{2});
\]
\item 
$k_\pi \ge 0$, $k = k_\pi + 2$, 
\[
 (r,s) = (r_\pi -1, s_\pi -1), 
\]
and the conditions \eqref{atobe-1}, \eqref{atobe-2}, \eqref{atobe-3} in Theorem \ref{t:atobe} hold.
\end{itemize}
\end{cor}

\begin{proof}
We may assume that $r - r_\pi \ge s - s_\pi$, so that
\[
 (r,s) = 
 \begin{cases}
  (r_\pi + l + 2t + 1, s_\pi + l) & \text{if $k_\pi = -1$;} \\
  (r_\pi + l + 2t, s_\pi + l) & \text{if $k_\pi \ge 0$}
 \end{cases}
\]
for some integers $l,t$ with $t \ge 0$.
Since $m \le n-2$ and
\[
 n = 
 \begin{cases}
  r_\pi + s_\pi & \text{if $k_\pi = -1$;} \\
  r_\pi + s_\pi + k_\pi  & \text{if $k_\pi \ge 0$,}
 \end{cases}
\]
we have $2l+2t \le k_\pi -2$.
Hence it follows from Theorem \ref{t:atobe} that one of the following holds:
\begin{itemize}
\item $k_\pi \ge 2$, $0 \le l \le \frac{k_\pi}{2} -1$, and $t = 0$;
\item $k_\pi \ge 0$, $l=-1$, $t=0$, and the conditions \eqref{atobe-1}, \eqref{atobe-2}, \eqref{atobe-3} in Theorem \ref{t:atobe} hold.
\end{itemize}
This implies the assertion.
\end{proof}

\begin{cor}
\label{c:going-up}
Let $\pi$ be a (limit of) discrete series representation of $\U(p,q)$ and write $\pi = \nA_\b(\lambda)$ as in \eqref{eq:lds}.
Assume that $m \ge n+2$.
Put $k = m-n$.
Assume further that $\lambda$ is of the form
\begin{align*}
 \lambda & = \bigg( \alpha_1, \dots, \alpha_{p^+},
 \underbrace{\frac{k-1}{2}, \dots, \frac{k-1}{2}}_{p'}, 
 \underbrace{-\frac{k-1}{2}, \dots, -\frac{k-1}{2}}_{p''},
 \beta_1, \dots, \beta_{p^-}, \\
 & \phantom{{} = \bigg(} \gamma_1, \dots, \gamma_{q^+}, 
 \underbrace{\frac{k-1}{2}, \dots, \frac{k-1}{2}}_{q'}, 
 \underbrace{-\frac{k-1}{2}, \dots, -\frac{k-1}{2}}_{q''},
 \delta_1, \dots, \delta_{q^-} \bigg)
 + \bigg( \frac{m_0}{2}, \dots, \frac{m_0}{2} \bigg),
\end{align*}
where 
\begin{itemize}
\item $\alpha_i, \gamma_j > \frac{k-1}{2}$ and $\beta_i, \delta_j < -\frac{k-1}{2}$;
\item $p',p'',q',q'' \ge 0$;
\item $|p'-q'|, |p''-q''| \le 1$;
\item $p^++p^-+p'+p'' = p$ and $q^++q^-+q'+q'' = q$,
\end{itemize}
and that $\b = \q(x)$ is associated to
\begin{align*}
 x & = (x^+_1, \dots, x^+_{p^+}, x'_1, \dots, x'_{p'}, x''_1, \dots, x''_{p''}, x^-_1, \dots, x^-_{p^-}, \\
 & \phantom{{} = (} y^+_1, \dots, y^+_{q^+}, y'_1, \dots, y'_{q'}, y''_1, \dots, y''_{q''}, y^-_1, \dots, y^-_{q^-})
\end{align*}
such that either the conditions
\begin{enumerate}
\item
\label{going-up-1}
$p'-q'=0$ and 
\[
 x'_1 > y'_1 > \dots > x'_{p'} > y'_{q'},
\]
or $p'-q'=-1$ and 
\[
 y'_1 > x'_1 > \dots > y'_{p'} > x'_{p'} > y'_{q'};
\]
\item
\label{going-up-2}
$p'' - q'' = 0$ and 
\[
 x''_1 > y''_1 > \dots > x''_{p''} > y''_{q''},
\]
or $p'' - q'' = 1$ and
\[
 x''_1 > y''_1 > \dots > x''_{q''} > y''_{q''} > x''_{p''},
\]
\end{enumerate}
or the conditions
\begin{enumerate}[resume]
\item 
\label{going-up-3}
$p'-q'=0$ and 
\[
 y'_1 > x'_1 > \dots > y'_{q'} > x'_{p'},
\]
or $p'-q'=1$ and 
\[
 x'_{1} > y'_{1}> \dots > x'_{q'} > y'_{q'} > x'_{p'};
\]
\item
\label{going-up-4}
$p'' - q'' = 0$ and 
\[
 y''_1 > x''_1 > \dots > y''_{q''} > x''_{p''},
\]
or $p'' - q'' = -1$ and
\[
 y''_1 > x''_1 > \dots > y''_{p''} > x''_{p''} > y''_{q''}
\]
\end{enumerate}
hold.
Then we have
\[
 \theta_{r,s}(\pi) \ne 0,
\]
where
\[
 (r,s) = 
 \begin{cases}
  (p^+ + p' + q^- + q'' + k, p^- + p'' + q^+ + q') & \text{if the conditions \eqref{going-up-1}, \eqref{going-up-2} hold;} \\
  (p^+ + p' + q^- + q'', p^- + p'' + q^+ + q' + k) & \text{if the conditions \eqref{going-up-3}, \eqref{going-up-4} hold.}
 \end{cases}
\]
\end{cor}

\begin{proof}
Put
\[
 m' = p'+q', \quad
 m'' = p''+q'', \quad
 t = \frac{k+k_0}{2}, 
\]
where $k_0 = -1$ or $0$ is such that $k_0 \equiv k \bmod 2$.
Write $\pi = \pi(\phi, \eta)$, where $\phi$ is a (limit of) discrete series $L$-parameter for $\U_n$ and $\eta$ is a character of $S_\phi$.
Then we have
\[
 \eta(\tilde{e}_{\frac{k-1}{2}}) = \eta(\tilde{e}_{-\frac{k-1}{2}}) =  \epsilon_0 \cdot (-1)^{p^+ + q^+ +m'}
\]
with the notation of \S \ref{ss:atobe-def}, where
\[
 \epsilon_0 =
 \begin{cases}
  +1 & \text{if the conditions \eqref{going-up-1}, \eqref{going-up-2} hold;} \\
  -1 & \text{if the conditions \eqref{going-up-3}, \eqref{going-up-4} hold.}
 \end{cases}
\]
(When $m'=0$ or $m''=0$, we ignore the corresponding identity.)
Replacing $\pi$ by $\bar{\pi} \otimes (\chi_V \circ \det)$ if necessary, we may assume that $\epsilon_0 = +1$.
Then we have $k_\pi = k_0$ and
\[
 (r_\pi, s_\pi) = (p^+ + p' + q^- + q'', p^- + p'' + q^++ q').
\]
Moreover, we have the following.
\begin{itemize}
\item If $k_0 = -1$, then we have $(0, +1), (0, -1) \notin \XX_{\pi}$.
\item If $m'$ is even, then we have $(\frac{k-1}{2}, +1), (\frac{k-1}{2}, -1) \notin \XX_{\pi}$, so that $\CC^+_\pi(t) = \varnothing$.
\item If $m'$ is odd, then we have $(\frac{k-1}{2}, -1) \in \XX_{\pi}$ but $(\frac{k-1}{2}, +1) \notin \XX_{\pi}$, so that $\CC^+_\pi(t) = \varnothing$.
\item If $m''$ is even, then we have $(-\frac{k-1}{2}, +1), (-\frac{k-1}{2}, -1) \notin \XX_{\pi}$, so that $\CC^-_\pi(t) = \varnothing$.
\item If $m''$ is odd, then we have $(-\frac{k-1}{2}, +1) \in \XX_{\pi}$ but $(-\frac{k-1}{2}, -1) \notin \XX_{\pi}$, so that $\CC^-_\pi(t) = \varnothing$.
\end{itemize}
Hence by Theorem \ref{t:atobe}, we have
\[
 \theta_{r_\pi + k, s_\pi}(\pi) \ne 0.
\]
This completes the proof.
\end{proof}

\section{Nonvanishing of integrals of matrix coefficients}

In this section, we show that the nonvanishing of theta lifts is equivalent to that of the associated integrals of matrix coefficients in some cases, which is a crucial step in the proof of Theorem \ref{t:main-lds}.

\subsection{A key proposition}
\label{ss:key}

Let $V$ be an $m$-dimensional Hermitian space over $\C$ and $W$ an $n$-dimensional skew-Hermitian space over $\C$.
Recall the associated symplectic space $V \otimes_\C W$ over $\R$ and fix a maximal isotropic subspace $\X$ of $V \otimes_\C W$.
We take the datum $(\chi_V, \chi_W, \psi)$ given in \S \ref{ss:main-theorem-setup} and realize the (smooth) Weil representation $\omega = \omega_{V,W,\chi_V,\chi_W,\psi}$ of $\U(V) \times \U(W)$ on the space $\SS(\X)$ of Schwartz functions on $\X$.
Here $\SS(\X)$ is endowed with the usual topology which makes it into a Fr\'echet space.
Let $(\cdot, \cdot)$ be the invariant Hermitian inner product on $\SS(\X)$ given by
\[
 (\varphi_1, \varphi_2) = \int_{\X} \varphi_1(x) \overline{\varphi_2(x)} \, dx.
\]
Let $\pi$ be an irreducible tempered representation of $\U(W)$.
Here we work with a smooth representation of moderate growth, so that the space $\VV$ of $\pi$ is a Fr\'echet space.
Let $(\cdot, \cdot)$ be an invariant Hermitian inner product on $\VV$.
If $m \ge n$, then we have a separately continuous map
\[
 \ZZ_{V,W,\chi_V,\chi_W,\psi}(\pi) : \mathcal{S}(\X) \times \mathcal{S}(\X) \times \VV \times \VV \rightarrow \C
\]
given by 
\[
 (\varphi_1, \varphi_2, v_1, v_2) \mapsto \int_{\U(W)} (\omega(g) \varphi_1, \varphi_2) \overline{(\pi(g) v_1, v_2)} \, dg,
\]
where the integral above is absolutely convergent (see \S \ref{ss:matcoeff-weil} below).

\begin{prop}
\label{p:key}
Let $\pi$ be an irreducible tempered representation of $\U(W)$.
Assume that $m \ge n$.
Then $\ZZ_{V,W,\chi_V,\chi_W,\psi}(\pi)$ is nonzero if and only if $\theta_{V,W,\chi_V,\chi_W,\psi}(\pi)$ is nonzero.
\end{prop}

\begin{rem}
Since $\ZZ_{V,W,\chi_V,\chi_W,\psi}(\pi)$ is separately continuous, we may replace $\ZZ_{V,W,\chi_V,\chi_W,\psi}(\pi)$ in the proposition above by its restriction to the dense subspace
\[
 S(\X) \times S(\X) \times \VV_K \times \VV_K.
\]
Here $S(\X)$ is the subspace of $\SS(\X)$ consisting of functions which correspond to polynomials in the Fock model and $\VV_K$ is the space of $K$-finite vectors in $\VV$, where $K$ is a maximal compact subgroup of $\U(W)$.
Also, we mean $\theta_{V,W,\chi_V,\chi_W,\psi}(\pi)$ by either the smooth version of the theta lift of $\pi$ or the algebraic version of the theta lift of the Harish-Chandra module associated to $\pi$.
In fact, they coincide by \cite{bao-sun}.
\end{rem}

\begin{rem}
\label{r:Z}
As explained in Remark \ref{r:atobe}, we have $\theta_{V,W,\chi_V,\chi_W,\psi}(\pi) \ne 0$ if and only if $\theta_{-V,W,\chi_V,\chi_W,\psi}(\tilde{\pi}) \ne 0$, where we write $-V$ for the space $V$ equipped with the Hermitian form $- \langle \cdot, \cdot \rangle_V$ and put $\tilde{\pi} = \bar{\pi} \otimes (\chi_V \circ \det)$.
This is an immediate consequence of the fact that
\[
  \overline{\omega_{V,W,\chi_V,\chi_W,\psi}}
  = \omega_{-V,W,\chi_V,\chi_W,\psi} \otimes ((\chi_W^{-1} \circ \det) \boxtimes (\chi_V^{-1} \circ \det))
\]
as representations of $\U(V) \times \U(W)$, which in fact induces a natural identification
\[
 \overline{\ZZ_{V,W,\chi_V,\chi_W,\psi}(\pi)} = \ZZ_{-V,W,\chi_V,\chi_W,\psi}(\tilde{\pi}).
\]
In particular, we also have $\ZZ_{V,W,\chi_V,\chi_W,\psi}(\pi) \ne 0$ if and only if $\ZZ_{-V,W,\chi_V,\chi_W,\psi}(\tilde{\pi}) \ne 0$.
\end{rem}

The rest of this section is devoted to the proof of Proposition \ref{p:key}.

\subsection{Harish-Chandra Schwartz spaces}

Put $G = \U(W)$.
Let $\g$ be the complexified Lie algebra of $G$ and $U(\g)$ the universal enveloping algebra of $\g$.
Let $\Xi = \Xi_G$ and $\sigma = \sigma_G$ be the spherical functions on $G$ as in \cite[p.~329]{va} and \cite[p.~320]{va}, respectively.
For $X, Y \in U(\g)$, $r \in \R$, and a smooth function $f$ on $G$, put
\[
 p_{X,Y,r}(f) = \sup_{g \in G} |(L(X)R(Y)f)(g)| \Xi(g)^{-1} (1 + \sigma(g))^r,
\]
where $L$ and $R$ are the left and right translations, respectively.
We denote by $\CC(G)$ the Harish-Chandra Schwartz space, which is defined as the space of smooth functions $f$ on $G$ such that
\[
 p_{X,Y,r}(f) < \infty 
\]
for all $X, Y \in U(\g)$ and $r > 0$.
We endow $\CC(G)$ with the topology given by seminorms $p_{X,Y,r}$ for all $X, Y \in U(\g)$ and $r>0$, which makes it into a Fr\'echet space.
For $f_1, f_2 \in \CC(G)$, we may define their convolution $f_1 * f_2$ by 
\[
 (f_1 * f_2)(g) = \int_G f_1(h) f_2(h^{-1} g) \, dh, 
\]
where the integral above is absolutely convergent.
Then $f_1 * f_2$ belongs to $\CC(G)$ and the associated map
\[
 \CC(G) \times \CC(G) \rightarrow \CC(G)
\]
is continuous (see \cite[p.~357, Theorem 18]{va}).

Let $\pi$ be an irreducible tempered representation of $G$ on a Fr\'echet space $\VV$.
Namely, for any $v_1, v_2 \in \VV$, there exists a constant $C$ such that
\[
 |(\pi(g) v_1, v_2)| \le C \Xi(g)
\]
for all $g \in G$.
We may extend the action of $G$ on $\VV$ to an action of $\CC(G)$ by
\begin{equation}
\label{eq:C(G)-action}
 (\pi(f) v_1, v_2) = \int_G f(g) (\pi(g) v_1, v_2) \, dg  
\end{equation}
for $f \in \CC(G)$ and $v_1, v_2 \in \VV$.

Assume that $\pi$ is a discrete series representation.
Then the function $g \mapsto (\pi(g) v_1, v_2)$ belongs to $\CC(G)$ for all $v_1, v_2 \in \VV$.
Let $\AA(\pi)$ be the closure in $\CC(G)$ of the subspace spanned by these functions, on which the left and right translations define an irreducible representation of $G \times G$ (see \cite[p.~468, Theorem 11 and p.~469, Theorem 13]{va}).
For $f_1 \in \CC(G)$ and $f_2 \in \AA(\pi)$, we have $f_1 * f_2 \in \AA(\pi)$ by \eqref{eq:C(G)-action}.

\subsection{Matrix coefficients of Weil representations}
\label{ss:matcoeff-weil}

Assume that $m \ge n$.
Recall the Weil representation $\omega$ of $G$ on $\SS(\X)$.

\begin{lem}
\label{l:matcoeff-weil}
For $\varphi_1, \varphi_2 \in \SS(\X)$, the function $g \mapsto (\omega(g) \varphi_1, \varphi_2)$ belongs to $\CC(G)$.
\end{lem}

\begin{proof}
The assertion follows from \cite[(25)]{li89} and \cite[p.~330, Proposition 17]{va}.
\end{proof}

Let $\pi$ be an irreducible tempered representation of $G$ on $\VV$.
Then the lemma above implies that the integral defining $\ZZ_{V,W,\chi_V,\chi_W,\psi}(\pi)$ is absolutely convergent.
Moreover, as in the proof of \cite[Lemma 6.2]{li89}, we can deduce from this and the bounds for matrix coefficients due to Sun \cite{sun} that $\ZZ_{V,W,\chi_V,\chi_W,\psi}(\pi)$ is separately continuous.

\begin{lem}
\label{l:proj-pi}
Assume that $\pi$ is a discrete series representation and that $\ZZ_{V,W,\chi_V,\chi_W,\psi}(\pi)$ is nonzero.
Then the functions 
\[
 g \mapsto \int_G (\omega(gh) \varphi_1, \varphi_2) \overline{(\pi(h) v_1, v_2)} \, dh 
\]
for all $\varphi_1, \varphi_2 \in \SS(\X)$ and $v_1, v_2 \in \VV$ span a dense subspace of $\AA(\pi)$.
\end{lem}

\begin{proof}
We may write the integral above as $(f_1 * f_2)(g)$, where
\[
 f_1(g) = (\omega(g) \varphi_1, \varphi_2), \quad
 f_2(g) = (\pi(g) v_2, v_1).
\]
Since $f_1 \in \CC(G)$ by Lemma \ref{l:matcoeff-weil} and $f_2 \in \AA(\pi)$, we have $f_1 * f_2 \in \AA(\pi)$.
Hence these functions span a $G \times G$-invariant subspace of $\AA(\pi)$.
Since $\ZZ_{V,W,\chi_V,\chi_W,\psi}(\pi) \ne 0$, this subspace is nonzero, so that its closure in $\CC(G)$ agrees with $\AA(\pi)$ by the irreducibility of $\AA(\pi)$.
This completes the proof.
\end{proof}

\subsection{Inductive step}

The only if part of Proposition \ref{p:key} is obvious.
For the if part, we proceed by induction on $\dim V - \dim W$ as in the proof of \cite[Proposition 5.4]{atobe}, where $V$ is fixed but $W$ and $\pi$ vary.
In particular, we use a seesaw diagram
\[
\begin{tikzcd}
 \U(W') \arrow[rd,dash] & \U(V) \times \U(V) \arrow[ld,dash] \\
 \U(W) \times \U(W^\perp) \arrow[u,dash] & \U(V) \arrow[u,dash]
\end{tikzcd},
\]
where $V$ is an $m$-dimensional Hermitian space over $\C$, $W'$ is an $(n+1)$-dimensional skew-Hermitian space over $\C$, $W$ is an $n$-dimensional skew-Hermitian subspace of $W'$, and $W^\perp$ is the orthogonal complement of $W$ in $W'$.
We take a datum $(\chi_V, \chi_{W'}, \psi)$, where $\chi_V$ and $\psi$ are as above, and $\chi_{W'}$ is a characters of $\C^\times$ given by
\[
 \chi_{W'}(z) = \left( \frac{z}{\sqrt{z \bar{z}}} \right)^{n'_0}
\]
for some integer $n'_0$ such that $n'_0 \equiv n+1 \bmod 2$.

\begin{lem}
\label{l:seesaw}
Let $\pi$ be an irreducible tempered representation of $\U(W)$.
Assume that $m > n$ and that there exists a discrete series representation $\pi'$ of $\U(W')$ such that 
\begin{enumerate}
\item 
\label{seesaw-1}
$\ZZ_{V,W',\chi_V,\chi_{W'},\psi}(\pi') \ne 0$;
\item 
\label{seesaw-2}
$\Hom_{\U(W)}(\pi', \pi) \ne 0$.
\end{enumerate}
Then we have $\ZZ_{V,W,\chi_V,\chi_W,\psi}(\pi) \ne 0$.
\end{lem}

\begin{proof}
Put $G = \U(W)$ and $G' = \U(W')$.
Let $\VV$ and $\VV'$ be the spaces of $\pi$ and $\pi'$, respectively.
Define a continuous map 
\[
 \LL(\pi', \pi) : \VV' \times \VV' \times \VV \times \VV \rightarrow \C
\]
by 
\[
 (v_1', v_2', v_1, v_2) \mapsto \int_G (\pi'(g) v_1',v_2') \overline{(\pi(g) v_1,v_2)} \, dg,
\]
where the integral above is absolutely convergent (see \cite[Lemma 6.5.1(i)]{bp}).
By \eqref{seesaw-2} and a result of Beuzart-Plessis \cite[Theorem 7.2.1]{bp}, we have $\LL(\pi', \pi) \ne 0$.
Hence by \eqref{seesaw-1} and Lemma \ref{l:proj-pi}, we have
\[
 \int_G \bigg( \int_{G'} (\omega'(g g') \varphi_1', \varphi_2') \overline{(\pi'(g')v_1', v_2')} \, dg' \bigg) \overline{(\pi(g) v_1, v_2)} \, dg \ne 0\]
for some $\varphi_1', \varphi_2' \in \SS(\X')$, $v_1', v_2' \in \VV'$, and $v_1, v_2 \in \VV$.
Here $\omega' = \omega_{V,W',\chi_V, \chi_{W'}, \psi}$ is the Weil representation of $G'$ on $\SS(\X')$, where $\X'$ is a maximal isotropic subspace of $V \otimes_\C W'$.
If we formally interchange the order of integration, then we have
\begin{equation}
\label{eq:nonvanish-2} 
 \int_G (\omega'(g g') \varphi_1', \varphi_2') \overline{(\pi(g) v_1, v_2)} \, dg \ne 0
\end{equation}
for some $g' \in G'$.

To justify the manipulation, we show that the double integral
\[
 \int_G \int_{G'} (\omega'(g g') \varphi_1', \varphi_2') \overline{(\pi'(g')v_1', v_2') (\pi(g) v_1, v_2)} \, dg' \, dg
\]
is absolutely convergent.
By Lemma \ref{l:matcoeff-weil}, the integral above is bounded by
\[
 C \int_G \int_{G'} \Xi_{G'}(gg') (1 + \sigma_{G'}(gg'))^{-r} \Xi_{G'}(g') \Xi_G(g) \, dg' \, dg 
\]
for some $C>0$ and $r \gg 0$.
By \cite[p.~356, Lemma 17]{va}, we have
\[
 \int_{G'} \Xi_{G'}(gg') (1 + \sigma_{G'}(gg'))^{-r} \Xi_{G'}(g') \, dg' \le C' \Xi_{G'}(g) 
\]
for some $C'>0$.
Hence the absolute convergence follows from \cite[Lemma 6.5.1(i)]{bp}.

We may assume that $\X' = \X \oplus \X^\perp$, where $\X^\perp$ is a maximal isotropic subspace of $V \otimes_\C W^\perp$.
Then we have an identification
\[
 \SS(\X') = \SS(\X) \hotimes \SS(\X^\perp),
\]
where $\hotimes$ denotes the projective tensor product, such that 
\[
 \omega'(g) (\varphi \otimes \varphi^\perp) =  (\omega(g) \varphi) \otimes \varphi^\perp
\]
for $g \in G$, $\varphi \in \SS(\X)$, and $\varphi^\perp \in \SS(\X^\perp)$.
On the other hand, it follows from the argument in the proof of \cite[Lemma 6.2]{li89} that the map
\[
 (\varphi', \varphi'') \mapsto 
 \int_G (\omega'(g) \varphi', \varphi'') \overline{(\pi(g) v_1, v_2)} \, dg
\]
on $\SS(\X') \times \SS(\X')$ is separately continuous.
From this and \eqref{eq:nonvanish-2}, we can deduce that
\[
 \int_G (\omega(g) \varphi_1, \varphi_2) \overline{(\pi(g) v_1, v_2)} \, dg \ne 0
\]
for some $\varphi_1, \varphi_2 \in \SS(\X)$, so that $\ZZ_{V,W,\chi_V,\chi_W,\psi}(\pi) \ne 0$.
This completes the proof.
\end{proof}

Fix an $m$-dimensional Hermitian space $V$ over $\C$ and a character $\chi_V$ of $\C^\times$ such that $\chi_V|_{\R^\times} = \omega_{\C/\R}^m$.
Let $(r,s)$ be the signature of $V$.
For any $n$-dimensional skew-Hermitian space $W$ over $\C$ and any irreducible tempered representation $\pi$ of $\U(W)$, we define integers $k_\pi, r_\pi, s_\pi$ as in \S \ref{ss:atobe-def} with respect to $k_0$ and $\chi_V$, where $k_0 = -1$ or $0$ is determined by
\[
 m \equiv n + k_0 \bmod 2. 
\]
By Remarks \ref{r:atobe} and \ref{r:Z}, we may assume that $r-r_\pi \ge s - s_\pi$, so that
\begin{equation}
\label{eq:(r,s)}
 (r,s) = 
 \begin{cases}
  (r_\pi + l + 2t + 1, s_\pi + l) & \text{if $k_\pi = -1$;} \\
  (r_\pi + l + 2t, s_\pi + l) & \text{if $k_\pi \ge 0$}
 \end{cases}
\end{equation}
for some integers $l,t$ with $t \ge 0$.

\begin{lem}
\label{l:induction-disc}
Let $\pi$ be a discrete series representation of $\U(W)$ such that $k_\pi, r_\pi, s_\pi$ satisfy \eqref{eq:(r,s)} for some integers $l,t$ with
\[
\begin{cases}
 l \ge 0, t \ge 0 & \text{if $k_\pi = -1$;} \\
 l \ge 0, t \ge 1 & \text{if $k_\pi = 0$;} \\
 l \ge k_\pi, t \ge 0 & \text{if $k_\pi \ge 1$}
\end{cases}
\]
(and hence $m>n$).
Assume that $\theta_{V,W,\chi_V,\chi_W,\psi}(\pi) \ne 0$.
Then there exist an $(n+1)$-dimensional skew-Hermitian space $W'$ over $\C$ containing $W$ and a discrete series representation $\pi'$ of $\U(W')$ such that
\begin{itemize}
\item $\theta_{V,W',\chi_V,\chi_{W'},\psi}(\pi') \ne 0$;
\item $\Hom_{\U(W)}(\pi', \pi) \ne 0$;
\item $k_{\pi'}, r_{\pi'}, s_{\pi'}$ satisfy the following conditions:
\begin{itemize}
\item 
if $k_\pi = -1$, then $k_{\pi'} = 0$ and
\[
 (r,s) = (r_{\pi'} + l + 2t, s_{\pi'} + l);
\]
\item 
if $k_\pi = 0$, then $k_{\pi'} = -1$ and
\[
 \text{$(r,s) = (r_{\pi'} + l + 2(t-1) + 1, s_{\pi'} + l)$ or $(r_{\pi'} + (l-1) + 2t + 1, s_{\pi'} + (l-1))$,}
\]
where the second case happens only if $l \ge 1$;
\item
if $k_\pi \ge 1$, then $k_{\pi'} = k_\pi-1$ and 
\[
 (r,s) = (r_{\pi'} + (l-1) + 2t, s_{\pi'} + (l-1)).
\]
\end{itemize}
\end{itemize}
\end{lem}

\begin{proof}
The assertion was essentially proved by Atobe (see Lemmas 5.1, 5.2, 5.3 of \cite{atobe} in the cases $k_\pi = -1, k_\pi = 0, k_\pi \ge 1$, respectively). 
We only give some details in the case $k_\pi = 0$.
Since $\theta_{V,W,\chi_V,\chi_W,\psi}(\pi) \ne 0$, we have $\# \CC_\pi^\pm(l+t) \le l$ by Theorem \ref{t:atobe}.
Let $W'$ be the skew-Hermitian space over $\C$ of signature $(p+1,q)$, where $(p,q)$ is the signature of $W$.
Then by \cite[Lemma 5.2]{atobe}, there exists a discrete series representation $\pi'$ of $\U(W')$ (relative to the choice of integers $l+t \ll \beta_0 < \beta_1 < \cdots$) satisfying the following conditions:
\begin{itemize}
\item $\Hom_{\U(W)}(\pi', \pi) \ne 0$;
\item $k_{\pi'} = -1$;
\item $(r_{\pi'}, s_{\pi'}) =
\begin{cases}
 (r_\pi + 1, s_\pi) & \text{if $(\frac{1}{2}, +1)$ and $(-\frac{1}{2}, -1)$ do not belong to $\XX_\pi$;} \\
 (r_\pi, s_\pi + 1) & \text{if $(\frac{1}{2}, +1)$ or $(-\frac{1}{2}, -1)$ belongs to $\XX_\pi$;}
\end{cases}
$
\item $\# \CC^\pm_{\pi'}(l+t-1) \le 
\begin{cases}
 l & \text{if $(\frac{1}{2}, +1)$ and $(-\frac{1}{2}, -1)$ do not belong to $\XX_\pi$;} \\
 l-1 & \text{if $(\frac{1}{2}, +1)$ or $(-\frac{1}{2}, -1)$ belongs to $\XX_\pi$.}
\end{cases}
$
\end{itemize}
Note that if $(\frac{\epsilon}{2}, \epsilon) \in \XX_\pi$ for some $\epsilon = \pm 1$, then $(\frac{\epsilon}{2}, \epsilon) \in \XX_\pi^{(\infty)}$.
Since $t \ge 1$, this implies that $\# \CC_\pi^\epsilon(t) \ge 1$ and hence $l \ge 1$.
By Theorem \ref{t:atobe} again, the conditions above imply that $\theta_{V,W',\chi_V,\chi_{W'},\psi}(\pi') \ne 0$.
This completes the proof.
\end{proof}

\begin{lem}
\label{l:induction-temp}
Let $\pi$ be an irreducible tempered representation of $\U(W)$ such that $k_\pi, r_\pi, s_\pi$ satisfy \eqref{eq:(r,s)} for some integers $l,t$ with $t \ge 0$.
Assume that $m>n$ and $\theta_{V,W,\chi_V,\chi_W,\psi}(\pi) \ne 0$.
Then there exist an $(n+1)$-dimensional skew-Hermitian space $W'$ over $\C$ containing $W$ and a discrete series representation $\pi'$ of $\U(W')$ such that
\begin{itemize}
\item $\theta_{V,W',\chi_V,\chi_{W'},\psi}(\pi') \ne 0$;
\item $\Hom_{\U(W)}(\pi', \pi) \ne 0$.
\end{itemize}
\end{lem}

\begin{proof}
As in Lemma \ref{l:induction-disc}, we can deduce the assertion from Theorem \ref{t:atobe} and \cite[Lemmas 5.1, 5.2, 5.3]{atobe}.
We omit the details.
\end{proof}

We now prove Proposition \ref{p:key}.
Let $\pi$ be an irreducible tempered representation of $\U(W)$.
If $m=n$ or $n+1$, then Proposition \ref{p:key} was proved in \cite[Proposition B.4.1]{hls}, \cite[Proposition 11.5(ii)]{gqt}.
(Note that these cases were used by Xue \cite{xue} in the proof of the Gan--Gross--Prasad conjecture.)
Thus we assume that $m \ge n+2$.
Then we need to show that if $\theta_{V,W,\chi_V,\chi_W,\psi}(\pi) \ne 0$, then $\ZZ_{V,W,\chi_V,\chi_W,\psi}(\pi) \ne 0$.

As above, we may assume that $k_\pi, r_\pi, s_\pi$ satisfy \eqref{eq:(r,s)} for some integers $l,t$ with $t \ge 0$.
By Lemmas \ref{l:seesaw} and \ref{l:induction-temp}, we are reduced to the case when $\pi$ is a discrete series representation.
In this case, it follows from Theorem \ref{t:atobe} that if $\theta_{V,W,\chi_V,\chi_W,\psi}(\pi) \ne 0$, then 
\[
\begin{cases}
 l \ge 0, t \ge 0 & \text{if $k_\pi = -1$;} \\
 l \ge 0, t \ge 0 & \text{if $k_\pi = 0$;} \\
 \text{$l \ge 0, t = 0$ or $l \ge k_\pi, t \ge 1$} & \text{if $k_\pi \ge 1$.}
\end{cases}
\]

We first consider the case $k_\pi = -1$ or $0$ (and hence $l \ge 0, t \ge 0$).
By Lemmas \ref{l:seesaw} and \ref{l:induction-disc}, and an induction on $l+t$, we are reduced to the case $k_\pi = 0, l \ge 0, t = 0$.
In this case, the assertion will be proved in Lemma \ref{l:base} below.

We next consider the case $k_\pi \ge 1$.
If $t \ge 1$ (and hence $l \ge k_\pi$), then by Lemmas \ref{l:seesaw} and \ref{l:induction-disc}, and an induction on $k_\pi$, we are reduced to the case $k_\pi = 0, l \ge 0, t \ge 1$.
But this case has already been treated above.
If $t = 0$ (and hence $l \ge \frac{k_\pi}{2}$), then the assertion will be proved in Lemma \ref{l:base} below.

\subsection{Base step}

We continue with the setup of the previous subsection.
To finish the proof of Proposition \ref{p:key}, it remains to prove the following.

\begin{lem}
\label{l:base}
Let $\pi$ be a discrete series representation of $\U(W)$.
Assume that $k_\pi \ge 0$ and 
\[
 (r,s) = (r_\pi + l, s_\pi + l)
\]
for some integer $l \ge \frac{k_\pi}{2}$ (and hence $m \ge n$).
Then we have $\ZZ_{V,W,\chi_V,\chi_W,\psi}(\pi) \ne 0$.
\end{lem}

This lemma is an immediate consequence of a result of Li \cite[Theorem 4.1]{li90}, but we include some details for the convenience of the reader.
Note that by Theorem \ref{t:atobe}, the assumption automatically implies that $\theta_{V,W,\chi_V,\chi_W,\psi}(\pi) \ne 0$.

To prove Lemma \ref{l:base}, we need the notion of $K$-types of minimal degrees introduced by Howe \cite{howe2}.
Let $(p,q)$ be the signature of $W$.
We take the maximal compact subgroup $K \cong \U(p) \times \U(q)$ of $\U(W) = \U(p,q)$ as in \S \ref{ss:real-notation} and parametrize the irreducible representations of $K$ by highest weights
\[
 (a_1, \dots, a_p; b_1, \dots, b_q),
\]
where 
\begin{itemize}
\item $a_i, b_j \in \Z$;
\item $a_1 \ge \dots \ge a_p$ and $b_1 \ge \dots \ge b_q$.
\end{itemize}
Similarly, we take the maximal compact subgroup $K' \cong \U(r) \times \U(s)$ of $\U(V) = \U(r,s)$ and parametrize the irreducible representations of $K'$.

Let $\mathcal{P} = \bigoplus_{d=0}^\infty \mathcal{P}_d$ be the Fock model of the Weil representation $\omega_{V,W,\chi_V,\chi_W,\psi}$ of $\U(W) \times \U(V)$ relative to the datum $(\chi_V,\chi_W,\psi)$ given in \S \ref{ss:main-theorem-setup}, where $\mathcal{P}$ is the space of polynomials in $mn$ variables and $\mathcal{P}_d$ is the subspace of homogeneous polynomials of degree $d$.
Note that $\mathcal{P}_d$ is invariant under the action of $K \times K'$.
For any irreducible representation $\mu$ of $K$ occurring in $\mathcal{P}$, we define the $(r,s)$-degree of $\mu$ as the smallest nonnegative integer $d$ such that the $\mu$-isotypic component of $\mathcal{P}_d$ is nonzero, which depends only on $r-s$ (see \cite[Lemma 1.4.5]{paul1}).

Let $\mathcal{H}$ be the space of joint harmonics, which is a $K \times K'$-invariant subspace of $\mathcal{P}$.
For any irreducible representations $\mu$ and $\mu'$ of $K$ and $K'$, respectively, we say that $\mu$ and $\mu'$ correspond if $\mu \boxtimes \mu'$ occurs in $\mathcal{H}$, in which case $\mu$ and $\mu'$ determine each other.
This correspondence can be described as follows.

\begin{lem}
\label{l:K-type-corresp}
Let $\mu$ and $\mu'$ be irreducible representations of $K$ and $K'$, respectively.
Then $\mu$ and $\mu'$ correspond if and only if $\mu$ and $\mu'$ are of the form
\begin{align*}
 \mu & = (a_1, \dots, a_{p^+}, 0, \dots, 0, b_1, \dots, b_{p^-}; c_1, \dots, c_{q^+}, 0, \dots, 0, d_1, \dots, d_{q^-}) \\
 & + \bigg( \frac{r-s}{2}, \dots, \frac{r-s}{2}; \frac{s-r}{2}, \dots, \frac{s-r}{2} \bigg) + \bigg( \frac{m_0}{2}, \dots, \frac{m_0}{2} \bigg)
\end{align*}
and 
\begin{align*}
 \mu' & = (a_1, \dots, a_{p^+}, 0, \dots, 0, d_1, \dots, d_{q^-}; c_1, \dots, c_{q^+}, 0, \dots, 0, b_1, \dots, b_{p^-}) \\
 & + \bigg( \frac{p-q}{2}, \dots, \frac{p-q}{2}; \frac{q-p}{2}, \dots, \frac{q-p}{2} \bigg) + \bigg( \frac{n_0}{2}, \dots, \frac{n_0}{2} \bigg),
\end{align*}
where
\begin{itemize}
\item $a_i, b_j, c_k, d_l \in \Z$;
\item $a_1 \ge \dots \ge a_{p^+} > 0 > b_1 \ge \dots \ge b_{p^-}$ and $c_1 \ge \dots \ge c_{q^+} > 0 > d_1 \ge \dots \ge d_{q^-}$;
\item $p^+ + p^- \le p$ and $q^+ + q^- \le q$;
\item $p^+ + q^- \le r$ and $p^- + q^+ \le s$.
\end{itemize}
\end{lem}

\begin{proof}
Given our choice of the datum $(\chi_V, \chi_W, \psi)$, the assertion follows from \cite[Theorem 5.4]{konno}.
We remark that the convention in \cite{konno} is different from ours (see \cite[Lemma 3.1]{konno} and \cite[p.~758]{gi2}).
In particular, to switch the left and right actions of $\U(W)$ on $W$, we need to compose the Weil representation $\omega_{V,W, \underline{\xi}}$ as in \cite[\S 3.3]{konno} relative to the pair $\underline{\xi} = (\chi_W, \chi_V^{-1})$ with the automorphism $g \mapsto {}^t g^{-1}$ of $\U(p,q)$.
\end{proof}

Let $\pi$ be an irreducible representation of $\U(W)$ such that the theta lift $\theta_{V,W,\chi_V,\chi_W,\psi}(\pi)$ to $\U(V)$ is nonzero.
Let $\mu$ be a $K$-type of $\pi$, i.e.~an irreducible representation of $K$ occurring in $\pi|_K$.
We say that $\mu$ is of minimal $(r,s)$-degree in $\pi$ if the $(r,s)$-degree of $\mu$ is minimal among all $K$-types of $\pi$, in which case $\mu$ occurs in $\mathcal{H}$.

\begin{lem}
\label{l:lkt->mindeg}
Let $\pi$ be a discrete series representation of $\U(W)$ satisfying the assumption of Lemma \ref{l:base}.
Let $\mu$ be the lowest $K$-type of $\pi$.
Then $\mu$ is of minimal $(r,s)$-degree in $\pi$.
\end{lem}

\begin{proof}
Put $(r_0, s_0) = (r_\pi + [\frac{k_\pi}{2}], s_\pi + [\frac{k_\pi}{2}])$, so that $r_0+s_0 = n$ or $n-1$.
Let $V_0$ be the Hermitian space over $\C$ of signature $(r_0, s_0)$.
Then by Theorem \ref{t:atobe}, the theta lift $\theta_{V_0,W,\chi_V,\chi_W,\psi}(\pi)$ to $\U(V_0)$ is nonzero.
Moreover, by \cite[Proposition 0.5]{paul1} and \cite[Proposition 1.4]{paul2}, $\mu$ is of minimal $(r_0,s_0)$-degree in $\pi$.
On the other hand, since $r_0-s_0 = r-s$, the $(r_0,s_0)$-degree of any $K$-type $\nu$ of $\pi$ agrees with the $(r,s)$-degree of $\nu$.
Hence $\mu$ is of minimal $(r,s)$-degree in $\pi$.
\end{proof}

We also need a seesaw diagram
\[
\begin{tikzcd}
 \U(W) \times \U(W) \arrow[rd,dash] & \U(V) \arrow[ld,dash] \\
 \U(W) \arrow[u,dash] & \U(V_1) \times \U(V_2) \arrow[u,dash]
\end{tikzcd},
\]
where $V_1$ and $V_2$ are the Hermitian spaces over $\C$ of signatures $(r,0)$ and $(0,s)$, respectively, such that $V = V_1 \oplus V_2$ and $K' = \U(V_1) \times \U(V_2)$.
Consider the symplectic spaces
\[
 \W = V \otimes_E W, \quad 
 \W_1 = V_1 \otimes_E W, \quad
 \W_2 = V_2 \otimes_E W
\]
over $F$, so that $\W = \W_1 \oplus \W_2$.
We may take complete polarizations
\[
 \W = \X \oplus \Y, \quad
 \W_1 = \X_1 \oplus \Y_1, \quad
 \W_2 = \X_2 \oplus \Y_2
\]
such that $\X = \X_1 \oplus \X_2$ and $\Y = \Y_1 \oplus \Y_2$.
Write 
\[
 \omega = \omega_{V,W,\chi_V,\chi_W,\psi}, \quad
 \omega_1 = \omega_{V_1,W,\chi_{V_1},\chi_W,\psi}, \quad
 \omega_2 = \omega_{V_2,W,\chi_{V_2},\chi_W,\psi},
\]
where $\chi_{V_1}, \chi_{V_2}$ are characters of $\C^\times$ given by 
\[
 \chi_{V_1}(z) = \left( \frac{z}{\sqrt{z \bar{z}}} \right)^{m_1}, \quad
 \chi_{V_2}(z) = \left( \frac{z}{\sqrt{z \bar{z}}} \right)^{m_2} 
\]
for some integers $m_1, m_2$ such that
\[
 m_1 \equiv r \bmod 2, \quad
 m_2 \equiv s \bmod 2, \quad
 m_1 + m_2 = m_0.
\]
Then we have an identification
\[
 (\omega, S(\X)) = (\omega_1 \boxtimes \omega_2, S(\X_1) \otimes S(\X_2))
\]
as representations of $\U(W) \times \U(V_1) \times \U(V_2)$.
In particular, we have
\begin{equation}
\label{eq:matcoeff-weil} 
 (\omega(g) \varphi_1, \varphi_2) =
 (\omega_1(g) \varphi_{1,1}, \varphi_{2,1})
 (\omega_2(g) \varphi_{1,2}, \varphi_{2,2})
\end{equation}
for $g \in \U(W)$ and $\varphi_1 = \varphi_{1,1} \otimes \varphi_{1,2}, \varphi_2 = \varphi_{2,1} \otimes \varphi_{2,2} \in S(\X)$ with $\varphi_{1,i}, \varphi_{2,i} \in S(\X_i)$.

We now prove Lemma \ref{l:base}.
Let $\pi$ be a discrete series representation of $\U(W)$ satisfying the assumption of Lemma \ref{l:base}.
Let $\mu$ be the lowest $K$-type of $\pi$.
By Lemma \ref{l:lkt->mindeg}, we may write
\begin{align*}
 \mu & = (a_1, \dots, a_{p^+}, 0, \dots, 0, b_1, \dots, b_{p^-}; c_1, \dots, c_{q^+}, 0, \dots, 0, d_1, \dots, d_{q^-}) \\
 & + \bigg( \frac{r-s}{2}, \dots, \frac{r-s}{2}; \frac{s-r}{2}, \dots, \frac{s-r}{2} \bigg) + \bigg( \frac{m_0}{2}, \dots, \frac{m_0}{2} \bigg)
\end{align*}
as in Lemma \ref{l:K-type-corresp}.
Put
\begin{align*}
 \mu_1 & = (a_1, \dots, a_{p^+}, 0, \dots, 0; 0, \dots, 0, d_1, \dots, d_{q^-})
 + \bigg( \frac{r}{2}, \dots, \frac{r}{2}; -\frac{r}{2}, \dots, -\frac{r}{2} \bigg) + \bigg( \frac{m_1}{2}, \dots, \frac{m_1}{2} \bigg), \\ 
 \mu_2 & = (0, \dots, 0, b_1, \dots, b_{p^-}; c_1, \dots, c_{q^+}, 0, \dots, 0)
 + \bigg( {-\frac{s}{2}}, \dots, -\frac{s}{2}; \frac{s}{2}, \dots, \frac{s}{2} \bigg) + \bigg( \frac{m_2}{2}, \dots, \frac{m_2}{2} \bigg),
\end{align*}
so that the tensor product representation $\mu_1 \otimes \mu_2$ contains $\mu$.
Let $\mu'$ be the irreducible representation of $K'$ corresponding to $\mu$.
Let $\mu_1'$ and $\mu_2'$ be the irreducible representations of $\U(V_1)$ and $\U(V_2)$, respectively, given by 
\begin{align*}
 \mu_1' & = (a_1, \dots, a_{p^+}, 0, \dots, 0, d_1, \dots, d_{q^-})
 + \bigg( \frac{p-q}{2}, \dots, \frac{p-q}{2} \bigg) + \bigg( \frac{n_0}{2}, \dots, \frac{n_0}{2} \bigg), \\
 \mu_2' & = (c_1, \dots, c_{q^+}, 0, \dots, 0, b_1, \dots, b_{p^-})
 + \bigg( \frac{q-p}{2}, \dots, \frac{q-p}{2} \bigg) + \bigg( \frac{n_0}{2}, \dots, \frac{n_0}{2} \bigg),
\end{align*}
so that $\mu' = \mu_1' \boxtimes \mu_2'$.
Then the theta lift $\pi_i = \theta_{V_i,W,\chi_{V_i},\chi_W, \psi}(\mu_i')$ to $\U(W)$ is nonzero.
In fact, $\pi_i$ is the unitary highest weight module with lowest $K$-type $\mu_i$ (see \cite{kashiwara-vergne}).
Since $\U(V_i)$ is compact, we may realize the representation $\pi_i \boxtimes \mu_i'$ of $\U(W) \times \U(V_i)$ on the $\mu_i'$-isotypic component $S(\X_i)_{\mu_i'}$ of $S(\X_i)$.
In particular, for $\varphi_{1,i}, \varphi_{2,i} \in S(\X_i)_{\mu_i'}$, the function 
\[
 g \mapsto (\omega_i(g) \varphi_{1,i}, \varphi_{2,i})
\]
is a matrix coefficient of $\pi_i$.
By \eqref{eq:matcoeff-weil}, it suffices to show that the integral
\begin{equation}
\label{eq:matcoeff-integral}
\int_{\U(W)} (\omega_1(g) \varphi_{1,1}, \varphi_{2,1}) (\omega_2(g) \varphi_{1,2}, \varphi_{2,2}) \overline{(\pi(g) v_1, v_2)} \, dg
\end{equation}
is nonzero for some $\varphi_{1,1}, \varphi_{2,1} \in S(\X_1)_{\mu_1'}$, $\varphi_{1,2}, \varphi_{2,2} \in S(\X_2)_{\mu_2'}$, and $v_1, v_2 \in \pi$.
Let $\Psi$ be the Flensted-Jensen function given by
\[
 \Psi(g) = \frac{1}{\dim \mu} \cdot \Tr(P_\mu \pi(g) P_\mu),
\]
where $P_\mu$ is the orthogonal projection to the $\mu$-isotypic component of $\pi$ (see \cite[\S 7]{flensted-jensen}).
Similarly, let $\Psi_i$ be the function given by 
\[
 \Psi_i(g) = \frac{1}{\dim \mu_i} \cdot \Tr(P_{\mu_i} \pi_i(g) P_{\mu_i}).
\]
Then it follows from the proof of \cite[Theorem 4.1]{li90} that
\[
 \int_{\U(W)} \Psi_1(g) \Psi_2(g) \overline{\Psi(g)} \, dg \ne 0.
\]
Since $\Psi_1, \Psi_2, \Psi$ are linear combinations of matrix coefficients of $\pi_1, \pi_2, \pi$, respectively, this integral is a linear combination of integrals of the form \eqref{eq:matcoeff-integral}.
This completes the proof of Lemma \ref{l:base} and hence of Proposition \ref{p:key}.

\section{Proof of Theorem \ref{t:main-lds}\eqref{main-lds-1}}

In this section, we consider the theta lifting from $\U(p,q)$ to $\U(r,s)$ with $p+q=n$ and $r+s=m$ in the case $m>n$ and determine the theta lifts of (limits of) discrete series representations by a global-to-local argument.

\subsection{Local theta lifting}

Let $F$ be a local field of characteristic zero and $E$ an \'etale quadratic algebra over $F$.
We consider the theta lifting from $\U(W)$ to $\U(V)$, where $W$ is an $n$-dimensional skew-Hermitian space over $E$ and $V$ is an $m$-dimensional Hermitian space over $E$ with $m>n$.

\subsubsection{The real case}
\label{sss:local-theta-real}

Suppose that $F=\R$ and $E=\C$.
Let $(p,q)$ and $(r,s)$ be the signatures of $W$ and $V$, respectively.
We take the datum $(\chi_V, \chi_W, \psi)$ given in \S \ref{ss:main-theorem-setup}.

Let $\pi$ be a (limit of) discrete series representation of $\U(p,q)$ and write $\pi = \nA_\b(\lambda)$ as in \eqref{eq:lds} with
\[
 \lambda = (\alpha_1, \dots, \alpha_{p^+}, \beta_1, \dots, \beta_{p^-}, \gamma_1, \dots, \gamma_{q^+}, \delta_1, \dots, \delta_{q^-}) + \bigg( \frac{m_0}{2}, \dots, \frac{m_0}{2} \bigg),
\]
where 
\begin{itemize}
\item $\alpha_i, \gamma_j > 0$ and $\beta_i, \delta_j \le 0$;
\item $p^++p^- = p$ and $q^++q^- = q$.
\end{itemize}
Define $L$- and $A$-parameters $\phi$ and $\phi'$ for $\U_n$ and $\U_m$, respectively, by 
\begin{align*}
 \phi & = \chi_{\kappa_1} \oplus \dots \oplus \chi_{\kappa_n}, \\
 \phi' & = \chi_{\kappa_1} \chi_V^{-1} \chi_W \oplus \dots \oplus \chi_{\kappa_n} \chi_V^{-1} \chi_W \oplus (\chi_W \boxtimes S_{m-n}), 
\end{align*}
where 
\begin{itemize}
\item $\kappa_1 \ge \dots \ge \kappa_{i_0-1} > \frac{m_0}{2} \ge \kappa_{i_0} \ge \dots \ge \kappa_n$;
\item $\{ \kappa_1 - \frac{m_0}{2}, \dots, \kappa_{i_0-1} - \frac{m_0}{2} \} = \{ \alpha_1, \dots, \alpha_{p^+}, \gamma_1, \dots, \gamma_{q^+} \}$ as multi-sets;
\item $\{ \kappa_{i_0} - \frac{m_0}{2}, \dots, \kappa_n - \frac{m_0}{2} \} = \{ \beta_1, \dots, \beta_{p^-}, \delta_1, \dots, \delta_{q^-} \}$ as multi-sets;
\item $i_0 = p^++q^++1$.
\end{itemize}
Then we have 
\[
 \pi = \pi(\phi, \eta)
\]
for some character $\eta$ of $S_\phi$.
We identify $S_\phi, S_{\phi'}$ with quotients of 
\begin{align*}
 \widetilde{S}_\phi & = (\Z/2\Z) e_1 \oplus \dots \oplus (\Z/2 \Z) e_n, \\
 \widetilde{S}_{\phi'} & = (\Z/2\Z) e'_1 \oplus \dots \oplus (\Z/2 \Z) e'_n \oplus (\Z/2 \Z) e'_0
\end{align*}
as in \S \ref{sss:packets-real-lds}, \S \ref{sss:packets-real-coh}, respectively.
Define a character $\eta'$ of $\widetilde{S}_{\phi'}$ by 
\[
 \eta'(e_i') = \zeta_i \times 
 \begin{cases}
  \eta(e_i) & \text{if $i \ne 0$;} \\
  (-1)^{\frac{1}{2}(p-q)(p-q-1) + \frac{1}{2}(r-s)(r-s-1)} & \text{if $i = 0$,}
 \end{cases}
\]
where
\[
 \zeta_i = 
 \begin{cases}
  +1 & \text{if $m \equiv n \bmod 2$ and $0 < i < i_0$}; \\
  -1 & \text{if $m \equiv n \bmod 2$ and $i \ge i_0$;} \\
  +1 & \text{if $m \not \equiv n \bmod 2$ and $i \ne 0$}
 \end{cases}
\]
and 
\[
 \zeta_0 = \zeta_1 \cdots \zeta_n.
\]

\begin{lem}
\label{l:Aq-eta}
Assume that $\theta_{r,s}(\pi) \ne 0$.
Then $\eta'$ descends to a character of $S_{\phi'}$ and the associated representation $\sigma(\phi', \eta')$ of $\U(r,s)$ is equal to $\nA_\q(\lambda')$, where $\q$ and $\lambda'$ are as in Theorem \ref{t:main-lds}\eqref{main-lds-1}.
\end{lem}

\begin{proof}
Recall that $\b = \q(x)$ is associated to
\[
 x = (\underbrace{x_1,\dots,x_1}_{p_1}, \dots, \underbrace{x_n,\dots,x_n}_{p_n}, \underbrace{x_1,\dots,x_1}_{q_1}, \dots, \underbrace{x_n,\dots,x_n}_{q_n})
\]
for any $x_1, \dots, x_n \in \R$ such that $x_1 > \dots > x_n$, where 
\begin{equation}
\label{eq:piqi}
 (p_i, q_i) = 
 \begin{cases}
  (1,0) & \text{if $\eta(e_i) = (-1)^{i-1}$;} \\
  (0,1) & \text{if $\eta(e_i) = (-1)^i$.}
 \end{cases}
\end{equation}
In particular, we have
\[
\begin{aligned}
 p^+ & = p_1 + \dots + p_{i_0-1}, \, & 
 p^- & = p_{i_0} + \dots + p_n, \\
 q^+ & = q_1 + \dots + q_{i_0-1}, \, & 
 q^- & = q_{i_0} + \dots + q_n.
\end{aligned}
\]

For the first assertion, it suffices to show that if $m=n+1$ and $\kappa_{i_0} = \frac{m_0}{2}$, then
\begin{equation}
\label{eq:eta(e_{i_0})}
 \eta(e_{i_0}) = (-1)^{\frac{1}{2}(p-q)(p-q-1) + \frac{1}{2}(r-s)(r-s-1)}.
\end{equation}
Let $k \ge 1$ be the multiplicity of $\frac{m_0}{2}$ in $\{ \kappa_1, \dots, \kappa_n \}$.
Then we have
\[
\begin{cases}
 k_\pi = -1 & \text{if $k$ is even;} \\
 k_\pi \ge 1 & \text{if $k$ is odd,}
\end{cases} 
\]
and 
\[
 (r_\pi, s_\pi) = 
\begin{cases}
 (p^+ + q^-, p^- + q^+) & \text{if $k$ is even;} \\
 (p^+ + q^- - \frac{k_\pi - 1}{2}, p^- + q^+ - \frac{k_\pi + 1}{2}) & \text{if $k$ is odd and $\eta(e_{i_0}) = (-1)^{i_0-1}$;} \\
 (p^+ + q^- - \frac{k_\pi + 1}{2}, p^- + q^+ - \frac{k_\pi - 1}{2}) & \text{if $k$ is odd and $\eta(e_{i_0}) = (-1)^{i_0}$.}
\end{cases}
\]
Hence by Theorem \ref{t:atobe}, we must have
\[
 (r, s) = 
\begin{cases}
 (p^+ + q^- + 1, p^- + q^+) & \text{if $\eta(e_{i_0}) = (-1)^{i_0-1}$;} \\
 (p^+ + q^-, p^- + q^+ + 1) & \text{if $\eta(e_{i_0}) = (-1)^{i_0}$.}
\end{cases}
\]
If $\eta(e_{i_0}) = (-1)^{i_0-1}$, then we have
\begin{align*}
 & \frac{1}{2}(p-q)(p-q-1) + \frac{1}{2}(r-s)(r-s-1) \\
 & = \frac{1}{2}(p^+ - q^+ + p^- - q^-)(p^+ - q^+ + p^- - q^- -1) \\
 & + \frac{1}{2}(p^+ - q^+ - p^- + q^- +1)(p^+ - q^+ - p^- + q^-) \\
 & = (p^+ - q^+)^2 + (p^- - q^-)(p^- - q^- -1) \\
 & \equiv p^+ + q^+ \bmod 2, 
\end{align*}
so that \eqref{eq:eta(e_{i_0})} follows.
If $\eta(e_{i_0}) = (-1)^{i_0}$, then \eqref{eq:eta(e_{i_0})} follows similarly.

For the second assertion, we first note that
\[
 p^+ + q^- \le r, \quad
 p^- + q^+ \le s
\]
by Corollary \ref{c:local-theta-xyzw}.
For $1 \le i \le n+1$, we define a pair of integers $(r_i, s_i)$ as in \S \ref{sss:packets-real-coh}, so that 
\begin{equation}
\label{eq:risi}
 (r_i, s_i) =  
 \begin{cases}
  (1,0) & \text{if $i<i_0$ and $\eta(e_i) = (-1)^{i-1}$;} \\
  (0,1) & \text{if $i<i_0$ and $\eta(e_i) = (-1)^i$;} \\
  (1,0) & \text{if $i>i_0$ and $\eta(e_{i-1}) = (-1)^{i-1}$;} \\
  (0,1) & \text{if $i>i_0$ and $\eta(e_{i-1}) = (-1)^i$}
 \end{cases}
\end{equation}
and 
\[
 (r_{i_0}, s_{i_0}) = (r - p^+ - q^-, s - p^- - q^+).
\]
Since $r_{i_0}, s_{i_0} \ge 0$ and 
\begin{align*}
 \eta'(e_1' + \dots + e_n' + e_0') & = \eta(e_1 + \dots + e_n) \cdot (-1)^{\frac{1}{2}(p-q)(p-q-1) + \frac{1}{2}(r-s)(r-s-1)} \\
 & = (-1)^{\frac{1}{2}(r-s)(r-s-1)},
\end{align*}
we have
\[
 \sigma(\phi', \eta') = \nA_{\tilde{\q}}(\tilde{\lambda}'), 
\]
where $\tilde{\lambda}'$ is given by
\[
 \tilde{\lambda}' = (\underbrace{\tilde{\lambda}'_1,\dots,\tilde{\lambda}'_1}_{r_1}, \dots, \underbrace{\tilde{\lambda}'_{n+1},\dots,\tilde{\lambda}'_{n+1}}_{r_{n+1}}, \underbrace{\tilde{\lambda}'_1,\dots,\tilde{\lambda}'_1}_{s_1}, \dots, \underbrace{\tilde{\lambda}'_{n+1},\dots,\tilde{\lambda}'_{n+1}}_{s_{n+1}})
\]
with 
\[
 \tilde{\lambda}'_i = 
 \begin{cases}
  \kappa_i - \frac{m_0}{2} + \frac{n_0}{2} & \text{if $i < i_0$;} \\
  \frac{n_0}{2} & \text{if $i = i_0$;} \\
  \kappa_{i-1} - \frac{m_0}{2} + \frac{n_0}{2} & \text{if $i > i_0$}
 \end{cases}
\]
and $\tilde{\q} = \q(\tilde{x})$ is associated to
\[
 \tilde{x} = (\underbrace{\tilde{x}_1,\dots,\tilde{x}_1}_{r_1}, \dots, \underbrace{\tilde{x}_{n+1},\dots,\tilde{x}_{n+1}}_{r_{n+1}}, \underbrace{\tilde{x}_1,\dots,\tilde{x}_1}_{s_1}, \dots, \underbrace{\tilde{x}_{n+1},\dots,\tilde{x}_{n+1}}_{s_{n+1}})
\]
for any $\tilde{x}_1, \dots, \tilde{x}_{n+1} \in \R$ such that $\tilde{x}_1 > \dots > \tilde{x}_{n+1}$.
However, we can deduce from \eqref{eq:piqi} and \eqref{eq:risi} that $\tilde{\q} = \q$ and $\tilde{\lambda}' = \lambda'$.
This completes the proof.
\end{proof}

\begin{lem}
\label{l:local-theta-1}
Assume that $\pi$ is a discrete series representation such that
\[
 p^+ + q^- \le r, \quad
 p^- + q^+ \le s
\]
and 
\[
 \alpha_i, \gamma_j \ge \frac{m-n+1}{2}, \quad
 \beta_i, \delta_j \le -\frac{m-n+1}{2}.
\]
Then we have
\[
 \theta_{r,s}(\pi) = \sigma(\phi', \eta').
\]
\end{lem}

\begin{proof}
By a result of Li \cite{li90}, we have $\theta_{r,s}(\pi) \ne 0$ and 
\[
 \theta_{r,s}(\pi) = \nA_\q(\lambda'),
\]
where $\q$ and $\lambda'$ are as in Theorem \ref{t:main-lds}\eqref{main-lds-1}.
Hence the assertion follows from Lemma \ref{l:Aq-eta}.
\end{proof}

\subsubsection{The nonarchimedean case}

Suppose that $F$ is nonarchimedean and $E \ne F \times F$.
Let $\phi$ and $\phi'$ be $L$- and $A$-parameters for $\U_n$ and $\U_m$, respectively, of the form
\begin{align*}
 \phi & = \chi_1 \oplus \dots \oplus \chi_n, \\
 \phi' & = \chi_1 \chi_V^{-1} \chi_W \oplus \dots \oplus \chi_n \chi_V^{-1} \chi_W \oplus (\chi_W \boxtimes S_{m-n})
\end{align*}
with (not necessarily distinct) conjugate-selfdual characters $\chi_1, \dots, \chi_n$ of $E^\times$ with sign $(-1)^{n-1}$.
If $m \equiv n \bmod 2$, we assume further the condition on the $\epsilon$-factor
\begin{equation}
\label{eq:epsilon-assumption}
 \epsilon(\tfrac{1}{2}, \chi_i \chi_V^{-1}, \psi^E_2) = 1 
\end{equation}
for all $i$, where $\psi^E_2$ is the character of $E$ given by $\psi^E_2(x) = \psi(\Tr_{E/F}(\delta x))$.

\begin{lem}
\label{l:local-theta-2}
Assume that $\epsilon(V) = \epsilon(W) = +1$.
Then we have
\[
 \theta_{V,W,\chi_V,\chi_W,\psi}(\pi(\phi, \mathbbm{1})) = \sigma(\phi', \mathbbm{1}).
\]
\end{lem}

\begin{proof}
Write $\pi = \pi(\phi, \mathbbm{1})$ for brevity.
For any $\epsilon = \pm 1$, we define the first occurrence index $m^\epsilon(\pi)$ as the smallest nonnegative integer $m'$ with $m' \equiv m \bmod 2$ such that $\theta_{V_{m'}^\epsilon, W, \chi_V, \chi_W, \psi}(\pi) \ne 0$.
Put
\[
 m^\up(\pi) = \max \{ m^+(\pi), m^-(\pi) \}, \quad 
 m^\down(\pi) = \min \{ m^+(\pi), m^-(\pi) \}.
\]

Assume first that $m \equiv n \bmod 2$.
Then by \cite[Theorem 4.1]{atobe-gan} and \eqref{eq:epsilon-assumption}, we have
\[
 m^\up(\pi) = n+2, \quad
 m^\down(\pi) = n
\]
with $m^\down(\pi) = m^+(\pi)$.
Moreover, it follows from \cite[Theorem 4.3]{atobe-gan} that 
\[
 \theta_{V,W,\chi_V,\chi_W,\psi}(\pi) = J(\chi_W|\cdot|^{\frac{1}{2}(m-n-1)}, \chi_W|\cdot|^{\frac{1}{2}(m-n-3)}, \dots, \chi_W|\cdot|^{\frac{1}{2}}, \pi(\phi_0, \mathbbm{1}))
\]
with an $L$-parameter
\[
 \phi_0 = \chi_1 \chi_V^{-1} \chi_W \oplus \dots \oplus \chi_n \chi_V^{-1} \chi_W
\]
for $\U_n$.
Hence the assertion follows from Lemma \ref{l:local-Apacket}.

Assume next that $m \not \equiv n \bmod 2$.
Then by \cite[Theorem 4.1]{atobe-gan}, we have
\[
\begin{cases}
 m^\up(\pi) = n+1, \quad m^\down(\pi) = n+1 & \text{if $\chi_i \ne \chi_V$ for all $i$;} \\
 m^\up(\pi) = n+3, \quad m^\down(\pi) = n-1 & \text{otherwise}
\end{cases} 
\]
with $m^\down(\pi) = m^+(\pi)$.
Moreover, it follows from \cite[Theorem 4.3]{atobe-gan} that
\[
 \theta_{V,W,\chi_V,\chi_W,\psi}(\pi) = J(\chi_W|\cdot|^{\frac{1}{2}(m-n-1)}, \chi_W|\cdot|^{\frac{1}{2}(m-n-3)}, \dots, \chi_W|\cdot|^1, \pi(\phi_1, \mathbbm{1}))
\]
with an $L$-parameter
\[
 \phi_1 =  \chi_1 \chi_V^{-1} \chi_W \oplus \dots \oplus \chi_n \chi_V^{-1} \chi_W \oplus \chi_W 
\]
for $\U_{n+1}$.
Hence the assertion follows from Lemma \ref{l:local-Apacket}.
\end{proof}

\subsubsection{The split case}

Suppose that $F$ is nonarchimedean and $E = F \times F$.
In this case, we may identify $\chi_V, \chi_W$ with unitary characters of $F^\times$ via the first projection.
Let $\phi$ and $\phi'$ be $L$- and $A$-parameters for $\U_n$ and $\U_m$, respectively, of the form
\begin{align*}
 \phi & = \chi_1 \oplus \dots \oplus \chi_n, \\
 \phi' & = \chi_1 \chi_V^{-1} \chi_W \oplus \dots \oplus \chi_n \chi_V^{-1} \chi_W \oplus (\chi_W \boxtimes S_{m-n})
\end{align*}
with (not necessarily distinct) unitary characters $\chi_1, \dots, \chi_n$ of $F^\times$.

\begin{lem}
\label{l:local-theta-3}
We have
\[
 \theta_{V,W,\chi_V,\chi_W,\psi}(\pi(\phi, \mathbbm{1})) = \sigma(\phi', \mathbbm{1}).
\]
\end{lem}

\begin{proof}
The assertion was proved by M\'inguez \cite{minguez}.
\end{proof}

\subsection{Global theta lifting}
\label{ss:global-theta}

Let $\F$ be a totally real number field with ad\`ele ring $\A = \A_\F$.
Let $\E$ be a totally imaginary quadratic extension of $\F$ and $\omega_{\E/\F}$ the quadratic character of $\A^\times/\F^\times$ associated to $\E/\F$ by global class field theory.
We consider the theta lifting from $\U(\W)$ to $\U(\V)$, where $\W$ is an $n$-dimensional skew-Hermitian space over $\E$ and $\V$ is an $m$-dimensional Hermitian space over $\E$ with $m>n$.
For simplicity, we assume that $\W$ and $\V$ are anisotropic.

Let $\omega_{\V,\W,\chi_\V,\chi_\W,\varPsi}$ be the Weil representation of $\U(\W)(\A) \times \U(\V)(\A)$ relative to $(\chi_\V,\chi_\W,\varPsi)$, where $\chi_\V, \chi_\W$ are characters of $\A_\E^\times / \E^\times$ such that $\chi_\V|_{\A^\times} = \omega_{\E/\F}^m, \chi_\W|_{\A^\times} = \omega_{\E/\F}^n$ and $\varPsi$ is a nontrivial additive character of $\A/\F$.
This is equipped with a natural equivariant map $\varphi \mapsto \theta(\varphi)$ to the space of left $\U(\W)(\F) \times \U(\V)(\F)$-invariant smooth functions on $\U(\W)(\A) \times \U(\V)(\A)$ of moderate growth.
For any irreducible automorphic representation $\varPi$ of $\U(\W)(\A)$, we denote by $\theta_{\V,\W,\chi_\V,\chi_\W,\varPsi}(\varPi)$ the space spanned by automorphic forms on $\U(\V)(\A)$ of the form
\[
 \theta(\varphi, f)(h) = \int_{\U(\W)(\F) \backslash \U(\W)(\A)} \theta(\varphi)(g,h) \overline{f(g)} \, dg
\]
for $\varphi \in \omega_{\V,\W,\chi_\V,\chi_\W,\varPsi}$ and $f \in \varPi$.
If $\theta_{\V,\W,\chi_\V,\chi_\W,\varPsi}(\varPi)$ is nonzero, then it follows from the Howe duality that $\theta_{\V,\W,\chi_\V,\chi_\W,\varPsi}(\varPi)$ is irreducible and isomorphic to $\bigotimes_v \theta_{\V_v,\W_v,\chi_{\V,v},\chi_{\W,v},\varPsi_v}(\varPi_v)$ (see \cite[Corollary 7.1.3]{kudla-rallis}).

We now discuss the nonvanishing of $\theta_{\V,\W,\chi_\V,\chi_\W,\varPsi}(\varPi)$.
For simplicity, we assume that $\varPi_v$ is tempered for all $v$ and that the partial standard $L$-function $L^S(s, \varPi, \chi_\V^{-1})$ of $\varPi$ twisted by $\chi_\V^{-1}$ is holomorphic and nonzero at $s = \frac{1}{2}(m-n+1)$, where $S$ is a sufficiently large finite set of places of $\F$.
Then the Rallis inner product formula, which is a consequence of the Siegel--Weil formula in the convergent range \cite{weil2, ichino}, says that
\[
 \langle \theta(\varphi_1, f_1), \theta(\varphi_2, f_2) \rangle = \frac{L^S(\frac{1}{2}(m-n+1), \varPi, \chi_\V^{-1})}{d^S(\frac{1}{2}(m-n))} \cdot \prod_{v \in S} \ZZ(\varphi_{1,v}, \varphi_{2,v}, f_{1,v}, f_{2,v})
\]
for $\varphi_1 = \bigotimes_v \varphi_{1,v}, \varphi_2 = \bigotimes_v \varphi_{2,v} \in \omega_{\V,\W,\chi_\V,\chi_\W,\varPsi}$ and $f_1 = \bigotimes_v f_{1,v}, f_2 = \bigotimes_v f_{2,v} \in \varPi$, where
\begin{itemize}
\item $\langle \cdot, \cdot \rangle$ is the Petersson inner product; 
\item $d^S(s) = \prod_{i=1}^n L^S(2s+i, \omega_{\E/\F}^{m-n+i})$, which is holomorphic and nonzero at $s = \frac{1}{2}(m-n)$;
\item $\ZZ(\varphi_{1,v}, \varphi_{2,v}, f_{1,v}, f_{2,v})$ is an integral of matrix coefficients defined as in \S \ref{ss:key}, which can be regarded as a doubling zeta integral of Piatetski-Shapiro--Rallis and which is absolutely convergent by \cite[Lemma 9.5]{gi1}, \cite[Lemma 7.2]{yamana}.
\end{itemize}
Hence $\theta_{\V,\W,\chi_\V,\chi_\W,\varPsi}(\varPi)$ is nonzero if and only if there exist $\varphi_{1,v}, \varphi_{2,v} \in \omega_{\V_v,\W_v,\chi_{\V,v},\chi_{\W,v},\varPsi_v}$ and $f_{1,v}, f_{2,v} \in \varPi_v$ such that
\[
 \ZZ(\varphi_{1,v}, \varphi_{2,v}, f_{1,v}, f_{2,v}) \ne 0
\]
for all $v$.

\subsection{Arthur's multiplicity formula}

In this subsection, we review Arthur's multiplicity formula for unitary groups \cite{mok,kmsw}, which is a key ingredient in the proof of Theorem \ref{t:main-lds}\eqref{main-lds-1}.
Let $\F$ be a number field and $\E$ a quadratic extension of $\F$.
Let $\overline{\F}$ and $\overline{\F}_v$ be algebraic closures of $\F$ and $\F_v$, respectively, and fix an embedding $\overline{\F} \hookrightarrow \overline{\F}_v$ over $\F$ for each place $v$ of $\F$.
We also fix an embedding $\E \hookrightarrow \overline{\F}$ over $\F$, which determines an embedding $\E \hookrightarrow \overline{\F}_v$ for each place $v$ of $\F$ and hence a distinguished place $\tilde{v}$ of $\E$ above $v$.
If $v$ is split in $\E$, we identify $\E_v$ with $\F_v \times \F_v$ so that $\tilde{v}$ corresponds to the composition of the natural embedding $\E \hookrightarrow \E_v$ with the first projection $\E_v \rightarrow \F_v$.

Let $\V$ be an $n$-dimensional $\varepsilon$-Hermitian space over $\E$.
Then Arthur's endoscopic classification gives a decomposition of the automorphic discrete spectrum into near equivalence classes of representations:
\[
 L^2_{\mathrm{disc}}(\U(\V)(\F) \backslash \U(\V)(\A)) = \bigoplus_\varPhi L^2_\varPhi(\U(\V)),
\]
where $\varPhi$ runs over global $A$-parameters for $\U_n$, which is a formal unordered finite direct sum of the form
\[
 \varPhi = \bigoplus_i \varPhi_i \boxtimes S_{d_i},
\]
where
\begin{itemize}
\item $\varPhi_i$ is an irreducible conjugate-selfdual cuspidal automorphic representation of $\GL_{n_i}(\A_\E)$ with sign $(-1)^{n-d_i}$;
\item $S_{d_i}$ is the unique $d_i$-dimensional irreducible representation of $\SL_2(\C)$;
\item $(\varPhi_i, d_i) \ne (\varPhi_j, d_j)$ if $i \ne j$;
\item $\sum_i n_i d_i = n$.
\end{itemize}
Moreover, the multiplicity of each irreducible representation in $L^2_\varPhi(\U(\V))$ can be described as follows.

For each place $v$ of $\F$, we regard the localization $\varPhi_v = \bigoplus_i \varPhi_{i,v} \boxtimes S_{d_i}$ of $\varPhi$ at $v$ (where $\varPhi_{i,v}$ is viewed as a representation of $L_{\E_{\tilde{v}}}$ via the local Langlands correspondence) as a local $A$-parameter $\varPhi_v : L_{\F_v} \times \SL_2(\C) \rightarrow {}^L \U_n$ for $\U_n$.
Let $S_{\varPhi_v}$ be the local component group of $\varPhi_v$.
Recall that the local $A$-packet $\Pi_{\varPhi_v}(\U(\V_v))$ consists of semisimple representations of $\U(\V_v)$ of finite length.
We fix a global Whittaker datum, and with respect to its localization at $v$, we denote by $\sigma(\varPhi_v, \eta_v)$ the representation in $\Pi_{\varPhi_v}(\U(\V_v))$ associated to $\eta_v \in \widehat{S}_{\varPhi_v}$.
Let $S_\varPhi$ be the global component group of $\varPhi$, which is defined formally as a free $\Z/2\Z$-module
\[
 S_\varPhi = \bigoplus_i (\Z/2 \Z) e_i, 
\]
where $e_i$ corresponds to $\varPhi_i \boxtimes S_{d_i}$, and which is equipped with a natural homomorphism $S_\varPhi \rightarrow S_{\varPhi_v}$ for each $v$.
This gives rise to a compact group $S_{\varPhi, \A} = \prod_v S_{\varPhi_v}$ equipped with the diagonal map $\Delta : S_\varPhi \rightarrow S_{\varPhi, \A}$.
We denote by $\widehat{S}_{\varPhi, \A}$ the group of continuous characters of $S_{\varPhi, \A}$.
For any $\eta = \bigotimes_v \eta_v \in \widehat{S}_{\varPhi, \A}$, we may form a representation
\[
 \sigma(\varPhi,\eta) = \bigotimes_v \sigma(\varPhi_v,\eta_v)
\]
of $\U(\V)(\A)$.
Finally, let $\epsilon_\varPhi$ be the character of $S_\varPhi$ defined by \cite[(2.5.5)]{mok}.
Then Arthur's multiplicity formula \cite[Theorem* 1.7.1]{kmsw} says that 
\begin{equation}
\label{eq:amf}
 L^2_\varPhi(\U(\V)) \cong \bigoplus_\eta \sigma(\varPhi,\eta),  
\end{equation}
where $\eta$ runs over elements in $\widehat{S}_{\varPhi, \A}$ such that $\eta \circ \Delta = \epsilon_\varPhi$. 

We can describe the character $\epsilon_\varPhi$ more explicitly as follows.

\begin{lem}
\label{l:epsilon}
We have
\[
 \epsilon_\varPhi(e_i) = \prod_{j \ne i} \epsilon(\tfrac{1}{2}, \varPhi_i \times \varPhi_j^\vee)^{\min\{ d_i, d_j \}}, 
\]
where $\epsilon(s, \varPhi_i \times \varPhi_j^\vee)$ is the global $\epsilon$-factor of the pair $(\varPhi_i, \varPhi_j^\vee)$.
In particular, $\epsilon_\varPhi$ is trivial if $d_i=1$ for all $i$.
\end{lem}

\begin{proof}
The character $\epsilon_\varPhi$ is explicated in \cite[Proposition-Definition 8.3.7]{chenevier-lannes} in the case of orthogonal and symplectic groups.
We can apply the same argument to the case of unitary groups, noting that $\epsilon(\tfrac{1}{2}, \varPhi_i \times \varPhi_j^\vee) = 1$ if $\varPhi_i$ and $\varPhi_j$ have the same sign (see \cite[Theorem 2.5.4]{mok}).
\end{proof}

\subsection{Conjugate-selfdual characters}

In this subsection, we collect some results on conjugate-selfdual characters which we will use in the proof of Theorem \ref{t:main-lds}\eqref{main-lds-1}.
Let $F$ be a local field of characteristic zero and $E$ an \'etale quadratic algebra over $F$.
Let $\psi$ be a nontrivial additive character of $F$ and define a nontrivial additive character $\psi^E_2$ of $E$ by $\psi^E_2(x) = \psi(\Tr_{E/F}(\delta x))$.
Let $\chi$ be a character of $E^\times$.
Then $\chi$ is conjugate-selfdual if and only if $\chi$ is trivial on $\mathrm{N}_{E/F}(E^\times)$.
Also, if $E \ne F \times F$, then $\chi$ is conjugate-orthogonal (resp.~conjugate-symplectic) if and only if $\chi|_{F^\times} = \mathbbm{1}$ (resp.~$\chi|_{F^\times} =\omega_{E/F}$).
We consider the value of the $\epsilon$-factor $\epsilon(s, \chi, \psi_2^E)$ at $s = \frac{1}{2}$.

\begin{lem}
\label{l:csd-char-epsilon}
Let $\chi$ be a conjugate-selfdual character of $E^\times$.
\begin{enumerate}
\item
\label{item:csd-char-epsilon1}
 If $E = F \times F$, then we have 
\[
 \epsilon(\tfrac{1}{2}, \chi, \psi_2^E) = 1.
\]
\item 
\label{item:csd-char-epsilon2}
If $E \ne F \times F$, then we have
\[
 \epsilon(\tfrac{1}{2}, \chi, \psi_2^E) = \pm 1.
\]
If further $\chi$ is conjugate-orthogonal, then we have
\[
 \epsilon(\tfrac{1}{2}, \chi, \psi_2^E) = 1.
\]
\item
\label{item:csd-char-epsilon3}
Suppose that $F = \R$ and $E = \C$.
Write
\[
 \chi(z) = \left( \frac{z}{\sqrt{z \bar{z}}} \right)^{2 \kappa}
\]
for some $\kappa \in \frac{1}{2} \Z$.
Assume further that $\delta = \sqrt{-1}$ and $\psi(x) = e^{- 2 \pi \sqrt{-1} x}$, so that $\psi_2^E(z) = e^{2 \pi (z - \bar{z})}$.
If $\kappa \in \Z$, then we have
\[
 \epsilon(\tfrac{1}{2}, \chi, \psi_2^E) = 1.
\]
If $\kappa \notin \Z$, then we have
\[
 \epsilon(\tfrac{1}{2}, \chi, \psi_2^E) = 
 \begin{cases}
  +1 & \text{if $\kappa > 0$;} \\
  -1 & \text{if $\kappa < 0$.}
 \end{cases}
\]
\end{enumerate}
\end{lem}

\begin{proof}
If $F = E \times E$, then we may write $\chi = \chi_0 \boxtimes \chi_0^{-1}$ and $\psi_2^E = \psi_0 \boxtimes \psi_0^{-1}$ for some characters $\chi_0$ and $\psi_0$ of $F^\times$ and $F$, respectively.
Then we have
\[
 \epsilon(\tfrac{1}{2}, \chi, \psi_2^E)
 = \epsilon(\tfrac{1}{2}, \chi_0, \psi_0) \cdot 
 \epsilon(\tfrac{1}{2}, \chi_0^{-1}, \psi_0^{-1})
 = 1, 
\]
so that \eqref{item:csd-char-epsilon1} follows.
For \eqref{item:csd-char-epsilon2} and \eqref{item:csd-char-epsilon3}, see \cite[Propositions 5.1 and 5.2]{ggp1} and \cite[Proposition 2.1]{ggp2}, respectively.
\end{proof}

Let $F$ be a nonarchimedean local field of characteristic zero and $E$ a quadratic extension of $F$.
We prove the existence of a conjugate-symplectic character $\chi$ of $E^\times$ with a prescribed value of $\epsilon(\tfrac{1}{2}, \chi, \psi^E_2)$.

\begin{lem}
\label{l:csd-char-local}
Assume that either
\begin{itemize}
\item $E$ is unramified over $F$; or
\item the residual characteristic of $F$ is odd and $E$ is ramified over $F$.
\end{itemize}
Then there exists a conjugate-symplectic character $\chi$ of $E^\times$ such that
\[
 \epsilon(\tfrac{1}{2}, \chi, \psi^E_2) = 1.
\]
\end{lem}

\begin{proof}
If $E$ is unramified over $F$, then the assertion follows from \cite[Proposition 3.1]{ggp2}.
Hence we may assume that the residual characteristic of $F$ is odd and $E$ is ramified over $F$.
Let $\mathfrak{o}_F$ (resp.~$\mathfrak{o}_E$) be the maximal compact subring of $F$ (resp.~$E$), $\mathfrak{p}_F$ (resp.~$\mathfrak{p}_E$) the maximal ideal of $\mathfrak{o}_F$ (resp.~$\mathfrak{o}_E$), and $\varpi_F$ (resp.~$\varpi_E$) a uniformizer of $\mathfrak{o}_F$ (resp.~$\mathfrak{o}_E$).
Then the different of $E/F$ is $\mathfrak{p}_E$ and we may assume that $\varpi_E^2 = \varpi_F$, so that $\Tr_{E/F}(\varpi_E) = 0$ and $\varpi_E^{-1} \delta \in F^\times$.
In particular, there exists an integer $a$ such that $\psi^E_2$ is trivial on $\varpi_E^{-2a} \mathfrak{o}_E$ but nontrivial on $\varpi_E^{-2a-1} \mathfrak{o}_E$.
Since $\omega_{E/F}|_{1 + \mathfrak{p}_F} = \mathbbm{1}$ and $\mathfrak{o}_F^\times / (1 + \mathfrak{p}_F) \cong \mathfrak{o}_E^\times / (1 + \mathfrak{p}_E)$, there are precisely two conjugate-symplectic characters of $E^\times$ of conductor $1$.
Indeed, such a character $\chi$ is given by 
\[
 \chi|_{F^\times} = \omega_{E/F}, \quad
 \chi|_{1 + \mathfrak{p}_E} = \mathbbm{1}, \quad
 \chi(\varpi_E) = \zeta
\]
for some square root $\zeta$ of $\omega_{E/F}(\varpi_F)$.
By \cite[(3.6.3)]{tate}, we have
\[
 \epsilon(\tfrac{1}{2}, \chi, \psi^E_2) = \chi(\varpi_E^{2a+1}) \cdot \frac{\mathcal{I}}{|\mathcal{I}|}, 
\]
where 
\[
 \mathcal{I} = \int_{\mathfrak{o}_E^\times} \chi(x)^{-1} \psi^E_2(\varpi_E^{-2a-1} x) \, dx.
\]
Note that $\chi(\varpi_E^{2a+1}) = \zeta \cdot \omega_{E/F}(\varpi_F^a)$ but that $\mathcal{I}$ does not depend on $\zeta$ since
\begin{align*}
 \mathcal{I} & = \operatorname{vol}(1 + \mathfrak{p}_E) \cdot \sum_{x \in \mathfrak{o}_E^\times/ (1 + \mathfrak{p}_E)} \chi(x)^{-1} \psi^E_2(\varpi_E^{-2a-1} x)  \\
 & = \operatorname{vol}(1 + \mathfrak{p}_E) \cdot \sum_{x \in \mathfrak{o}_F^\times/ (1 + \mathfrak{p}_F)} \omega_{E/F}(x)^{-1} \psi(2 \delta_0 x),
\end{align*}
where $\delta_0 = \varpi_E^{-2a-1} \delta \in F^\times$.
On the other hand, by Lemma \ref{l:csd-char-epsilon}, we have $\epsilon(\tfrac{1}{2}, \chi, \psi_2^E) = \pm 1$.
Hence we can choose $\zeta$ so that $\epsilon(\tfrac{1}{2}, \chi, \psi^E_2) = 1$.
\end{proof}

Let $\F$ be a number field and $\E$ a quadratic extension of $\F$.
Let $\Sigma$ be the set of places $v$ of $\F$ such that $\E_v \ne \F_v \times \F_v$.
We globalize local conjugate-selfdual characters to a global conjugate-selfdual character.

\begin{lem}
\label{l:csd-char-global}
For each $v \in \Sigma$, let $\chi_v$ be a conjugate-orthogonal (resp.~conjugate-symplectic) character of $\E_v^\times$.
Assume that $\chi_v$ is unramified for almost all $v \in \Sigma$.
Then there exists a conjugate-orthogonal (resp.~conjugate-symplectic) character $\chi_0$ of $\A_\E^\times/\E^\times$ such that
\[
 \chi_{0,v} = \chi_v
\]
for all $v \in \Sigma$.
\end{lem}

\begin{proof}
We may reduce the conjugate-symplectic case to the conjugate-orthogonal case by taking a conjugate-symplectic character $\chi'$ of $\A_\E^\times / \E^\times$ and applying the lemma to the character $\chi_v \cdot \chi_v'$ of $\E_v^\times$ for $v \in \Sigma$.
To treat the conjugate-orthogonal case, we consider an anisotropic torus
\[
 T = {\Res}_{E/F}(\mathbb{G}_m) / \mathbb{G}_m
\]
over $\F$.
For each $v \in \Sigma$, let $\nu_v$ be a character of $T_v$.
Assume that $\nu_v$ is unramified for almost all $v \in \Sigma$.
Then we may form a character $\nu_\Sigma = \bigotimes_{v \in \Sigma} \nu_v$ of $T_\Sigma = \prod_{v \in \Sigma} T_v$.
Since $T_\Sigma$ is compact, the image of the natural continuous embedding
\[
 T_\Sigma \hookrightarrow T(\A) / T(\F)
\]
is closed. 
Hence we may extend $\nu_\Sigma$ to a character $\nu_0$ of $T(\A) / T(\F)$, so that 
\[
 \nu_{0,v} = \nu_v
\]
for all $v \in \Sigma$.
This completes the proof.
\end{proof}

\subsection{Global-to-local argument}

We now return to the setup of \S \ref{sss:local-theta-real}, so that $\pi$ is a (limit of) discrete series representation of $\U(p,q)$.
Assume that $\theta_{r,s}(\pi) \ne 0$.
Then by Corollary \ref{c:local-theta-xyzw}, we have
\[
 p^+ + q^- \le r, \quad
 p^- + q^+ \le s.
\]
To prove Theorem \ref{t:main-lds}\eqref{main-lds-1}, we appeal to a global-to-local argument and derive the information about $\theta_{r,s}(\pi)$ from the knowledge of $\theta_{r,s}(\pi_+)$, where $\pi_+$ is an auxiliary discrete series representation of $\U(p,q)$ with sufficiently regular infinitesimal character.
More precisely, let $\pi_+$ be a discrete series representation $\pi_+$ of $\U(p,q)$ of the form $\pi_+ = \nA_\b(\lambda_+)$ with
\[
 \lambda_+ = (\tilde{\alpha}_1, \dots, \tilde{\alpha}_{p^+}, \tilde{\beta}_1, \dots, \tilde{\beta}_{p^-}, \tilde{\gamma}_1, \dots, \tilde{\gamma}_{q^+}, \tilde{\delta}_1, \dots, \tilde{\delta}_{q^-}) + \bigg( \frac{m_0}{2}, \dots, \frac{m_0}{2} \bigg)
\]
such that
\[
 \tilde{\alpha}_i, \tilde{\gamma}_j \ge \frac{m-n+1}{2}, \quad
 \tilde{\beta}_i, \tilde{\delta}_j \le -\frac{m-n+1}{2}
\]
and such that the $\theta$-stable Borel subalgebra determined by $\lambda_+$ agrees with $\b$.
As in \S \ref{sss:local-theta-real}, we define $L$- and $A$-parameters $\phi_+$ and $\phi'_+$ for $\U_n$ and $\U_m$ with respect to $\pi_+$, respectively, so that $\phi_+$ is of the form 
\[
 \phi_+ = \chi_{\tilde{\kappa}_1} \oplus \dots \oplus \chi_{\tilde{\kappa}_n}
\]
with $\tilde{\kappa}_1 > \dots > \tilde{\kappa}_{i_0-1} > \frac{m_0}{2} > \tilde{\kappa}_{i_0} > \dots > \tilde{\kappa}_n$.
Then we have
\[
 \pi_+ = \pi(\phi_+, \eta),  
\]
where $\eta$ is viewed as a character of $S_{\phi_+}$ via the natural isomorphism $S_{\phi_+} \cong \widetilde{S}_\phi$.
Moreover, by Lemma \ref{l:local-theta-1}, we have $\theta_{r,s}(\pi_+) \ne 0$ and
\begin{equation}
\label{eq:theta(pi+)}
 \theta_{r,s}(\pi_+) = \sigma(\phi'_+,\eta'),
\end{equation}
where $\eta'$ is viewed as a character of $S_{\phi'_+}$ via the natural isomorphism $S_{\phi'_+} \cong \widetilde{S}_{\phi'}$.

To simplify the argument, we also need an auxiliary irreducible representation $\pi_0$ of $\U(n,0)$ with Harish-Chandra parameter $(\kappa_{0,1}, \dots, \kappa_{0,n})$ such that
\[
 \kappa_{0,1} > \dots > \kappa_{0,n} > \frac{m_0+m-n+1}{2}.
\]
As in \S \ref{sss:local-theta-real}, we define an $L$-parameter $\phi_0$ for $\U_n$ by 
\[
 \phi_0 = \chi_{\kappa_{0,1}} \oplus \dots \oplus \chi_{\kappa_{0,n}}.
\]
Then we have
\[
 \pi_0 = \pi(\phi_0, \eta_0)
\]
for some character $\eta_0$ of $S_{\phi_0}$.

We now globalize everything in sight.
Let $\F$ be a real quartic field and $\E$ a totally imaginary quadratic extension of $\F$ such that $\E_v = \F_v \times \F_v$ for all places $v$ of $\F$ above $2$.
Let $v_0,v_1,v_2,v_3$ be the four real places of $\F$.
Fix an element $\delta \in \E^\times$ with $\Tr_{\E/\F}(\delta) = 0$ and a nontrivial additive character $\varPsi$ of $\A/\F$ such that 
\begin{itemize}
\item $\delta$ belongs to the $(\F_{v_i}^\times)^2$-orbit of $\sqrt{-1}$ for $i=0,1,2,3$;
\item $\varPsi_{v_i}$ belongs to the $(\F_{v_i}^\times)^2$-orbit of $\psi$ for $i=0,1,2,3$.
\end{itemize}
We will take the global Whittaker datum determined by $\delta$ and $\varPsi$ as in \S \ref{ss:whit}.
Let $\W$ be the $n$-dimensional anisotropic skew-Hermitian space over $\E$ such that
\begin{itemize}
\item the signature of $\W_{v_i}$ is $(p,q)$ for $i=0,1$;
\item the signature of $\W_{v_i}$ is $(n,0)$ for $i=2,3$;
\item $\epsilon(\W_v) = 1$ for all nonarchimedean places $v$ of $\F$.
\end{itemize}
Similarly, let $\V$ be the $m$-dimensional anisotropic Hermitian space over $\E$ such that
\begin{itemize}
\item the signature of $\V_{v_i}$ is $(r,s)$ for $i=0,1$;
\item the signature of $\V_{v_i}$ is $(m,0)$ for $i=2,3$;
\item $\epsilon(\V_v) = 1$ for all nonarchimedean places $v$ of $\F$.
\end{itemize}
Note that such spaces $\W$ and $\V$ exist since $\prod_v \epsilon(\W_v) = \prod_v \epsilon(\V_v) = 1$.
By Lemma \ref{l:csd-char-global}, we may take two (unitary) characters $\chi_\V, \chi_\W$ of $\A_\E^\times/\E^\times$ such that
\begin{alignat*}{2}
 \chi_\V|_{\A^\times} & = \omega_{\E/\F}^m, \quad & \chi_{\V, v_i} & = \chi_V, \\
 \chi_\W|_{\A^\times} & = \omega_{\E/\F}^n, \quad & \chi_{\W, v_i} & = \chi_W
\end{alignat*}
for $i=0,1,2,3$.
Similarly, by Lemmas \ref{l:csd-char-local} and \ref{l:csd-char-global}, we may take conjugate-selfdual characters $\chi_1, \dots, \chi_n$ of $\A_\E^\times/\E^\times$ with sign $(-1)^{n-1}$ satisfying the following conditions:
\begin{itemize}
\item $\chi_{i,v_0} = \chi_{\kappa_i}$ for all $i$;
\item $\chi_{i,v_1} = \chi_{\tilde{\kappa}_i}$ for all $i$;
\item $\chi_{i,v_2} = \chi_{i,v_3} = \chi_{\kappa_{0,i}}$ for all $i$;
\item if $m \equiv n \bmod 2$, then 
\begin{equation}
\label{eq:assumption-epsilon}
 \epsilon(\tfrac{1}{2}, \chi_{i,v} \chi_{\V,v}^{-1}, \varPsi_{2,v}^\E) = 1 
\end{equation}
for all nonarchimedean places $v$ of $\F$ such that $\E_v \ne \F_v \times \F_v$, where $\varPsi_{2,v}^\E$ is the character of $\E_v$ given by $\varPsi_{2,v}^\E(x) = \varPsi_v(\Tr_{\E_v/\F_v}(\delta x))$.
\end{itemize}
In particular, $\chi_1, \dots, \chi_n, \chi_\V$ are pairwise distinct.

Define a global $A$-parameter $\varPhi$ for $\U_n$ by
\[
 \varPhi = \chi_1 \oplus \dots \oplus \chi_n,
\]
so that
\[
 \varPhi_{v_0} = \phi, \quad \varPhi_{v_1} = \phi_+, \quad \varPhi_{v_2} = \varPhi_{v_3} = \phi_0.
\]
Let $S_\varPhi$ be the global component group of $\varPhi$, which is defined formally as a free $\Z/2\Z$-module
\[
 S_\varPhi = (\Z/2\Z) e_1 \oplus \dots \oplus (\Z/2 \Z) e_n,
\]
where $e_i$ corresponds to $\chi_i$.
For each place $v$ of $\F$, let $S_{\varPhi_v}$ be the local component group of $\varPhi_v$ equipped with a natural homomorphism $S_\varPhi \rightarrow S_{\varPhi_v}$.
We denote by $e_{i,v}$ the image of $e_i$ in $S_{\varPhi_v}$.
Recall the compact group $S_{\varPhi, \A} = \prod_v S_{\varPhi_v}$ equipped with the diagonal map $\Delta : S_\varPhi \rightarrow S_{\varPhi, \A}$.
Define a continuous character $\eta = \bigotimes_v \eta_v$ of $S_{\varPhi,\A}$ by
\begin{itemize}
\item $\eta_{v_0} = \eta_{v_1} = \eta$;
\item $\eta_{v_2} = \eta_{v_3} = \eta_0$;
\item $\eta_v = \mathbbm{1}$ for all nonarchimedean places $v$ of $\F$.
\end{itemize}
Note that $\eta_v(e_{1,v} + \dots + e_{n,v}) = \epsilon(\W_v)$ for all $v$.
Let 
\[
 \varPi_v = \pi(\varPhi_v, \eta_v)
\]
be the irreducible tempered representation in the local $L$-packet $\Pi_{\varPhi_v}(\U(\W_v))$ associated to $\eta_v$.
Then we may form an irreducible representation $\varPi = \bigotimes_v \varPi_v$ of $\U(\W)(\A)$.
Since
\[
 (\eta \circ \Delta) (e_i) = \prod_v \eta_v(e_{i,v}) = 1
\]
for all $i$, it follows from Arthur's multiplicity formula \eqref{eq:amf} that $\varPi$ is automorphic.
Moreover, since $m>n$ and $\chi_i \chi_{\V}^{-1} \ne \mathbbm{1}$ for all $i$, the partial standard $L$-function
\[
 L^S(s, \varPi, \chi_\V^{-1}) = L^S(s, \chi_1 \chi_{\V}^{-1}) \cdots L^S(s, \chi_n \chi_{\V}^{-1})
\]
is holomorphic and nonzero at $s = \frac{1}{2}(m-n+1)$.

We consider the global theta lift $\varSigma = \theta_{\V, \W, \chi_\V, \chi_\W, \varPsi}(\varPi)$ to $\U(\V)(\A)$.
Recall that the local theta lift $\varSigma_v = \theta_{\V_v, \W_v, \chi_{\V,v}, \chi_{\W,v}, \varPsi_v}(\varPi_v)$ to $\U(\V_v)$ is nonzero by assumption if $v = v_0$ and by Lemmas \ref{l:local-theta-1}, \ref{l:local-theta-2}, \ref{l:local-theta-3} if $v \ne v_0$.
Hence there exist $\varphi_{1,v}, \varphi_{2,v} \in \omega_{\V_v,\W_v,\chi_{\V,v},\chi_{\W,v},\varPsi_v}$ and $f_{1,v}, f_{2,v} \in \varPi_v$ such that
\[
 \ZZ(\varphi_{1,v}, \varphi_{2,v}, f_{1,v}, f_{2,v}) \ne 0
\]
by Proposition \ref{p:key} if $v$ is real and by \cite[Proposition 11.5]{gqt}, \cite[Lemma 8.6]{yamana} if $v$ is nonarchimedean.
As explained in \S \ref{ss:global-theta}, this implies that $\varSigma$ is nonzero.
Thus we obtain an irreducible automorphic representation $\varSigma = \bigotimes_v \varSigma_v$ of $\U(\V)(\A)$.

Finally, we derive the information about $\varSigma_{v_0} = \theta_{r,s}(\pi)$ from the knowledge of $\varSigma_v$ for $v \ne v_0$ and Arthur's multiplicity formula.
Define a global $A$-parameter $\varPhi'$ for $\U_m$ by
\[
 \varPhi' = \chi_1 \chi_\V^{-1} \chi_\W \oplus \dots \oplus \chi_n \chi_\V^{-1} \chi_\W \oplus (\chi_\W \boxtimes S_{m-n}),
\]
so that
\[
 \varPhi'_{v_0} = \phi', \quad \varPhi'_{v_1} = \phi_+'.
\]
Let $S_{\varPhi'}$ be the global component group of $\varPhi'$, which is defined formally as a free $\Z/2\Z$-module
\[
 S_{\varPhi'} = (\Z/2\Z) e'_1 \oplus \dots \oplus (\Z/2 \Z) e'_n \oplus (\Z/2 \Z) e'_0,
\]
where $e'_i$ corresponds to $\chi_1 \chi_\V^{-1} \chi_\W$ (resp.~$\chi_\W \boxtimes S_{m-n}$) if $i \ne 0$ (resp.~$i=0$).
For each place $v$ of $\F$, let $S_{\varPhi'_v}$ be the local component group of $\varPhi'_v$ equipped with a natural homomorphism $S_{\varPhi'} \rightarrow S_{\varPhi'_v}$.
We denote by $e'_{i,v}$ the image of $e'_i$ in $S_{\varPhi'_v}$.
By Lemmas \ref{l:local-theta-2} and \ref{l:local-theta-3}, $\varSigma$ occurs in the near equivalence class $L^2_{\varPhi'}(\U(\V))$.
Hence $\varSigma_v$ belongs to the local $A$-packet $\Pi_{\varPhi'_v}(\U(\V_v))$ for all $v$.
Since $\Pi_{\varPhi'_v}(\U(\V_v))$ is multiplicity-free, we may associate to $\varSigma_v$ a character $\eta'_v$ of $S_{\varPhi'_v}$.
Then it follows from Arthur's multiplicity formula \eqref{eq:amf} and Lemma \ref{l:epsilon} that
\[
 \prod_v \eta_v'(e_{i,v}') =
 \begin{cases}
  \epsilon(\tfrac{1}{2}, \chi_i \chi_{\V}^{-1}) & \text{if $i \ne 0$;} \\
  \epsilon(\tfrac{1}{2}, \varPi, \chi_{\V}^{-1}) & \text{if $i = 0$.}
 \end{cases}
\]
However, it follows from Lemma \ref{l:csd-char-epsilon} and \eqref{eq:assumption-epsilon} that
\[
 \epsilon(\tfrac{1}{2}, \chi_i \chi_{\V}^{-1}) = \prod_v \epsilon(\tfrac{1}{2}, \chi_{i,v} \chi_{\V,v}^{-1}, \varPsi_{2,v}^\E) = 1,
\]
so that 
\[
 \epsilon(\tfrac{1}{2}, \varPi, \chi_{\V}^{-1})
 = \epsilon(\tfrac{1}{2}, \chi_1 \chi_{\V}^{-1}) \cdots \epsilon(\tfrac{1}{2}, \chi_n \chi_{\V}^{-1}) = 1.
\]
Hence we have
\[
  \prod_v \eta_v'(e_{i,v}') = 1
\]
for all $i$.
On the other hand, by Lemmas \ref{l:local-theta-2} and \ref{l:local-theta-3}, we have $\eta_v' = \mathbbm{1}$ for all nonarchimedean places $v$ of $\F$.
Since $\eta_{v_2}' = \eta_{v_3}'$, we conclude that
\[
 \eta_{v_0}'(e_{i,v_0}) = \eta_{v_1}'(e_{i,v_1}) 
\]
for all $i$.
But by \eqref{eq:theta(pi+)}, $\eta_{v_1}'$ agrees with the character $\eta'$ of $S_{\phi'_+} \cong \widetilde{S}_{\phi'}$ as in \S \ref{sss:local-theta-real}, so that
\[
 \theta_{r,s}(\pi) = \varSigma_{v_0} = \sigma(\varPhi_{v_0}', \eta_{v_0}') = \sigma(\phi', \eta').
\]
This combined with Lemma \ref{l:Aq-eta} proves Theorem \ref{t:main-lds}\eqref{main-lds-1}.

\section{Proof of Theorem \ref{t:main-lds}\eqref{main-lds-2}}

In this section, we consider the theta lifting from $\U(p,q)$ to $\U(r,s)$ with $p+q=n$ and $r+s=m$ in the case $m \le n$ and determine the theta lifts of (limits of) discrete series representations by switching the roles of $\U(p,q)$ and $\U(r,s)$.

Let $\pi$ be a (limit of) discrete series representation of $\U(p,q)$.
Assume that $\theta_{r,s}(\pi) \ne 0$.
If $m=n$ or $n-1$, then Theorem \ref{t:main-lds} was proved by Paul \cite{paul1, paul2}.
Thus we assume that $m \le n-2$.
Put $k = n-m$.
Write $\pi = \pi(\phi, \eta)$, where $\phi$ is a (limit of) discrete series $L$-parameter for $\U_n$ and $\eta$ is a character of $S_\phi$.
Write 
\[
 \phi = (m_1 \chi_{\kappa_1} \oplus \dots \oplus m_a \chi_{\kappa_a}) \otimes \chi_V
\]
and 
\[
 S_\phi = (\Z/2\Z) e_1 \oplus \dots \oplus (\Z/2\Z) e_a,
\]
where 
\begin{itemize}
\item $\kappa_i \in \Z + \frac{k-1}{2}$;
\item $\kappa_1 > \dots > \kappa_a$;
\item $m_i$ is a positive integer;
\item $m_1 + \dots + m_a = n$,
\end{itemize}
and $e_i$ corresponds to $\chi_{\kappa_i} \chi_V$.
Put $n_i = m_1 + \dots + m_{i-1}$ and
\[
 (p_i, q_i) =
 \begin{cases}
  (\frac{m_i+1}{2}, \frac{m_i-1}{2}) & \text{if $m_i$ is odd and $\eta(e_i) = (-1)^{n_i}$;} \\
  (\frac{m_i-1}{2}, \frac{m_i+1}{2}) & \text{if $m_i$ is odd and $\eta(e_i) = (-1)^{n_i+1}$;} \\
  (\frac{m_i}{2}, \frac{m_i}{2}) & \text{if $m_i$ is even.}
 \end{cases}
\]
Then by Corollary \ref{c:going-down}, there exist $0 \le i_0 \le a-k$ and $\epsilon_0 = \pm 1$ such that 
\begin{itemize}
\item $\kappa_{i_0 + i} = \frac{k+1}{2} - i$ for all $1 \le i \le k$;
\item $m_{i_0 + i}$ is odd for all $1 < i < k$;
\item $\eta(e_{i_0 + 1}) = \epsilon_0 \times 
\begin{cases}
 (-1)^{n_{i_0 + 1}} & \text{if $m_{i_0 + 1}$ is odd;} \\
 (-1)^{n_{i_0 + 1}+1} & \text{if $m_{i_0 + 1}$ is even;}
\end{cases}
$
\item $\eta(e_{i_0 + i}) = \epsilon_0 \cdot (-1)^{n_{{i_0 + i}}}$ for all $1 < i \le k$.
\end{itemize}
In particular, if we write $\pi = \nA_\b(\lambda)$ as in \eqref{eq:lds}, then $\b$ and $\lambda$ satisfy the conditions in Theorem \ref{t:main-lds}\eqref{main-lds-2}.
Put 
\[
\begin{aligned}
 p^+ & = p_1 + \dots + p_{i_0}, \, & 
 p^- & = p_{i_0+k+1} + \dots + p_a, \\
 q^+ & = q_1 + \dots + q_{i_0}, \, & 
 q^- & = q_{i_0+k+1} + \dots + q_a,
\end{aligned}
\]
and 
\[
 l = 
\begin{cases}
 q_{i_0 + 1} + \dots + q_{i_0 + k}
 & \text{if $\epsilon_0 = +1$;} \\
 p_{i_0 + 1} + \dots + p_{i_0 + k}
 & \text{if $\epsilon_0 = -1$.}
\end{cases}
\]
Then we have 
\begin{align*}
 (p,q) & = 
\begin{cases}
 (p^+ + p^- + l + k, q^+ + q^- + l)
 & \text{if $\epsilon_0 = +1$, $m_{i_0+1}$ is odd, $m_{i_0+k}$ is odd;} \\
 (p^+ + p^- + l + k-1, q^+ + q^- + l)
 & \text{if $\epsilon_0 = +1$, $m_{i_0+1}$ is odd, $m_{i_0+k}$ is even;} \\
 (p^+ + p^- + l + k-1, q^+ + q^- + l)
 & \text{if $\epsilon_0 = +1$, $m_{i_0+1}$ is even, $m_{i_0+k}$ is odd;} \\
 (p^+ + p^- + l + k-2, q^+ + q^- + l)
 & \text{if $\epsilon_0 = +1$, $m_{i_0+1}$ is even, $m_{i_0+k}$ is even;} \\
 (p^+ + p^- + l, q^+ + q^- + l + k)
 & \text{if $\epsilon_0 = -1$, $m_{i_0+1}$ is odd, $m_{i_0+k}$ is odd;} \\
 (p^+ + p^- + l, q^+ + q^- + l + k-1)
 & \text{if $\epsilon_0 = -1$, $m_{i_0+1}$ is odd, $m_{i_0+k}$ is even;} \\
 (p^+ + p^- + l, q^+ + q^- + l + k-1)
 & \text{if $\epsilon_0 = -1$, $m_{i_0+1}$ is even, $m_{i_0+k}$ is odd;} \\
 (p^+ + p^- + l, q^+ + q^- + l + k-2)
 & \text{if $\epsilon_0 = -1$, $m_{i_0+1}$ is even, $m_{i_0+k}$ is even,}
\end{cases} \\
 (r,s) & = 
\begin{cases}
 (p^+ + q^- + l, p^- + q^+ + l)
 & \text{if $\epsilon_0 = +1$, $m_{i_0+1}$ is odd, $m_{i_0+k}$ is odd;} \\
 (p^+ + q^- + l, p^- + q^+ + l-1)
 & \text{if $\epsilon_0 = +1$, $m_{i_0+1}$ is odd, $m_{i_0+k}$ is even;} \\
 (p^+ + q^- + l-1, p^- + q^+ + l)
 & \text{if $\epsilon_0 = +1$, $m_{i_0+1}$ is even, $m_{i_0+k}$ is odd;} \\
 (p^+ + q^- + l-1, p^- + q^+ + l-1)
 & \text{if $\epsilon_0 = +1$, $m_{i_0+1}$ is even, $m_{i_0+k}$ is even;} \\
 (p^+ + q^- + l, p^- + q^+ + l)
 & \text{if $\epsilon_0 = -1$, $m_{i_0+1}$ is odd, $m_{i_0+k}$ is odd;} \\
 (p^+ + q^- + l-1, p^- + q^+ + l)
 & \text{if $\epsilon_0 = -1$, $m_{i_0+1}$ is odd, $m_{i_0+k}$ is even;} \\
 (p^+ + q^- + l, p^- + q^+ + l-1)
 & \text{if $\epsilon_0 = -1$, $m_{i_0+1}$ is even, $m_{i_0+k}$ is odd;} \\
 (p^+ + q^- + l-1, p^- + q^+ + l-1)
 & \text{if $\epsilon_0 = -1$, $m_{i_0+1}$ is even, $m_{i_0+k}$ is even.}
\end{cases}
\end{align*}

Assume first that $m_i = 1$ for all $i_0 + 1 < i < i_0 + k$.
We only consider the case $\epsilon_0 = +1$; the case $\epsilon_0 = -1$ is similar.
Write $\pi = \nA_\b(\lambda)$ as in \eqref{eq:lds} and put 
\[
 p' = p_{i_0+1}, \quad
 p'' = p_{i_0+k}, \quad
 q' = q_{i_0+1}, \quad
 q'' = q_{i_0+k}.
\]
Then $\lambda$ is of the form 
\begin{align*}
 \lambda & = \bigg( \alpha_1, \dots, \alpha_{p^+},
 \underbrace{\frac{k-1}{2}, \dots, \frac{k-1}{2}}_{p'},
 \frac{k-3}{2}, \frac{k-5}{2}, \dots, -\frac{k-3}{2},
 \underbrace{-\frac{k-1}{2}, \dots, -\frac{k-1}{2}}_{p''},
 \beta_1, \dots, \beta_{p^-}, \\
 & \phantom{{} = \bigg(} \gamma_1, \dots, \gamma_{q^+}, 
 \underbrace{\frac{k-1}{2}, \dots, \frac{k-1}{2}}_{q'},
 \underbrace{-\frac{k-1}{2}, \dots, -\frac{k-1}{2}}_{q''},
 \delta_1, \dots, \delta_{q^-} \bigg)
 + \bigg( \frac{m_0}{2}, \dots, \frac{m_0}{2} \bigg),
\end{align*}
where $\alpha_i, \gamma_j > \frac{k-1}{2}$ and $\beta_i, \delta_j < -\frac{k-1}{2}$, and $\b = \q(x)$ is associated to
\begin{align*}
 x & = (x^+_1, \dots, x^+_{p^+}, x'_1, \dots, x'_{p'}, z_2, z_3, \dots, z_{k-1}, x''_1, \dots, x''_{p''}, x^-_1, \dots, x^-_{p^-}, \\
 & \phantom{{} = (} y^+_1, \dots, y^+_{q^+}, y'_1, \dots, y'_{q'}, y''_1, \dots, y''_{q''}, y^-_1, \dots, y^-_{q^-}) 
\end{align*}
such that 
\begin{align*}
&
\begin{cases}
 \mathrlap{x'_1 > y'_1 > \dots > x'_{q'} > y'_{q'} > x'_{p'}}\hphantom{x''_1 > y''_1 > \dots > x''_{q''} > y''_{q''} > x''_{p''}} &
 \text{if $m_{i_0+1}$ is odd;} \\
 y'_1 > x'_1 > \dots > y'_{q'} > x'_{p'} &
 \text{if $m_{i_0+1}$ is even,}
\end{cases} \\
&
\begin{cases}
 x''_1 > y''_1 > \dots > x''_{q''} > y''_{q''} > x''_{p''} &
 \text{if $m_{i_0+k}$ is odd;} \\
 x''_1 > y''_1 > \dots > x''_{p''} > y''_{q''} &
 \text{if $m_{i_0+k}$ is even.}
\end{cases} 
\end{align*}
We assume without loss of generality that
\[
 x'_{p'} > 0 > x''_1.
\]
Put $\sigma = \nA_{\b'}(\lambda')$, where
\begin{align*}
 \lambda' & = \bigg( \alpha_1, \dots, \alpha_{p^+},
 \underbrace{\frac{k-1}{2}, \dots, \frac{k-1}{2}}_{p'-1}, 
 \underbrace{-\frac{k-1}{2}, \dots, -\frac{k-1}{2}}_{q''},
 \delta_1, \dots, \delta_{q^-}, \\
 & \phantom{{} = \bigg(} \gamma_1, \dots, \gamma_{q^+}, 
 \underbrace{\frac{k-1}{2}, \dots, \frac{k-1}{2}}_{q'}, 
 \underbrace{-\frac{k-1}{2}, \dots, -\frac{k-1}{2}}_{p''-1},
 \beta_1, \dots, \beta_{p^-} \bigg)
 + \bigg( \frac{n_0}{2}, \dots, \frac{n_0}{2} \bigg)
\end{align*}
and $\b' = \q(x')$ with
\begin{align*}
 x' & = (x^+_1, \dots, x^+_{p^+}, x'_1, \dots, x'_{p'-1}, y''_1, \dots, y''_{q''}, y^-_1, \dots, y^-_{q^-}, \\
 & \phantom{{} = (} y^+_1, \dots, y^+_{q^+}, y'_1, \dots, y'_{q'}, x''_2, \dots, x''_{p''}, x^-_1, \dots, x^-_{p^-}).
\end{align*}
Then $\sigma$ is a (limit of) discrete series representation of $\U(r,s)$.
By Corollary \ref{c:going-up}, we have $\theta_{p,q}(\sigma) \ne 0$.
Hence it follows from Theorem \ref{t:main-lds}\eqref{main-lds-1} proved in the previous section that
\[
 \theta_{p,q}(\sigma) = \nA_{\tilde{\q}}(\tilde{\lambda}),
\]
where
\begin{align*}
 \tilde{\lambda} & = \bigg( \alpha_1, \dots, \alpha_{p^+},
 \underbrace{\frac{k-1}{2}, \dots, \frac{k-1}{2}}_{p' - 1},
 \underbrace{0, \dots, 0 \vphantom{\frac{k-1}{2}}}_{k},
 \underbrace{-\frac{k-1}{2}, \dots, -\frac{k-1}{2}}_{p'' - 1},
 \beta_1, \dots, \beta_{p^-}, \\
 & \phantom{{} = \bigg(} \gamma_1, \dots, \gamma_{q^+}, 
 \underbrace{\frac{k-1}{2}, \dots, \frac{k-1}{2}}_{q'}, 
 \underbrace{-\frac{k-1}{2}, \dots, -\frac{k-1}{2}}_{q''},
 \delta_1, \dots, \delta_{q^-} \bigg) + \bigg( \frac{m_0}{2}, \dots, \frac{m_0}{2} \bigg)
\end{align*}
and $\tilde{\q} = \q(\tilde{x})$ with
\begin{align*}
 \tilde{x} & = (x^+_1, \dots, x^+_{p^+}, x'_1, \dots, x'_{p'-1}, \underbrace{0, \dots, 0}_k, x''_2, \dots, x''_{p''}, x^-_1, \dots, x^-_{p^-}, \\
 & \phantom{{} = (} y^+_1, \dots, y^+_{q^+}, y'_1, \dots, y'_{q'}, y''_1, \dots, y''_{q''}, y^-_1, \dots, y^-_{q^-}).
\end{align*}
Since $\tilde{\l} \subset \k$ (where $\tilde{\l}$ is the Levi component of $\tilde{\q}$ and $\k$ is the complexified Lie algebra of $K$), we have
\[
 \nA_{\tilde{\q}}(\tilde{\lambda}) = \nA_{\b}(\lambda)
\]
by induction in stages \cite[Corollary 11.86]{kv}.
Thus we have shown that $\theta_{p,q}(\sigma) = \pi$, so that $\theta_{r,s}(\pi) = \sigma$ as desired.

Assume next that $m_{i_1} \ge 3$ for some $i_0 + 1 < i_1 < i_0 + k$.
In particular, we have $p,q,r,s > 0$.
Define a (limit of) discrete series $L$-parameter $\phi'$ for $\U_m$ by
\[
 \phi' = (m'_1 \chi_{\kappa_1} \oplus \dots \oplus m'_a \chi_{\kappa_a}) \otimes \chi_W
\]
and write
\[
 S_{\phi'} = (\Z/2\Z) e'_1 \oplus \dots \oplus (\Z/2\Z) e'_a,
\]
where
\[
 m'_i = 
 \begin{cases}
  m_i -1 & \text{if $i_0 + 1 \le i \le i_0 + k$;} \\
  m_i & \text{if $i \le i_0$ or $i > i_0+k$}
 \end{cases}
\]
and $e'_i$ corresponds to $\chi_{\kappa_i} \chi_W$.
(When $m'_i = 0$, we interpret $e'_i$ as zero.)
Define a character $\eta'$ of $S_{\phi'}$ by
\[
 \eta'(e'_i) = 
 \begin{cases}
  \epsilon_0 \cdot (-1)^{n_{i_0+2} + 1} & \text{if $i_0 + 1 \le i \le i_0 + k$ and $m'_i > 0$;} \\
  \zeta_i \cdot \eta(e_i) & \text{if $i \le i_0$ or $i > i_0+k$,}
 \end{cases}
\]
where 
\[
 \zeta_i = 
 \begin{cases}
  +1 & \text{if $k$ is even and $i \le i_0$;} \\
  -1 & \text{if $k$ is even and $i > i_0 + k$;} \\
  +1 & \text{if $k$ is odd.}
 \end{cases}
\]
Put $\sigma = \pi(\phi', \eta')$, so that $\sigma$ is a (limit of) discrete series representation of $\U(r,s)$.

\begin{lem}
We have $\sigma = \nA_\q(\lambda')$, where $\q$ and $\lambda'$ are as in Theorem \ref{t:main-lds}\eqref{main-lds-2}.
\end{lem}

\begin{proof}
Put $n_i' = m_1' + \dots + m_{i-1}'$.
Then we have the following.
\begin{itemize}
\item
We have
\begin{align*}
\eta(e_i) & = \epsilon_0 \times 
\begin{cases}
 (-1)^{n_i} & \text{if $i=i_0+1$ and $m_i$ is odd;} \\
 (-1)^{n_i+1} & \text{if $i=i_0+1$ and $m_i$ is even;} \\
 (-1)^{n_i} & \text{if $i_0+1 < i \le i_0+k$,}
\end{cases} \\
\eta'(e_i') & = \epsilon_0 \times 
\begin{cases}
 (-1)^{n_i'} & \text{if $i=i_0+1$ and $m_i$ is odd;} \\
 (-1)^{n_i'+1} & \text{if $i=i_0+1$ and $m_i$ is even;} \\
 (-1)^{n_i'} & \text{if $i_0+1 < i \le i_0+k$,}
\end{cases}
\end{align*}
noting that $n_{i_0+1}' = n_{i_0+2} - m_{i_0+1}$, $n_{i_0+2}' = n_{i_0+2} - 1$, and $n_i' \equiv n'_{i_0+2} \bmod 2$ for $i_0+1 < i \le i_0+k$.
(When $m_i'=0$, we ignore the corresponding identity.)
\item 
If $i \le i_0$, then we have $n_i' = n_i$, so that $\eta'(e_i') = (-1)^{n_i'}$ if and only if $\eta(e_i) = (-1)^{n_i}$.
\item 
If $i > i_0 + k$, then we have $n_i' = n_i-k$, so that $\eta'(e_i') = (-1)^{n_i'}$ if and only if $\eta(e_i) = (-1)^{n_i+1}$.
\end{itemize}
This implies the assertion.
\end{proof}

Thus Theorem \ref{t:main-lds} in this case amounts to
\[
 \theta_{r,s}(\pi) = \sigma.
\]
We now proceed by induction on $m'_{i_0+2} + \dots + m'_{i_0+k-1}$.
Define (limits of) discrete series $L$-parameters $\phi_0$ and $\phi_0'$ for $\U_{n-2}$ and $\U_{m-2}$, respectively, by
\begin{align*}
 \phi_0 & = (m_1 \chi_{\kappa_1} \oplus \cdots \oplus m_{i_1-1} \chi_{\kappa_{i_1-1}} \oplus (m_{i_1} - 2) \chi_{\kappa_{i_1}} \oplus m_{i_1+1} \chi_{\kappa_{i_1+1}} \oplus \cdots \oplus m_a \chi_{\kappa_a}) \otimes \chi_V, \\
 \phi'_0 & = (m'_1 \chi_{\kappa_1} \oplus \cdots \oplus m'_{i_1-1} \chi_{\kappa_{i_1-1}} \oplus (m'_{i_1} - 2) \chi_{\kappa_{i_1}} \oplus m'_{i_1+1} \chi_{\kappa_{i_1+1}} \oplus \cdots \oplus m'_a \chi_{\kappa_a}) \otimes \chi_W.
\end{align*}
Then we have a natural isomorphism $S_{\phi_0} \cong S_\phi$ and a natural embedding $S_{\phi'_0} \hookrightarrow S_{\phi'}$, which is an isomorphism if and only if $m_{i_1} \ge 5$.
Put $\pi_0 = \pi(\phi_0, \eta)$ and $\sigma_0 = \pi(\phi'_0, \eta'_0)$, where $\eta$ is viewed as a character of $S_{\phi_0}$ and $\eta_0'$ is the restriction of $\eta'$ to $S_{\phi'_0}$, so that $\pi_0$ and $\sigma_0$ are (limits of) discrete series representations of $\U(p-1, q-1)$ and $\U(r-1, s-1)$, respectively.
Then $\pi$ and $\sigma$ are subrepresentations of $I(\chi, \pi_0)$ and $I(\chi \chi_V^{-1} \chi_W, \sigma_0)$, respectively, where $\chi = \chi_{\kappa_{i_1}} \chi_V$.
In fact, $I(\chi, \pi_0)$ is irreducible and 
\[
 \pi = I(\chi, \pi_0).
\]
Since $\theta_{r,s}(\pi) \ne 0$, we have $\theta_{r-1,s-1}(\pi_0) \ne 0$ by Corollary \ref{c:jacquet}, so that $\theta_{r-1,s-1}(\pi_0) = \sigma_0$ by the induction hypothesis.
Hence by the induction principle \cite[Theorem 4.5.5]{paul1}, $\theta_{r,s}(\pi)$ is a subquotient of $I(\chi \chi_V^{-1} \chi_W, \sigma_0)$.
If $m_{i_1} \ge 5$, then $I(\chi \chi_V^{-1} \chi_W, \sigma_0)$ is irreducible and 
\[
 \sigma = I(\chi \chi_V^{-1} \chi_W, \sigma_0), 
\]
so that $\theta_{r,s}(\pi) = \sigma$ as desired.
Thus we assume that $m_{i_1} = 3$.
Then we have $S_{\phi'} = S_{\phi'_0} \oplus (\Z/2\Z) e'_{i_1}$ and
\[
 I(\chi \chi_V^{-1} \chi_W, \sigma_0) = \sigma \oplus \sigma'
\]
with $\sigma' = \pi(\phi', \eta'')$, where $\eta''$ is the character of $S_{\phi'}$ given by 
\[
 \eta''|_{S_{\phi_0'}} = \eta_0', \quad
 \eta''(e'_{i_1}) = \epsilon_0 \cdot (-1)^{n_{i_0+2}}.
\]
To prove $\theta_{r,s}(\pi) = \sigma$, it suffices to show that
\[
 \theta_{p,q}(\sigma') = 0.
\]
We only consider the case $\epsilon_0 = +1$; the case $\epsilon_0 = -1$ is similar.
Let $k_0 = -1$ or $0$ be such that $k_0 \equiv k \bmod 2$ and put
\[
 t = \frac{k+k_0}{2}.
\]
As in \S \ref{ss:atobe-def}, we define the invariants of $\sigma'$ (with respect to $k_0$ and $\chi_W$).
Then we have $k_{\sigma'} = k_0$ and 
\[
 (r_{\sigma'}, s_{\sigma'}) = 
\begin{cases}
 (p^+ + p^- + l, q^+ + q^- + l)
 & \text{if $\epsilon_0 = +1$, $m_{i_0+1}$ is odd, $m_{i_0+k}$ is odd;} \\
 (p^+ + p^- + l-1, q^+ + q^- + l)
 & \text{if $\epsilon_0 = +1$, $m_{i_0+1}$ is odd, $m_{i_0+k}$ is even;} \\
 (p^+ + p^- + l-1, q^+ + q^- + l)
 & \text{if $\epsilon_0 = +1$, $m_{i_0+1}$ is even, $m_{i_0+k}$ is odd;} \\
 (p^+ + p^- + l-2, q^+ + q^- + l)
 & \text{if $\epsilon_0 = +1$, $m_{i_0+1}$ is even, $m_{i_0+k}$ is even,}
\end{cases} 
\]
so that
\[
 (p,q) = (r_{\sigma'} + k, s_{\sigma'}).
\]
Moreover, we have the following. 
\begin{itemize}
\item 
If $\kappa_{i_1} > 0$, then we have $(\kappa_{i_1}, +1) \in \XX_{\sigma'}$ but $(\kappa_i, -1) \notin \XX_{\sigma'}$ for all $i_1 < i < i_0 + k$.
Hence we have $(\kappa_{i_1}, +1) \in \CC^+_{\sigma'}(t)$.
\item 
If $\kappa_{i_1} < 0$, then we have $(\kappa_{i_1}, -1) \in \XX_{\sigma'}$ but $(\kappa_i, +1) \notin \XX_{\sigma'}$ for all $i_0 + 1 < i < i_1$.
Hence we have $(\kappa_{i_1}, -1) \in \CC^-_{\sigma'}(t)$.
\item 
If $\kappa_{i_1} = 0$ (so that $k_{\sigma'} = -1$), then we have $(0, +1), (0, -1) \in \XX_{\sigma'}$.
\end{itemize}
Hence by Theorem \ref{t:atobe}, we have $\theta_{p,q}(\sigma') = 0$ as desired.
This completes the proof of Theorem \ref{t:main-lds}\eqref{main-lds-2}.

\section{Proof of Theorem \ref{t:main-temp}}

In this section, we consider the theta lifting from $\U(p,q)$ to $\U(r,s)$ and determine the theta lifts of tempered representations in terms of those of (limits of) discrete series representations.

Let $\pi$ be an irreducible tempered representation of $\U(p,q)$ and write $\pi = I(\xi_1, \dots, \xi_d, \pi_0)$ as in \eqref{eq:temp}.
Assume that $\theta_{r,s}(\pi) \ne 0$.
Then by Corollary \ref{c:jacquet}, we have $d \le \min \{ r, s \}$ and $\theta_{r-d,s-d}(\pi_0) \ne 0$.
Hence by the induction principle \cite[Theorem 4.5.5]{paul1}, $\theta_{r,s}(\pi)$ is a subquotient of
\[
 I(\xi_1 \chi_V^{-1} \chi_W, \dots, \xi_d \chi_V^{-1} \chi_W, \theta_{r-d,s-d}(\pi_0)).
\]
However, it follows Lemma \ref{l:irred-cohom} and Theorem \ref{t:main-lds} that the parabolically induced representation above is irreducible.
Thus we conclude that
\[
 \theta_{r,s}(\pi) = I(\xi_1 \chi_V^{-1} \chi_W, \dots, \xi_d \chi_V^{-1} \chi_W, \theta_{r-d,s-d}(\pi_0)).
\]
This completes the proof of Theorem \ref{t:main-temp}.

\end{document}